\documentclass{article}

\usepackage[utf8]{inputenc}   
\usepackage[T1]{fontenc}      
\usepackage{multicol}
\usepackage{amsmath,amsthm} 
\usepackage{amssymb,mathrsfs} 
\usepackage{amssymb}
\usepackage{amsfonts}
\usepackage[normalem]{ulem}
\usepackage{nicefrac}
\usepackage{geometry}
\usepackage{graphicx}
\usepackage {psfrag}

\usepackage{enumerate}
\usepackage[T1]{fontenc}
\usepackage{url}
\usepackage{graphicx} 
\usepackage{lmodern} 
\usepackage{mathtools}

\RequirePackage{tikz}
\usetikzlibrary{external}

\usepackage{ifthen}
\usepackage{pgfplots}
\pgfplotsset{compat=1.9}
\usepackage{scalefnt}
\usepackage[squaren,Gray]{SIunits}

\usepackage{physics}
\usepackage{dsfont}		
\usepackage{hyperref}

\DeclareMathOperator*{\argmin}{\arg\!\min}

\newcommand{\dt}{\dd{t}}

\newcommand{\paq}{\partial_q}
\newcommand{\pap}{\partial_p}

\newcommand{\papn}{\partial_{p_n}}
\newcommand{\papone}{\partial_{p_1}}
\newcommand{\papN}{\partial_{p_N}}
\newcommand{\parn}{\partial_{r_n}}

\newcommand{\half}{{\nicefrac{1}{2}}}

\newcommand{\Dt}{{\Delta t}}

\newcommand{\calB}{\mathcal{B}}
\newcommand{\calC}{\mathcal{C}}
\newcommand{\calD}{\mathcal{D}}

\newcommand{\calK}{\mathcal{K}}
\newcommand{\calL}{\mathcal{L}}
\newcommand{\calM}{\mathcal{M}}

\newcommand{\calS}{\mathcal{S}}

\newcommand{\calX}{\mathcal{X}}

\newcommand{\bbN}{\mathbb{N}}

\newcommand{\bbR}{\mathbb{R}}
\newcommand{\bbC}{\mathbb{C}}
\newcommand{\bbT}{\mathbb{T}}

\newcommand{\bbone}{\mathds{1}}

\newcommand{\Esp}{\mathbb{E}}
\newcommand{\bbE}{\mathbb{E}}

\newcommand{\rme}{\mathrm{e}}

\newcommand{\rmH}{\mathrm{H}}
\newcommand{\rmL}{\mathrm{L}}
\newcommand{\rmO}{\mathrm{O}}
\newcommand{\rmR}{\mathrm{R}}

\newcommand{\rmVar}{\mathrm{Var}}

\newcommand{\bfA}{\mathbf{A}}

\newcommand{\bfJ}{\mathbf{J}}
\newcommand{\bfK}{\mathbf{K}}

\newcommand{\bfM}{\mathbf{M}}
\newcommand{\bfN}{\mathbf{N}}

\newcommand{\bfR}{\mathbf{R}}
\newcommand{\bfS}{\mathbf{S}}

\newcommand{\bfone}{\mathbf{1}}

\newcommand{\scrR}{\mathscr{R}}

\usepackage{subfigure}
\usepackage{algorithm}
\usepackage{algorithmic}
\usepackage{a4wide}

\newtheorem{theorem}{Theorem}
\newtheorem{lemma}{Lemma}
\newtheorem{assumption}{Assumption}
\newtheorem{prop}{Proposition}
\newtheorem{corollary}{Corollary}
\newtheorem{remark}[theorem]{Remark}

\usepackage{cancel}

\newcommand{\Lnot}{\calL_0}
\newcommand{\Lnotinv}{\calL_0^{-1}}
\newcommand{\Leta}{\calL_\eta}
\newcommand{\Letainv}{{\calL_\eta^{-1}}}

\newcommand{\Lrep}{\widetilde{\mathcal{L}}}
\newcommand{\tcalL}{\widetilde{\calL}}

\newcommand{\LFD}{\calL_\mathrm{FD}}

\newcommand{\Lpert}{\calL_\mathrm{pert}}

\newcommand{\calLinv}{{\calL^{-1}}}

\newcommand{\Phinot}{{\Phi_0}}
\newcommand{\Phieta}{\Phi_\eta}

\newcommand{\Phinoteps}{\Phi_{0,\varepsilon}}

\newcommand{\pinot}{\pi_0}
\newcommand{\pieq}{\pi_{\rm eq}}

\newcommand{\pieta}{\pi_\eta}

\newcommand{\rmLinf}{{\rmL^\infty}}
\newcommand{\rmLinfn}{{\rmL_n^\infty}}
\newcommand{\rmLinfm}{{\rmL_m^\infty}}

\newcommand{\Lpinot}{{\rmL^2(\pinot)}}
\newcommand{\Lpieq}{{\rmL^2(\pieq)}}
\newcommand{\tLpi}{{\rmL_0^2(\pi)}}
\newcommand{\Lpieta}{{\rmL^2(\pieta)}}

\newcommand{\calBLmu}{{\calBLmu}}
\newcommand{\calBLnu}{{\calBLnu}}
\newcommand{\calBLkappa}{{\calBLkappa}}

\newcommand{\lang}{\left\langle}

\newcommand{\rangeq}{\right\rangle_{\rm eq}}
\newcommand{\rangnot}{\right\rangle_0}
\newcommand{\rangeta}{\right\rangle_{\eta}}

\newcommand{\Espnot}{\mathbb{E}_0}

\newcommand{\Espeta}{\mathbb{E}_\eta}

\newcommand{\TL}{T_{\rmL}}
\newcommand{\TR}{T_{\rmR}}

\newcommand{\betainv}{\beta^{-1}}
\newcommand{\betainvinv}{\beta^{-2}}

\newcommand{\Pinot}{\Pi_0}

\newcommand{\Pieta}{\Pi_\eta}

\newcommand{\PiM}{\Pi_M}

\newcommand{\VM}{V_M}
\newcommand{\PhiM}{\Phi_M}

\RequirePackage{xcolor}

\definecolor{line01}{RGB}{242, 201,  49}
\definecolor{line02}{RGB}{ 33, 110, 180}
\definecolor{line03}{RGB}{154, 153,  64}
\definecolor{line04}{RGB}{187,  77, 152}
\definecolor{line05}{RGB}{222, 139,  83}
\definecolor{line06}{RGB}{121, 187, 146}
\definecolor{line07}{RGB}{223, 154, 177}
\definecolor{line08}{RGB}{197, 163, 202}
\definecolor{line09}{RGB}{205, 200,  63}
\definecolor{line10}{RGB}{223, 176,  57}
\definecolor{line11}{RGB}{142, 101,  56}
\definecolor{line12}{RGB}{ 50, 142,  91}
\definecolor{line13}{RGB}{137, 199, 214}
\definecolor{line14}{RGB}{103,  50, 142}

\definecolor{linea}{RGB}{211,  94,  60}
\definecolor{lineb}{RGB}{ 78, 144, 204}

\colorlet{dark_color}{black!67}
\tikzstyle{dark}  =[very thick, color=dark_color]
\tikzstyle{blue}  =[very thick, color=lineb]
\tikzstyle{red}  =[very thick, color=linea]
\tikzstyle{green}  =[very thick, color=line06]

\newcommand{\Ldim}{\calL_{\rm dim}}
\newcommand{\Linter}{\calL_{\rm inter}}
\newcommand{\Lsol}{\calL_{\rm sol}}
\newcommand{\omegahar}{\hat \omega}
\newcommand{\rhar}{\hat r}
\newcommand{\PhinotM}{\Phi_{0,M}}
\newcommand{\tdeco}{t_{\rm deco}}

\newcommand{\tsimu}{T}
\newcommand{\rcut}{r_{\rm cut}}

\usepackage{subfigure}

\usepackage{graphicx}
\usepackage{grffile}

\graphicspath{{Img/}}

\title{A perturbative approach to control variates in molecular dynamics}
\author{Roussel Julien, Stoltz Gabriel \\
{\small Université Paris-Est, CERMICS (ENPC), Inria, F-77455 Marne-la-Vallée, France}}
\date{\today}

\begin{document}

\maketitle

\abstract
{
We propose a general variance reduction strategy to compute averages with diffusion processes. Our approach does not require the knowledge of the measure which is sampled, which may indeed be unknown as for nonequilibrium dynamics in statistical physics. We show by a perturbative argument that a control variate computed for a simplified version of the model can provide an efficient control variate for the actual problem at hand. We illustrate our method with numerical experiments and show how the control variate is built in three practical cases: the computation of the mobility of a particle in a periodic potential; the thermal flux in atom chains, relying on a harmonic approximation; and the mean length of a dimer in a solvent under shear, using a non-solvated dimer as the approximation.
}

\section{Introduction}

Diffusion processes have won an increasing interest in the past years in the statistical physics community, to model physical phenomena and to sample the underlying probability measure characterizing the state of the system~\cite{Balian07}. The average value of a thermodynamic function~$R$ (as the energy, the pressure, a length, a flux, ...) under this probability distribution is given by an integral over the very high-dimensional configurational space. An important motivation for this work is the averaging of mean properties for systems subject to an external driving. In this case the invariant probability measure is often not known explicitly. The goal can be to compute a transport coefficient, a free energy or more generally the response to a non-equilibrium forcing. From a practical point of view, the unknown probability measure is sampled by integrating a stochastic dynamics~\cite{Allen87,Frenkel02,Tuckerman10,Leimkuhler16}
\begin{equation}
\label{eq:process intro}
	\dd X_t = b(X_t) \, \dd t + \sigma(X_t) \, \dd W_t.
\end{equation}
Two prototypical dynamics in molecular simulation are the Langevin and overdamped Langevin dynamics~\cite{Allen87,Leimkuhler16_book}. At equilibrium, the Langevin dynamics evolves positions $q$ and momenta $p$ as
\begin{equation}
  \label{eq:Langevin_introduction}
\left\{
\begin{aligned}
  \dd q_t &= \frac {p_t} m \, \dd t, \\
  \dd p_t &= -\nabla V(q_t) \, \dd t - \frac \gamma m p_t \, \dd t + \sqrt{2 \gamma \betainv} \, \dd W_t,
\end{aligned}
\right.
\end{equation}
where $\gamma>0$ is the friction coefficient, $m>0$ is the mass of a particle and $\beta > 0$ is proportional to the inverse temperature. The potential energy function is denoted by $V$ and $W_t$ is a multi-dimensional standard Brownian motion. In the limit of large frictions $\gamma$, this equation becomes after proper rescaling the overdamped Langevin dynamics~\cite{Freidlin04}:
\begin{equation}
\label{eq:cv overdamped}
\dd q_t = -\nabla V(q_t) \, \dd t + \sqrt{2 \betainv} \, \dd W_t.
\end{equation}
Nonequilibrium versions of the above dynamics are obtained for instance by considering non-gradient forces rather than $-\nabla V$.

For any ergodic dynamics~\eqref{eq:process intro}, macroscopic properties are computed via averages over a trajectory as
\begin{equation*}
  \Esp[R] := \lim_{T \to \infty} \widehat R_T \quad \mbox{a.s.}, \qquad \widehat R_T = \frac 1 T \int_0^T R(X_t) \, \dd t.
\end{equation*}
The statistical error for these estimators is characterized by the asymptotic variance:
\begin{equation}
\label{eq:def variance}
	\sigma_R^2 = \lim_{T \to \infty} T \, \rmVar\left[\widehat R_T \right].
\end{equation}
In many cases of interest ergodic means converge very slowly, requiring the use of variance reduction techniques to speed up the computation. Two types of phenomena can lead to large statistical errors: first, the metastability arising from multimodal potentials, which can greatly increase the correlation of the trajectory in time and lead to large variances; second, a high signal-to-noise ratio, which is typical when averaging small linear responses as for the computation of transport coefficients~\cite{Evans13_book, Tuckerman10}.

When the system is at equilibrium, by which we mean that detailed balance holds, the invariant probability measure is often known and it is possible to use standard variance reduction techniques~\cite{Rubinstein81, Fishman96, Caflisch98, Liu01, Lapeyre03} such as importance sampling~\cite{Chen97, Liu01} or stratification~\cite{Geyer94, Meng96, Kong03, Tan04, Shirts08}. This allows to address both metastability issues and high noise-to-signal ratios.

For non-equilibrium systems, and more generally when the invariant probability measure is not known, reducing the variance is challenging since standard variance reduction methods cannot be used. Note that reducing the metastability would require modifying the dynamics while keeping the invariant probability measure unchanged, or at least knowing how it changes (see for instance~\cite[Section~3.4]{Lelievre16}). When the latter is not known, this task is hard to perform. On the other hand increasing the signal-to-noise ratio is feasible even for non-equilibrium dynamics. This is the goal of the present work, where we rely on control variates. Control variates laying on the concept of "zero-variance" principle have been already used in molecular simulation~\cite{Assaraf99}. This approach was also studied in Bayesian inference simulations~\cite{Henderson97, Mira03, Dellaportas12, Mira13, Oates17}, where configurations are sampled with Markov chains rather than diffusion processes. This type of techniques has however been restricted to cases where the invariant probability measure is known, except for specific settings such as~\cite{Goodman09}, where a coupling strategy is described.

In the present work, which relies on ideas announced in~\cite[Section 3.4.2]{Lelievre16}, control variates are constructed without any knowledge of the expression of the invariant probability measure. We build an unbiased modified observable $R+\xi=R+\calL \Phi$, where $\calL$ is the generator of the dynamics, which is of smaller variance (at least in some asymptotic regime):
\begin{equation*}
	\Esp[R + \xi] = \Esp[R] \quad \mbox{and} \quad \sigma_{R + \xi}^2 < \sigma_R^2.
\end{equation*}
The optimal choice for the control variate $\xi$ is $\xi = \calL \Phi$ where $\Phi$ is the solution of the following Poisson equation:
\begin{equation*}
	-\calL \Phi = R - \Esp[R].
\end{equation*}
The general strategy we consider consists in approximating this partial differential equation (PDE) by a simplified one, with an operator $\Lnot$, for which the solution $\Phinot$ can be analytically computed or numerically approximated with a good precision. Theorems~\ref{th:var order 2 eta exact} and~\ref{th:var order 2 eta} provide an analysis of the asymptotic variance $\sigma_{R + \xi}^2$.

We present numerical results illustrating the general method in three practical cases. In particular we provide in each case a simplified process, associated to a simplified Poisson problem. In these applications we are interested in averaging the linear response of an observable with respect to a non-equilibrium perturbation. This is a challenging class of problems since the average quantity is small and thus the relative statistical error is large. We present the problems we consider by increasing complexity of the setup. We start with the computation of the mobility of a particle in a periodic two-dimensional potential. The control variate can be approached with a very high precision by a numerical method based on a spectral basis, allowing to illustrate Theorem~\ref{th:var order 2 eta}. We next estimate the conductivity of an atom chain~\cite{Bonetto00, Lepri16, Dhar08}. The number of state variables is much larger (up to several hundreds of degrees of freedom in our simulations) but the geometrical setting is one-dimensional. The control variate can be computed analytically when taking a harmonic model as a reference, and thus it does not require any additional numerical procedure. The third application is a dimer in a solvent, whose mean length is estimated under an external shearing force. In the latter case the difficulty comes from the fact that the system is high-dimensional and not as structured as the atom chain.

This article is organized as follows. We present in Section~\ref{s:strategy} the general strategy for building control variates and state a result making precise how control variates behave in a perturbative framework. We then turn to the case of a single particle in a one-dimensional periodic potential under a non-gradient forcing in Section~\ref{s:sinus potential}; the computation of the thermal flux passing through a chain in Section~\ref{s:quasi harmonic}; and the estimation of the mean length of a dimer in a solvent under an external shearing stress in Section~\ref{s:dimer}. Some technical results are gathered in the appendices.

\section{General strategy}
\label{s:strategy}

The definition of the asymptotic variance of time averages along a trajectory requires to introduce more precisely the generator of the process and some associated functional spaces, which is done in Section~\ref{ss:asyvar}. The concept of control variate is then explained Section~\ref{ss:ideal cv}, as well as the so-called "zero-variance principle". We give in Section~\ref{ss:perturbative} our perturbative construction of control variate in an abstract setting and state the main theorem quantifying the variance reduction in a limiting regime. Finally Section~\ref{ss:numerics poisson} provides a generalized version of this theorem in the case when an approximate solver is used.

\subsection{Asymptotic variance}
\label{ss:asyvar}

\newcommand{\Xteta}{X_t^\eta}

The state space $\calX$ is typically the full space $\bbR^d$ or a bounded domain with periodic boundary conditions $\bbT^d$. For some dynamics such as Langevin dynamics, auxiliary variables with values in $\bbR^d$ are added, so that in this case $\calX = \bbR^d \times \bbR^d$ or $\calX = \bbT^d \times \bbR^d$. As suggested in the introduction, we decompose the generator of the process~\eqref{eq:process intro} as a sum $\calL = \Lnot + \Lrep$ of a reference generator $\Lnot$ and a perturbation $\Lrep$. In order to study the asymptotic regime corresponding to small perturbations, we use a parameter $\eta \in \bbR$ to interpolate smoothly between $\Lnot$ and $\calL$, and define
\begin{equation*}
	\Leta = \Lnot + \eta \Lrep.
\end{equation*}
We suppose that these operators $\Leta$ write:
\begin{equation*}
	\Leta = b_\eta\cdot \nabla + \frac 1 2 \sigma_\eta \sigma_\eta^\top : \nabla^2 = \sum_{i=1}^d \left( b_\eta \right)_i \partial_{x_i} + \frac 1 2 \sum_{i,j=1}^d \left(\sigma_\eta \sigma_\eta^\top \right)_{i,j} \partial_{x_i x_j},
\end{equation*}
with $b_\eta$ and $\sigma_\eta$ are smooth. They are then the generators of the following stochastic processes on $\calX$, indexed by $\eta$:
\begin{equation}
\label{eq:process}
	\dd \Xteta = b_\eta(\Xteta) \, \dd t + \sigma_\eta(\Xteta) \, \dd W_t,
\end{equation}
where $W_t$ is a $d$-dimensional standard Brownian motion. Let us assume that $b_\eta$ and $\sigma_\eta$ are such that the following holds.
\begin{assumption}
\label{as:pieta}
The dynamics~\eqref{eq:process} admits a unique invariant probability measure $\pieta$ for any $\eta \in \bbR$. Moreover, trajectorial ergodicity holds: for any observable $R \in \rmL^1(\pieta)$, 
\begin{equation*}
  \Espeta[R] := \int_\calX R(x) \, \dd \pieta(x) = \lim_{T \to \infty} \frac 1 T \int_0^T R(\Xteta) \, \dd t \quad \mathrm{a.s.}
\end{equation*}
\end{assumption}
Sufficient conditions for this to hold are discussed after Assumption~\ref{as:lyapunov in L2}. Let us now make the functional spaces precise. We denote by $(\calK_n)_{n \in \bbN}$ a family of so-called Lyapunov functions with values in $[1, +\infty)$. The associated weighted $\rmL^\infty$ spaces are:
\begin{equation*}
\begin{aligned}
  \forall n \in \bbN,\quad \rmLinfn &= \left\{ \varphi \ \mbox{measurable} \ \left| \ \|\varphi\|_\rmLinfn < \infty \right. \right\} , \qquad \|\varphi\|_{\rmL_n^\infty} = \left\| \frac \varphi {\calK_n} \right\|_\rmLinf.
\end{aligned}
\end{equation*}
We make the following assumption on the Lyapunov functions.
\begin{assumption}
\label{as:lyapunov in L2}
For any $\eta \in \bbR$, the function $\calK_n$ belongs to~$\Lpieta$. In particular,
\begin{equation*}
	\forall n \in \bbN,\, \forall \eta \in \bbR, \quad \rmLinfn \subset \Lpieta.
\end{equation*}
We also assume that for any $n,n' \in \bbN$, there exists $m \in \bbN$ such that $\calK_n \calK_{n'} \in \rmLinfm$.
\end{assumption}
The first part of Assumption~\ref{as:lyapunov in L2} is typically obtained by using a family of Lyapunov functions satisfying conditions of the form
\begin{equation}
\label{eq:lyapunov condition}
	\Leta \calK_n \leqslant -\alpha_{n,\eta} \calK_n + b_{n,\eta},
\end{equation}
for some $\alpha_{n,\eta}>0$ and $b_{n,\eta} \in \bbR$. Indeed, after integration against $\pieta$,
\begin{equation*}
  \begin{aligned}
    0 = \int_\calX \Leta \calK_n \dd \pieta \leqslant -\alpha_{n,\eta} \int_\calX \calK_n \dd \pieta + b_{n,\eta},
\end{aligned}
\end{equation*}
so that
\begin{equation*}
  1 \leqslant \int_\calX \calK_n \dd \pieta \leqslant \frac {b_{n,\eta}}{\alpha_{n,\eta}}.
\end{equation*}
We can then conclude with the second part of Assumption~\ref{as:lyapunov in L2} since, for any $n \geqslant 1$, there exist $C_n > 0$ and $m \geqslant 1$ such that $1 \leqslant \calK_n^2 \leqslant C_n \calK_m$. The condition~\eqref{eq:lyapunov condition} also implies the existence of an invariant probability measure for any $\eta \in \bbR$ when a minorization condition holds~\cite{Hairer11_chap}. A typical choice for the Lyapunov functions are the polynomials $\calK_n(x) = 1+|x|^n$. This choice satisfies the second part of Assumption~\ref{as:lyapunov in L2} when the invariant probability measure has moments of any order. Unless otherwise mentioned, we always consider this choice in the sequel (which is standard for Langevin and overdamped Langevin dynamics, see~\cite{Talay02,Mattingly02,Kopec14,Kopec15}). As for Assumption~\ref{as:pieta}, trajectorial ergodicity holds when the generator $\Leta$ is elliptic or hypoelliptic, and there exists an invariant probability measure with positive density with respect to the Lebesgue measure~\cite{Kliemann87}. The latter condition follows if the measure which appears in the minorization condition has a positive density with respect to the Lebesgue measure.

We denote for any function $\varphi \in \rmL^1(\pieta)$ the projection on the space of mean zero functions by:
\begin{equation*}
  \Pieta \varphi = \varphi - \Espeta[\varphi].
\end{equation*}
For any operator $A \in \calB(E)$ (bounded on the Banach space $E$), the operator norm is defined as
\begin{equation*}
	\| A \|_{\calB(E)} = \sup_{\| \varphi \|_E = 1} \| A \varphi \|_E.
\end{equation*}
Let us now make the following assumption.
\begin{assumption}
\label{as:generator eta}
For any $n \in \bbN$, the $\rmL^2(\pieta)$ norms of the Lyapunov functions are uniformly bounded on compact sets of $\eta$: for any $\eta_*>0$ there exists a constant $C_{n,\eta_*}$ such that
\begin{equation}
\label{eq:uniform lyapunov measure}
	\forall |\eta| \leqslant \eta_*, \quad \| \calK_n \|_\Lpieta \leqslant C_{n,\eta_*}.
\end{equation}
Moreover $\Leta$ is invertible on $\Pieta \rmLinfn$. Finally the inverse generator is bounded uniformly on compact sets of $\eta$: 
\begin{equation}
\label{eq:uniform inverse generator}
	\forall |\eta| \leqslant \eta_*, \quad \left\| -\Letainv \right\|_{\calB \left(\Pieta \rmLinfn \right)} \leqslant C_{n,\eta_*}.
\end{equation}
\end{assumption}
The invertibility of $\Leta$ on $\Pieta \rmLinfn$ is a standard result which follows typically from the Lyapunov conditions~\eqref{eq:lyapunov condition} and a so-called minorization condition~\cite{Hairer11_chap}. It has been proved for a large variety of problems~\cite{Eckmann99,Talay02,Mattingly02,Kopec14,Kopec15}. Conditions~\eqref{eq:uniform lyapunov measure} and~\eqref{eq:uniform inverse generator} are needed to prove Theorems~\ref{th:var order 2 eta exact} and~\ref{th:var order 2 eta} to come. Condition~\ref{eq:uniform lyapunov measure} can be obtained by showing uniform bounds on the coefficients which appear in the Lyapunov conditions, while condition~\eqref{eq:uniform inverse generator} additionally requires some uniformity on the minorization condition.

When Assumptions~1 to 3 hold, the asymptotic variance introduced in~\eqref{eq:def variance} is finite for any $\varphi \in \rmLinfn$ and the following formula holds~\cite{Lelievre16}:
\begin{equation}
  \label{eq:variance eta}
  \sigma_{\varphi, \eta}^2 = 2 \lang \varphi, -\Letainv \Pieta \varphi \rangeta,
\end{equation}
where $\lang \cdot, \cdot \rangeta$ denotes the canonical scalar product on~$\Lpieta$. We refer to Appendix~\ref{ap:variance} for more details on the numerical estimation of the asymptotic variance.
\begin{remark}
We choose to work directly with weighted $\rmL^\infty$ spaces as this is the relevant setting for Theorems~\ref{th:var order 2 eta exact} and~\ref{th:var order 2 eta}. Note that the asymptotic variance of an observable $\varphi \in \Lpieta$ can also be defined using perturbative arguments relatively to an equilibrium reference dynamics~\cite{Dolbeault15, Iacobucci17}. Contrarily to the $\rmLinfn$ framework one would however be restricted in this case to small non-equilibrium perturbations.
\end{remark}

\subsection{Ideal control variate}
\label{ss:ideal cv}

We recall in this section what a control variate is in our context and show how the construction of an optimal control variate can be reformulated as solving a Poisson problem. The functional framework is made precise in a second step using Assumption~\ref{as:generator eta}. We say that a function $\xi$ is a control variate of the observable $R$ for the process $\Xteta$ with generator $\Leta$ if
\begin{equation*}
	\Espeta[R + \xi] = \Espeta[R] \quad \mbox{and} \quad \sigma_{R + \xi, \eta}^2 < \sigma_{R, \eta}^2.
\end{equation*}
The principle of our method, already explained in~\cite{Lelievre16}, is based on the equation which characterizes the invariance of the measure $\pieta$: for any function $\Phi$,
\begin{equation*}
	\Espeta[\Leta \Phi] = 0.
\end{equation*}
This shows that control variates $\xi$ of the form $\xi = \Leta \Phi$ automatically ensure that $\Espeta[R + \xi] = \Espeta[R]$, whatever the choice of $\Phi$. In order for $\xi$ to be a good control variate, the modified observable $R + \xi = R + \Leta \Phi$ should however be of small asymptotic variance. The optimal choice, denoted by $\Phieta$ and referred to as the "zero-variance principle"~\cite{Assaraf99, Mira03}, is to make the modified observable constant. This constant is then necessarily equal to $\Espeta[R]$, and $\Phieta$ is the solution of the Poisson problem:
\begin{equation}
\label{eq:cv poisson problem}
	-\Leta \Phieta = R - \Espeta[R].
\end{equation}
Assuming that $R \in \rmLinfn$ for some $n\in \bbN$, the problem~\eqref{eq:cv poisson problem} admits a unique solution $\Phieta \in \rmLinfn$ when Assumption~\ref{as:generator eta} holds.

In practice two problems arise when trying to solve~\eqref{eq:cv poisson problem}. First, the equation~\eqref{eq:cv poisson problem} is a very high-dimensional PDE for most purposes and the complexity of such problems scales exponentially with the dimension. Second, $\Espeta[R]$ is not known since it is precisely the quantity we are trying to compute. We discuss in the next section how to approximate the solution of~\eqref{eq:cv poisson problem}, at least for small $\eta$.

\subsection{Perturbative control variate}
\label{ss:perturbative}

The key assumption in our approach is to assume that we can compute the solution $\Phi_0$ of the reference Poisson problem corresponding to $\eta=0$:
\begin{equation}
\label{eq:poisson problem simple}
-\Lnot \Phi_0 = R - \Esp_0[R].
\end{equation}
In practice $\Lnot$ is the generator of a simplified dynamics (depending on the problem), and $\Leta$ is the generator of the problem at hand (say for $\eta=1$). Let us emphasize that the dynamics associted with $\mathcal{L}_0$ need not be an equilibrium dynamics (see the example discussed in Section~\ref{s:quasi harmonic}). The small parameter $\eta$ is used to quantify the discrepancy between the optimal function $\Phieta$ and its approximation $\Phinot$ in a perturbative framework. We refer to Section~\ref{ss:numerics poisson} for a discussion on the numerical resolution of~\eqref{eq:poisson problem simple}, and to the end of this section for further comments on the decomposition.

We define the so-called core space $\calS$ as the set of all $\calC^\infty$ functions which grow at infinity at most like $K_n$ for some $n$, and whose derivatives also grow at most like $K_n$ for some $n$. Such a space was considered in~\cite{Talay02} for instance. More precisely,
\begin{equation}
\label{eq:core}
\begin{aligned}
	\calS &= \big\{ \varphi \in \calC^\infty(\calX) \big|\ \forall k \in \bbN,\, \exists n \in \bbN, \quad\partial_k \varphi \in \rmLinfn \big\}.
\end{aligned}
\end{equation}
The space $\calS$ is dense in $\Lpieta$ under Assumption~\ref{as:lyapunov in L2}, since $\calC^\infty$ functions with compact support are included in $\calS$. We need an additional assumption in our analysis to ensure that $\Phinot \in \calS$.
\begin{assumption}
\label{as:generator simple}
The space $\calS$ is stable by the generator $\Lnot$ and $\Lnot$ is invertible on the space $\Pinot \calS$ composed of functions with average $0$ with respect to the invariant probability measure~$\pinot$. This means that, for any $\varphi \in \Pinot \calS$, there exists a unique solution $\psi \in \Pinot \calS$ to the Poisson equation
\begin{equation*}
  -\Lnot \psi = \varphi.
\end{equation*}
\end{assumption}
Assumption~\ref{as:generator simple} can be proved to hold for Langevin and overdamped Langevin dynamics at equilibrium under certain assumptions on the potential $V$, see~\cite{Talay02, Kopec14, Kopec15}. The generator of the perturbation should also satisfy the following condition.
\begin{assumption}
\label{as:Lrep}
The generator $\tcalL$ of the perturbation is such that $\calS$ is stable by $\tcalL$ and $\Lrep^* \bfone \in \Lpinot$.
\end{assumption}
Here and in the following we denote by $B^*$ the adjoint of a closed operator $B$ on the functional space $\Lpinot$. Assumption~\ref{as:Lrep} is easy to check, since $\Lrep$ is typically a differential operator with coefficients in $\calS$.

Let us define the following modified observable involving $\Phinot \in \Pinot \calS$, defined in~\eqref{eq:poisson problem simple}:
\begin{equation*}
\begin{aligned}
  \phi_\eta &:= R + \Leta \Phi_0. \\
\end{aligned}
\end{equation*}
The following theorem makes precise the main properties of this modified observable.
\begin{theorem}
\label{th:var order 2 eta exact}
Fix $R \in \calS$. Under Assumptions~1 to 5, $\phi_\eta \in \calS$ is well defined for any $\eta \in \bbR$, and $\Espeta[\phi_\eta] = \Espeta[R]$. Moreover, for any $\eta_*>0$, there exists $C_{R, \eta_*} > 0$ such that, for any $| \eta | \leqslant \eta_*$, the asymptotic variance satisfies
\begin{equation}
\label{eq:var order 2 eta exact}
	\sigma_{\phi_\eta, \eta}^2 = 2 \eta^2 \lang A R, -\Lnotinv A R \rangnot + \eta^3 E_{R, \eta},
\end{equation}
with $A= - \Lrep \Lnotinv \Pi_0$ and $| E_{R, \eta} | \leqslant C_{R, \eta_*}$.
\end{theorem}

The scalar products involved in the previous theorem are well defined since $\calS$ is stable by $A$, and $\calS \subset \Lpieta$ for any $\eta \in \bbR$. Equation~\eqref{eq:var order 2 eta exact} shows that the standard error $\sqrt{\frac{\sigma_{\phi_\eta, \eta}^2} T}$ committed on the empirical estimator $\widehat \varphi_T$ of $\Espeta[R]$ after a time $T$ is of leading order $\eta/\sqrt{T}$.

The scaling $\eta^2$ of the asymptotic variance formally comes from the fact that the modified observable writes $\phi_\eta = \Espeta[R] + \rmO(\eta)$. Indeed,  
\begin{equation}
  \label{eq:Leta Linv}
  \Leta \Lnotinv \Pinot = \Pieta (\Lnot + \eta \Lrep) \Lnotinv \Pinot = \Pieta + \eta \Pieta \Lrep \Lnotinv \Pinot = \Pieta (1-\eta A),
\end{equation}
so that the modified observable can be rewritten as:
\begin{equation}
\begin{aligned}
\label{eq:modified obs}
\phi_{\eta} &= R + \Leta \Phi_0 = R - \Leta \Lnotinv \Pinot R = \Esp_\eta[R] + \eta \Pieta A R.
\end{aligned}
\end{equation}
In particular,
\begin{equation}
\label{eq:projected modified obs}
\Pi_\eta \phi_\eta = \eta \Pi_\eta A R.
\end{equation}
The remainder of the proof consists in carefully estimating remainders in some truncated series expansion of $-\mathcal{L}_\eta^{-1}\Pi_\eta$; see Appendix~\ref{ap:proof Th2}.

\begin{remark}
  The formula~\eqref{eq:var order 2 eta exact} can in fact be replaced by an expansion in powers of $\eta$ with a truncation at an arbitrarily high order and a remainder controlled uniformly in $|\eta| \leqslant \eta_*$. This can be proved by an immediate generalization of the proof we provide in Appendix~\ref{ap:proof Th2}.
\end{remark}

In the following applicative sections we cannot always prove that Assumptions~\ref{as:generator eta} and~\ref{as:generator simple} hold true, but the scaling of the variance predicted by Theorem~\ref{th:var order 2 eta exact} is nevertheless numerically observed to hold. More importantly, let us emphasize that the modeling process is crucial to write the generator as the sum of a reference generator~$\mathcal{L}_0$ for which we can solve Poisson equations, and an additional term. Such a decomposition can always be performed, but the quality of the resulting control variate will strongly depend on the choice of the decomposition. Typically, the control variate method works well when the additional term is a perturbation of the reference generator (which corresponds to the regime where Theorem~\ref{th:var order 2 eta exact} can be applied). In practice, this can be checked a posteriori, once the simulation has been carried out, by comparing the asymptotic variance of the two estimators, one with and one without the control variate. Of course, one should keep the one with the smallest variance.

\subsection{Numerical resolution of the reference Poisson problem}
\label{ss:numerics poisson}

We discuss in this section a strategy to compute the solution to~\eqref{eq:poisson problem simple} when this equation cannot be analytically solved. We rely for this on a Galerkin strategy, and look for an approximation of the solution $\Phinot$ to the Poisson problem~\eqref{eq:poisson problem simple} in a subspace $V_M \subset \Lpinot$ of finite dimension~$M$. For simplicity we suppose that $V_M \subset \Pinot \Lpinot$ (which corresponds to a conformal approximation). This implies in particular that $\PhiM \in \Pinot \Lpinot$ has mean zero with respect to $\pinot$. We also assume that $\VM \subset \rmH^2(\pinot)$ to avoid regularity issues. The optimal choice for the approximation $\PhinotM$ is to minimize the variance of the modified observable $R+\Lnot \PhinotM$ for the reference dynamics:
\begin{equation}
  \label{eq:minimization_variance_Galerkin}
  \min_{\varphi \in \VM} \sigma_{R+\Lnot \varphi,0}^2 = \sigma_{R,0}^2 + \min_{\varphi \in \VM} \frac 1 2 \lang (\Lnot + \Lnot^*) \varphi, 2 \Lnotinv \Pinot R + \varphi \rangnot.
\end{equation}
The latter equality follows from Lemma~\ref{lemma:variance LU} in Appendix~\ref{ap:proof Th2}, in the particular case when $\eta=0$. In the case when $\Lnot$ is not the generator of a stochastic differential equation the quantity $\sigma_{R+\Lnot \varphi,0}^2$ cannot be interpreted as a variance but the analysis we provide here remains valid. The necessary optimality condition for a minimizer $\Phi_{0,M}$ of~\eqref{eq:minimization_variance_Galerkin} is given by the following Euler-Lagrange equation:
\begin{equation*}
	\forall \psi \in \VM, \quad \lang (\Lnot + \Lnot^*) (\PhinotM + \Lnotinv \Pinot R), \psi \rangnot = 0.
\end{equation*}
Introducing the orthogonal projector $\PiM$ on $\VM$ (with respect to the scalar product on $\Lpinot$), the latter equation can be rewritten as
\begin{equation}
\label{eq:euler lagrange numerics}
\PiM (\Lnot + \Lnot^*) (\PhinotM + \Lnotinv \Pinot R) = 0.
\end{equation}
In practice we distinguish two cases.
\begin{itemize}
\item[(i)] For reversible dynamics such as the overdamped Langevin dynamics~\eqref{eq:cv overdamped} $\Lnot = \Lnot^*$, so the equation reduces to
\begin{equation}
\label{eq:minimizer reversible}
	-\PiM \Lnot \PhinotM = \PiM R.
\end{equation}
\item[(ii)] For Langevin dynamics at equilibrium (see~\eqref{eq:Langevin_introduction}), we consider a tensorized basis involving Hermite elements as in Section~\ref{s:sinus potential} or in~\cite{Roussel17}. The symmetric part of the generator:
\begin{equation*}
	\frac 1 2 (\Lnot + \Lnot^*) = -\gamma \betainv \nabla_p^* \nabla_p,
\end{equation*}
diagonalizes the Hermite polynomials so we have the commutation rule $\PiM (\Lnot + \Lnot^*) = (\Lnot + \Lnot^*) \PiM$. Moreover the kernel of $\Lnot + \Lnot^*$ is composed of functions depending only on the position variables. The condition~\eqref{eq:euler lagrange numerics} then implies that there exists $g= g(q)$ in $\Lpinot$ such that
\begin{equation}
\label{eq:minimizer langevin}
	\PhinotM = - \PiM \Lnotinv \Pinot R + g.
\end{equation}
\end{itemize}

The solution to~\eqref{eq:minimizer reversible} coincides with the result provided by the Galerkin method on the approximation space $\VM$. For~\eqref{eq:minimizer langevin} the optimal solution, for $g=0$, is given by the Galerkin method apart from a consistency error (see~\cite{Roussel17} for a detailed analysis). This justifies the use of Galerkin methods to determine a good approximation $\PhinotM$ of $\Phinot$ in the general case. 

For the Langevin equation the operator $\calL$ is not coercive on $\tLpi$ so the associated rigidity matrix is not automatically invertible. The existence of a unique solution $\PhinotM \in V_M$ converging to $\Phinot = -\calLinv \Pinot R$ when $M \to +\infty$, as well as error estimates and a discussion on non-conformal approximations, can be found in~\cite{Roussel17}.

When a Galerkin method (or any other approximation method) is used, the error committed on $\Phinot$ induces an error on the modified observable $\phi_\eta$. The modified asymptotic variance is then the sum of terms coming from~\eqref{eq:var order 2 eta exact} (depending on $\eta$) and terms coming from the approximation error (of order $\varepsilon$) committed on $\Phinot$, as made precise in the following result.

\begin{theorem}
\label{th:var order 2 eta}
Fix $R \in \calS$ and assume that $\Phi_0$ is approximated by $\Phinoteps = \Phi_0 + \varepsilon f$ with $f \in \calS$ and $\varepsilon \geqslant 0$. Denote by $\phi_{\eta,\varepsilon} = R + \Leta \Phinoteps$ the modified observable. Under Assumptions~1 to~5, for any $\eta_*,\varepsilon_* > 0$, there exists $E_{R,\eta_*,\varepsilon_*} > 0$ such that, for any $|\eta | \leqslant \eta_*$ and $| \varepsilon | \leqslant \varepsilon_*$,
\begin{equation}
\begin{aligned}
\label{eq:variance with cv}
\sigma_{\phi_{\eta, \varepsilon}, \eta}^2 &= 2 \varepsilon^2 \lang -\Lnot f, f \rangnot 
- 2 \varepsilon \eta \lang (\Lnot + \Lnot^*) f, \Lnotinv \Pi_0 A R \rangnot \\
&\quad +2 \eta^2 \lang A R, -\Lnotinv \Pi_0 A R \rangnot
+ (\eta^3 + \varepsilon^3) C_{R, \eta, \varepsilon},
\end{aligned}
\end{equation}
with $| C_{R, \eta, \varepsilon} | \leqslant E_{R,\eta_*,\varepsilon_*}$.
\end{theorem}
The proof of this result can be read in Appendix~\ref{ap:proof Th2}. It shows that the variance is globally of order 2 with respect to both $\eta$ and $\varepsilon$. This suggests to take $\eta$ and $\varepsilon$ of the same order.
\begin{remark}
The dependence of $\widetilde C_{R,\eta_*,\varepsilon_*}$ with respect to $R$ can be made more explicit (see for instance the discussion in~\cite{Leimkuhler16}).
\end{remark}
The error committed on $\Phinot$ can also arise from additional approximations on the right hand side of the Poisson problem, in situations when the observable $R_\eta = R_0 + \eta \widetilde R$ depends on $\eta$. A result similar to Theorem~\ref{th:var order 2 eta} can be obtained upon assuming that $\widetilde R \in \calS$, where $\Phinot$ is the solution to~\eqref{eq:poisson problem simple} with $R$ replaced by~$R_0$.

\begin{remark}
\label{rmk:variance irreversible}
Note that in the expression \eqref{eq:variance with cv} the error term $f$ only appears through $(\Lnot + \Lnot^*) f$. This term may vanish even if $f$ is not identically zero. For example, for a Langevin process at equilibrium, $(\Lnot+\Lnot^*) f$ vanishes when $f$ is a function depending only on the positions.
\end{remark}

\section{One-dimensional Langevin dynamics}
\label{s:sinus potential}

We construct in this section a control variate for a one-dimensional system by solving a simplified Poisson equation using a spectral Galerkin method. The simplification consists in neglecting a non-equilibrium perturbation, the small parameter $\eta$ being the amplitude of this perturbation. We first present in Section~\ref{ss:sinus full dynamics} the model and define the quantity of interest, namely the mobility. We next construct in Section~\ref{ss:sinus simplified} the approximate control variate and conclude in Section~\ref{ss:sinus numerics} with some numerical results.

\subsection{Full dynamics}
\label{ss:sinus full dynamics}
We consider the following Langevin process on the state space $\calX = 2 \pi \bbT \times \bbR$:
\begin{equation}
\label{eq:dynamics sin}
\left\{
\begin{aligned}
	\dd q_t &= \frac {p_t} m \, \dd t, \\
	\dd p_t &= (-v'(q_t)  + \eta ) \, \dd t - \frac \gamma m p_t \, \dd t + \sqrt{2 \gamma \betainv} \, \dd W_t,
\end{aligned}
\right.
\end{equation}
where $\gamma, m, \beta > 0$ and $v$ is a smooth $2 \pi$-periodic potential. The particle experiences a constant external driving of amplitude $\eta \in \bbR$. This force is not the gradient of a periodic function, so the system is out of equilibrium and the invariant measure is not known. We are interested in the average velocity $R(q,p)=\frac p m$ induced by the non-gradient force $\eta$, which can also be seen as a mass flux. The linear response of the average velocity with respect to the external driving is characterized by the mobility of the particle~\cite{Rodenhausen89}:
\begin{equation*}
	D = \lim_{\eta \to 0} \frac{\Espeta[R]} \eta.
\end{equation*}
The generator of~\eqref{eq:dynamics sin} is the sum of the generator associated with the Langevin dynamics at equilibrium and of a non-equilibrium perturbation:
\begin{equation*}
\begin{aligned}
	\Leta = \Lnot + \eta \Lrep,
\end{aligned}
\end{equation*}
where
\begin{equation*}
	\Lnot = -v'(q) \pap + \frac p m \paq - \frac \gamma m p \pap + \gamma \betainv \pap^2, \qquad \Lrep = \pap.
\end{equation*}
In this setting the Lyapunov functions are defined for all $n\in \bbN$ as:
\begin{equation*}
	\forall (q,p) \in \calX, \quad \calK_n(q,p) = 1 + |p|^n,
\end{equation*}
and Assumptions~\ref{as:pieta}, \ref{as:lyapunov in L2}, \ref{as:generator eta} and \ref{as:generator simple} correspond to standard results for Langevin dynamics~\cite{Talay02, ReyBellet06_chap1, Kopec15, Leimkuhler16, Lelievre16}. Assumption~\ref{as:Lrep} trivially holds: the core space $\calS$ is stable by $\Lrep = \pap$ by definition, while $\Lrep^* \bfone = (-\pap + \frac \beta m p) \bfone = \frac \beta m p \in \Lpinot$.

\subsection{Simplified dynamics and control variate}
\label{ss:sinus simplified}

Solving the Poisson problem $-\Leta \Phi_\eta = R-\Pi_\eta[R]$ associated with the nonequilibrium dynamics is not practical because $\Espeta[R]$ is not know. For this simple one-dimensional example it would still be technically doable. Since our purpose is however to illustrate both Theorems~\ref{th:var order 2 eta exact} and~\ref{th:var order 2 eta}, we do not follow this path. We therefore consider the control variate associated to a reference Poisson problem, namely
\begin{equation}
  \label{eq:simplified langevin 1D}
  -\Lnot \Phi_0 = R.
\end{equation}
Note that the average drift vanishes at equilibrium: $\Esp_0[R] = 0$. Equation~\eqref{eq:simplified langevin 1D} cannot be solved analytically, but it is possible to approach its solution by a Galerkin method as explained in Section~\ref{ss:numerics poisson}. The modified observable is
\begin{equation*}
  \phi_{\eta,M} = R + \Leta \PhinotM,
\end{equation*}
with the notation of Theorem~\ref{th:var order 2 eta} (the error committed when estimating $\Phinot$ is indexed by $M$ instead of $\varepsilon$). In practice we construct a basis $(e_m)_m$ of $V_M$ and write $\PhinotM = \sum_{m=1}^M a_m e_m$ where the coefficients $(a_m)_m \in \bbR^M$ are the solution of a linear system obtained from~\eqref{eq:euler lagrange numerics} (see~\cite{Roussel17}). The modified observable is then $\phi_{\eta,M} = R + \sum_{m=1}^M a_m \Leta e_m$ where the functions $\Leta e_m$ are explicit for appropriate choices of basis functions $(e_m)_m$; see Section~\ref{ss:sinus numerics}.

The mobility is estimated with an ergodic average of $R$ along a trajectory during a time~$T$, for a forcing $\eta$, by$\widehat D_{\eta, T} = \frac 1 \eta \widehat R_T$. This estimator has an expectation of order 1 and a large variance when $\eta$ is small and $T$ is large~\cite{Lelievre10, Lelievre16}:
\begin{equation*}
	\bbE[\widehat D_{\eta, T}] = D + \rmO \left( \eta + \frac 1 T\right), \qquad \rmVar[\widehat D_{\eta, T}] \sim \frac {\sigma_R^2} {\eta^2 T},
\end{equation*}
so that the relative statistical error scales as
\begin{equation*}
	\frac {\sqrt{\rmVar[\widehat D_{\eta, T}]}} {\bbE[\widehat D_{\eta, T}]} \sim \frac {\sigma_R} {D \eta \sqrt T}.
\end{equation*}
In order for this quantity to be small, the simulation time should be taken of order $T \sim \frac 1 {\eta^2}$, which is very large since $\eta$ is small.

When the Poisson problem is exactly solved, the modified observable is
\begin{equation*}
	\phi_{\eta} = R + \Leta \Phi_0 = \eta \tcalL \Phi_0,
\end{equation*}
which is proportional to $\eta$, so that the associated asymptotic variance scales as $\sigma_{\phi_\eta}^2 \sim \eta^2$. The relative statistical error for the mobility estimator $\widetilde D_{\eta, T} = \frac 1 \eta \widehat \phi_{\eta,T}$ is then bounded with respect to $\eta$:
\begin{equation*}
	\frac {\rmVar[\widetilde D_{\eta, T}]} {\bbE[\widetilde D_{\eta, T}]} \sim \frac 1 {D \sqrt T},
\end{equation*}
and the simulation time can be fixed independently of the value of $\eta$. Now, if an error of order $\varepsilon_M$ is committed on $\Phinot$, the asymptotic variance of $\phi_{\eta, M}$ scales as $\eta^2 + \varepsilon_M^2$ so the relative statistical error on $\widetilde D_{\eta, T}$ is of order 
\[
\frac {|\eta| + \varepsilon_M} {\sqrt{\tsimu} {|\eta|}} = \frac 1 {\sqrt{\tsimu}} \left( 1 + \frac {\varepsilon_M} {|\eta|} \right).
\]
This implies that the simulation time $T$ can be taken of order $\displaystyle{1+\left( \frac{\varepsilon_M}{\eta} \right)^2}$ instead of $\eta^{-2}$.

\subsection{Numerical results}
\label{ss:sinus numerics}

In order to simplify the numerical resolution of the Galerkin problem (see Section~\ref{ss:numerics poisson}) we consider the simple potential:
\begin{equation*}
  \forall q \in 2 \pi \bbT, \quad v(q) = 1 - \cos(q).
\end{equation*}
We construct $\VM$ using a tensorized basis made of weighted Fourier modes in position and Hermite modes in momenta. The particular weights of the Fourier modes are chosen so that the basis is orthogonal for the $\Lpinot$ scalar product. Obtaining error estimates on $\Phinot-\PhinotM$ requires some work, see~\cite[Section~4]{Roussel17} for a detailed analysis and a precise expression of the basis. In the following we take either $M=15 \times 10$ basis elements ($15$ Fourier modes and $10$ Hermite modes), $M=7 \times 5$ or $5 \times 3$ basis elements. Estimating $\Phinot$ allows to construct a control variate, and also to compute directly the mobility since the Green--Kubo formula~\cite{Kubo91} states that
\begin{equation}
  \label{eq:mobility algebra}
  D = \beta \lang R, -\Lnotinv R \rangnot = \beta \lang R, \Phi_0 \rangnot.
\end{equation}

In order to compute the mobility using Monte-Carlo simulations, we fix $\eta >0$ small and rely on the estimators $\widehat D_{\eta, T}$ or $\widetilde D_{\eta, T}$ defined in Section~\ref{ss:sinus simplified}. We are interested in the reduction of the asymptotic variance provided by our control variate, \textit{i.e.} comparing $\rmVar[\widetilde D_{\eta, T}]$ and $\rmVar[\widehat D_{\eta, T}]$. The Langevin dynamics is integrated over a time $\tsimu=2 \times 10^4$ with time steps $\Delta t = 0.02$ for an external forcing $\eta$ ranging from $0.01$ to $2.56$. The numerical integration is done with a Geometric Langevin Algorithm~\cite{BouRabee10}. This scheme ensures that the invariant probability measure is correct up to terms of order $\rmO(\Dt^2)$ at equilibrium. Moreover the transport coefficients estimated by linear response are also correct up to terms of order $\Dt^2$ (see~\cite{Leimkuhler16}). The scheme writes:
\begin{equation}
  \label{eq:gla}
  \left\{ \begin{aligned}
    p^{k+\half} &= p^k + \left(- v'(q^k)+\eta\right) \frac {\Delta t} 2, \\
    q^{k+1} &= q^k + \frac {p^{k+\half}} m \Delta t, \\
    \widetilde p^{k+1} &= p^{k+\half} + \left(- v'(q^{k+1})+\eta\right) \frac {\Delta t} 2, \\
    p^{k+1} &= \alpha_{\Dt} \widetilde p^{k+1} + \sqrt{m \betainv (1 - \alpha_{\Dt}^2)} \ G^k,
	\end{aligned} \right.
\end{equation}
where the superscript $k$ is the iteration index, $\alpha_{\Dt} = \exp(-\frac \gamma m \Delta t)$ and the $(G^k)_{k \in \bbN}$ are independent and identically distributed (i.i.d.) standard one-dimensional Gaussian random variables. The results which are reported are obtained for $m = \gamma = \beta = 1$.

\paragraph{Linear response.}

The results presented in Figure~\ref{fig:langevin mobility} (Left) show that the average velocity scales linearly with respect to the forcing for $\eta$ small, as predicted by linear response theory. The slope, which is the mobility $D$, matches the one computed using~\eqref{eq:mobility algebra} for $M$ large. An effective mobility is obtained by dividing the average velocity by the forcing, see Figure~\ref{fig:langevin mobility} (Right). We are interested in its limiting value for a small forcing. On the one hand the result is biased if $\eta$ is too large, but on the other hand the statistical error scales like $\frac 1 \eta$. For the standard observable we see that the optimal trade-off value of $\eta$ is around $0.2$ for the chosen simulation time $\tsimu$.  When using a control variate the variance is much smaller in the small forcing regime, so $\eta$ can be taken very small to reduce the bias while keeping the statistical error under control. We discuss next the estimation of the error bars plotted on Figure~\ref{fig:langevin mobility}.

\begin{figure}
\caption{Linear response for the standard MC simulation (black squares) compared to the version with control variate (blue, red) and to the asymptotic response $D \eta \approx 0.48 \eta$ (black line).}
\begin{center}
\includegraphics[scale=.58]{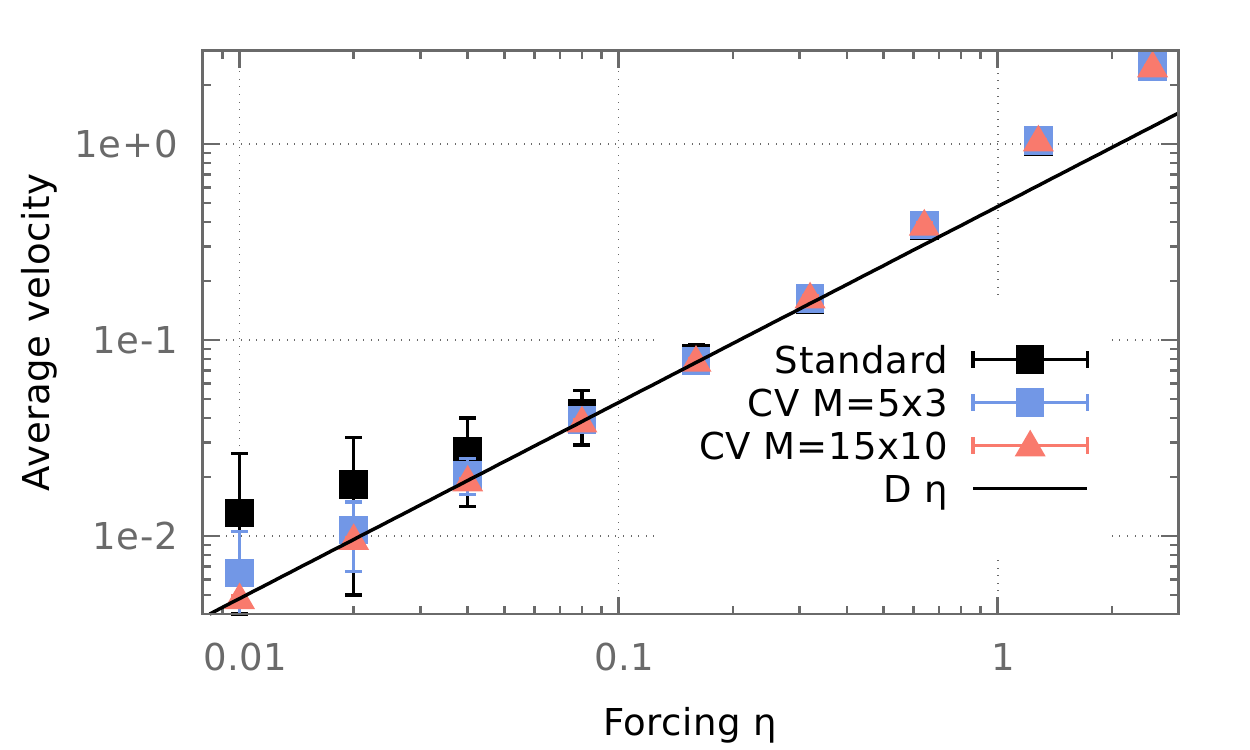}
\includegraphics[scale=.58]{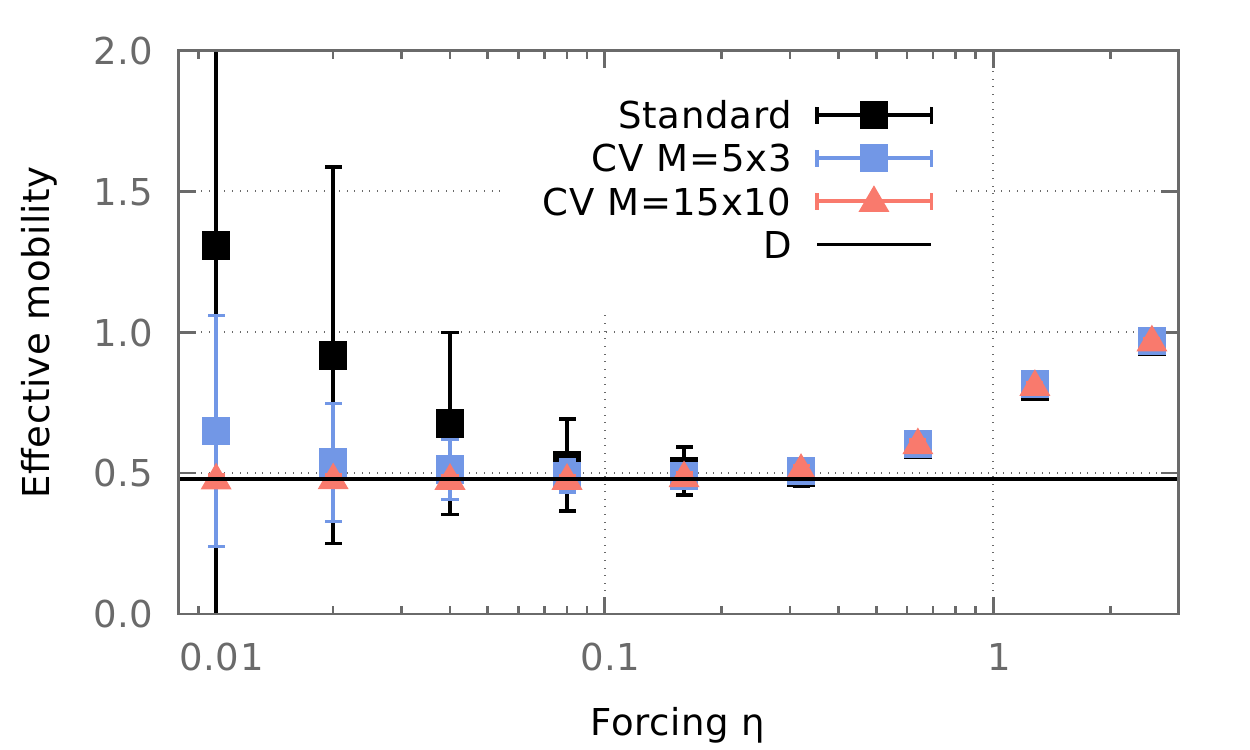}
\end{center}
\label{fig:langevin mobility}
\end{figure}

\paragraph{Correlation profiles.}

The asymptotic variance of time averages for an observable $\varphi \in \calS$ writes, using the Green--Kubo formula~\cite{Kubo91},
\begin{equation*}
	\sigma_\varphi^2 = 2 \int_0^\infty C_\varphi(t) \, \dd t, \qquad C_\varphi(t) = \Esp[(\varphi(X_0)-\Esp[\varphi]) (\varphi(X_t)-\Esp[\varphi])],
\end{equation*}
where the autocorrelation function $C_\varphi$ involves an expectation over all initial condition $X_0 = (q_0,p_0) \in \calX$ distributed according to the invariant probability measure~$\pi_0$ at equilibrium, and for all realizations of the dynamics with generator~$\mathcal{L}_0$. The integrability of $C_\varphi(t)$ can be guaranteed when the semi-group $\mathrm{e}^{t \mathcal{L}_0}$ decays sufficiently fast in $\rmLinfn$~\cite{Lelievre16}. The function $C_\varphi$ is characterized by three major features explaining the value of the asymptotic variance; see Figure~\ref{fig:correlation cartoon} for an illustration. 
\begin{enumerate}
\item[(i)] The first one is the amplitude of the signal $\| \varphi - \Esp[\varphi] \|_{\rmL^2(\pi_0)}^2 = C_\varphi(0)$ corresponding to the value of the autocorrelation at $t=0$.
\item[(ii)] The second one is the characteristic decay time $\tau$ of the autocorrelation, which can be related to the decay of its exponential envelope.
\item[(iii)] The last one is the presence of anticorrelations, which arise only for non-reversible dynamics such as Langevin dynamics.
\end{enumerate}
A proper estimation of the asymptotic variance requires to compute autocorrelation profiles on a sufficiently long time interval $[0,\tdeco]$ (here $\tdeco=6$). One can check a posteriori that this time is sufficient by looking at the convergence of the cumulated autocorrelation $t \mapsto \int_0^t C_\varphi$ toward its limit $\frac {\sigma_\varphi^2} 2 = \int_0^\infty C_\varphi$ (see Appendix~\ref{ap:variance} for more details on the variance estimators, and the computation of error bars for these quantities).

\begin{figure}
\caption{Illustrative autocorrelation profile. The dashed line is the exponential envelope of the correlation function.}
\begin{center}
\includegraphics[scale=.58]{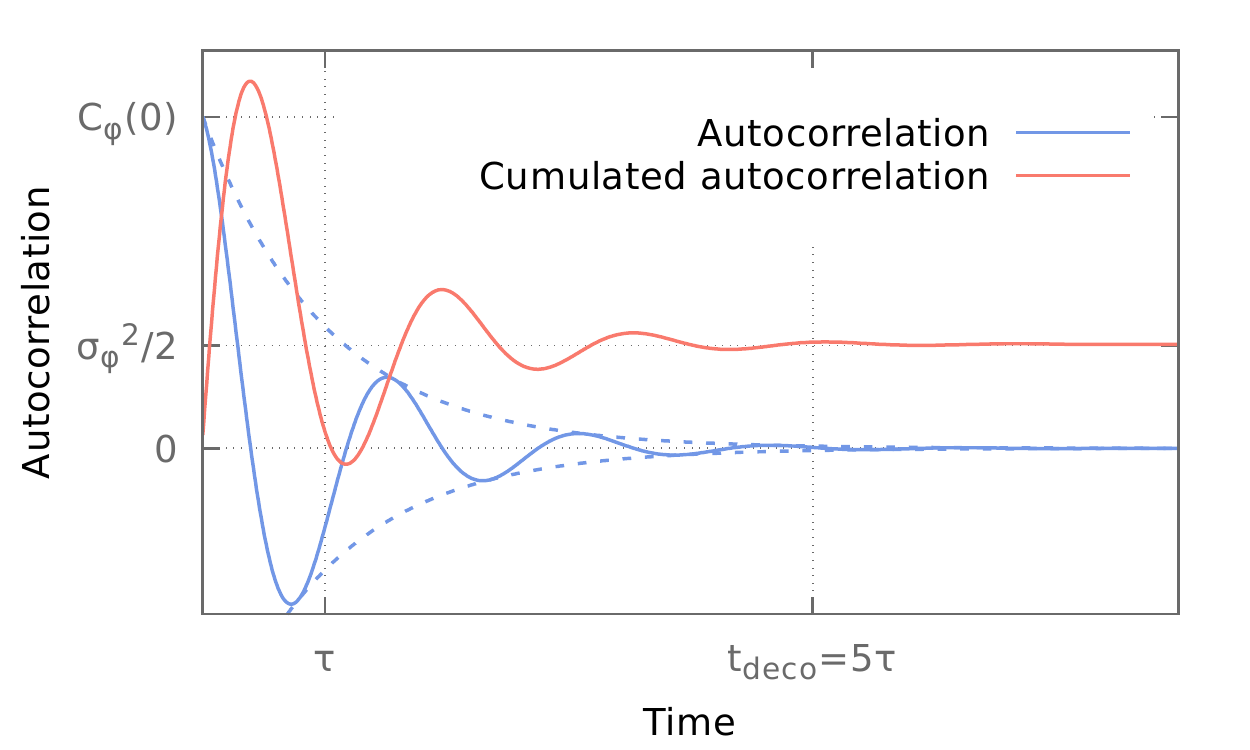}
\end{center}
\label{fig:correlation cartoon}
\end{figure}

Figure~\ref{fig:langevin autocorrelation} compares the autocorrelation profile of the velocity with the ones for the modified observables, for two different Galerkin basis sizes~$M$ and two different forcing amplitudes~$\eta$. For a small forcing $\eta=0.08$ the two modified observables have an amplitude and a decorrelation time which are both much smaller than for the standard velocity. Note that the two modified observables do not exhibit any anti-correlation, contrarily to the velocity observable. The cumulated plots show that the control variates drastically reduce the asymptotic variance in this case, especially for the one based on a more accurate Galerkin approximation. For a larger value $\eta=1.28$ the modified observables have a significantly larger amplitude (\textit{i.e.} $C_\varphi(0)$ is larger), especially in the case of a low accuracy~$M$. However the decorrelation times are small and there is anti-correlation, resulting in a reasonable variance reduction in both cases.

\begin{figure}
\caption{Left: Autocorrelation profile of the velocity compared to the one of the modified observables for two different accuracies, either for a small forcing $\eta=0.08$ (top) or a larger one $\eta=1.28$ (bottom). Right: Corresponding cumulated autocorrelations. The limit value is half the asymptotic variance of the observable.}
\begin{center}
\includegraphics[scale=.58]{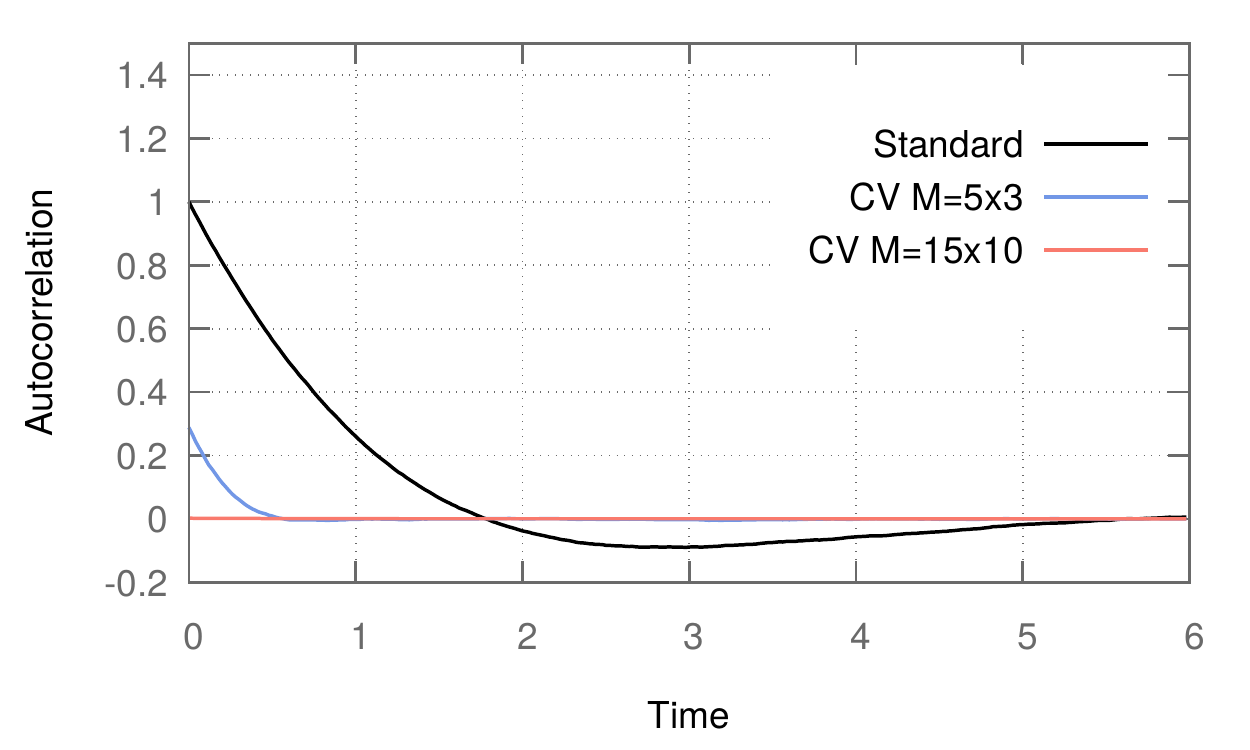}
\includegraphics[scale=.58]{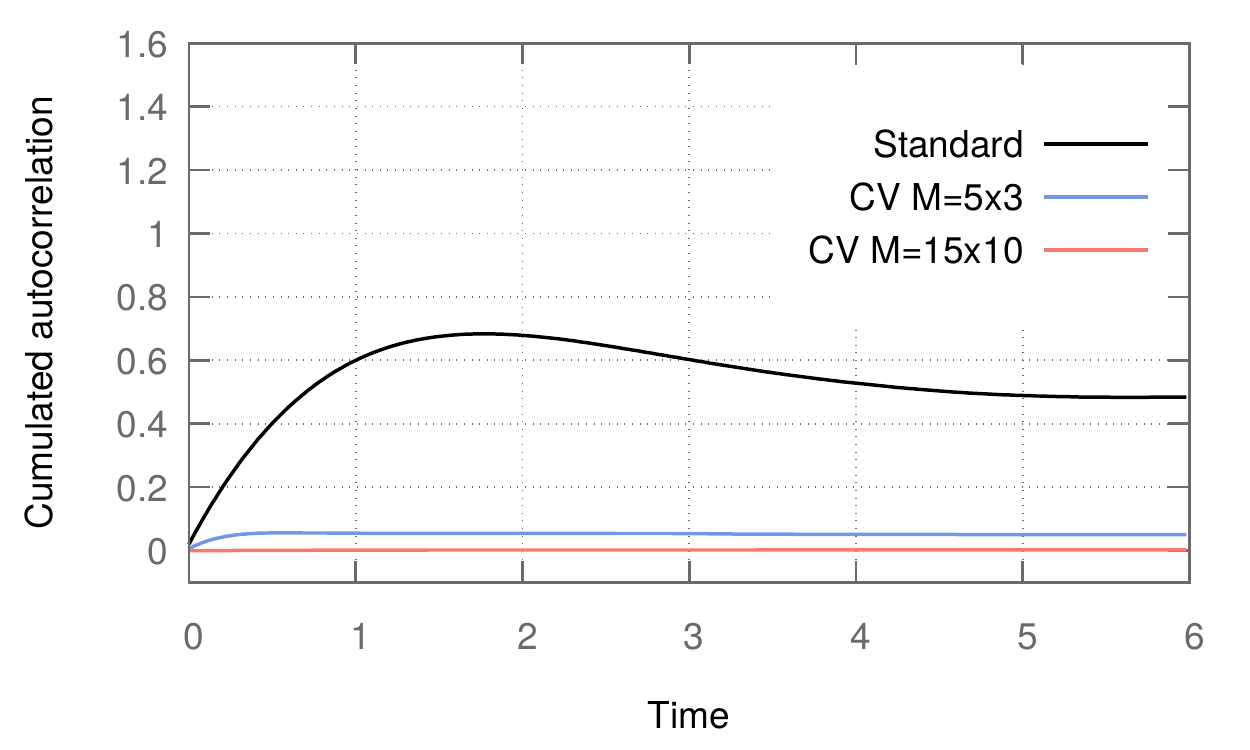}\\
\includegraphics[scale=.58]{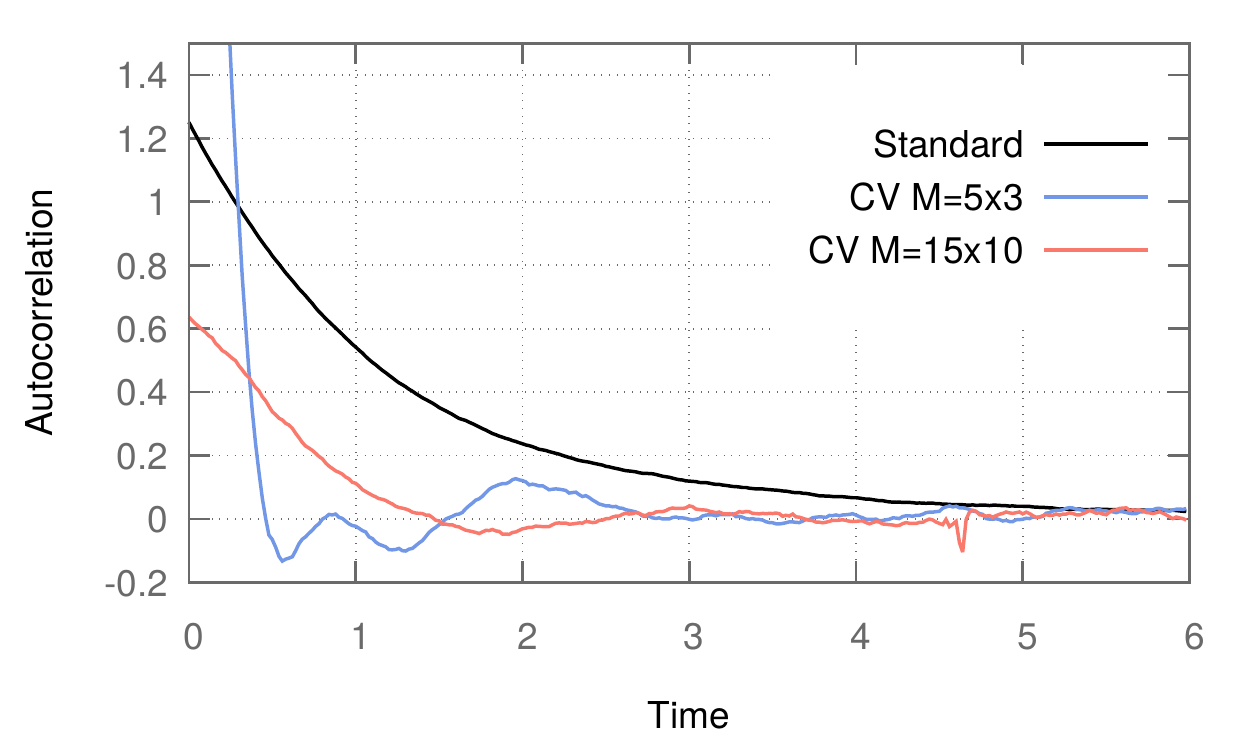}
\includegraphics[scale=.58]{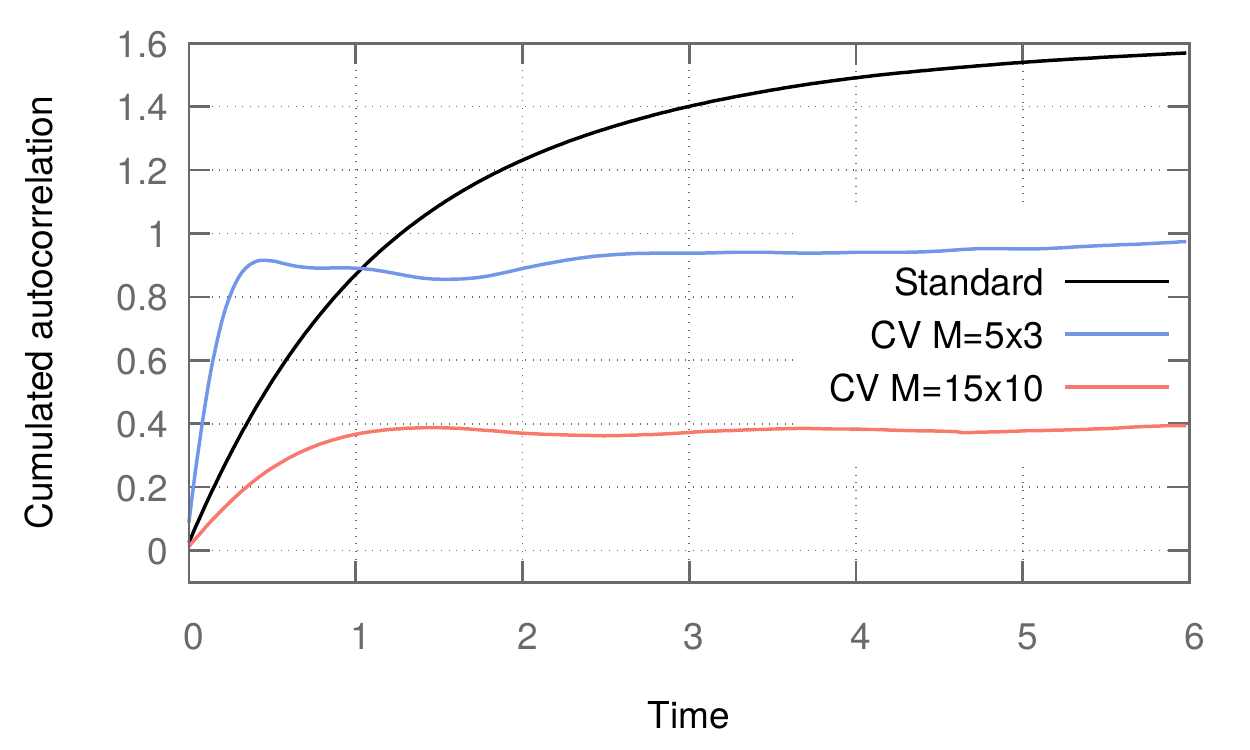}
\end{center}
\label{fig:langevin autocorrelation}
\end{figure}

\paragraph{Asymptotic variances.}

Let us now compute the asymptotic variance for a whole range of Galerkin accuracies~$M$ and forcing amplitudes $\eta$. The results presented in Figure~\ref{fig:langevin variance} confirm that for a very accurate Galerkin resolution the variance of the modified observable scales as $\eta^2$ with the prefactor $\alpha = \lang \Pi_0 A R, -\Lnotinv \Pi_0 A R \rangnot$ predicted theoretically in Theorem~\ref{th:var order 2 eta exact}. This prefactor has been computed independently by solving~\eqref{eq:simplified langevin 1D} using a Galerkin method and plugging this approximation in~\eqref{eq:mobility algebra}. When the Galerkin discretization is not sufficiently accurate, the variance reaches a plateau in the region of small forcings as predicted by Theorem~\ref{th:var order 2 eta}.

\begin{figure}
\caption{Asymptotic variance of the velocity (black squares) compared to its counterpart when using a control variate (blue, grey, red) and to the reduced variance (black line) predicted theoretically ($\alpha \approx 0.53$ computed with a Galerkin discretization).}
\begin{center}
\includegraphics[scale=.58]{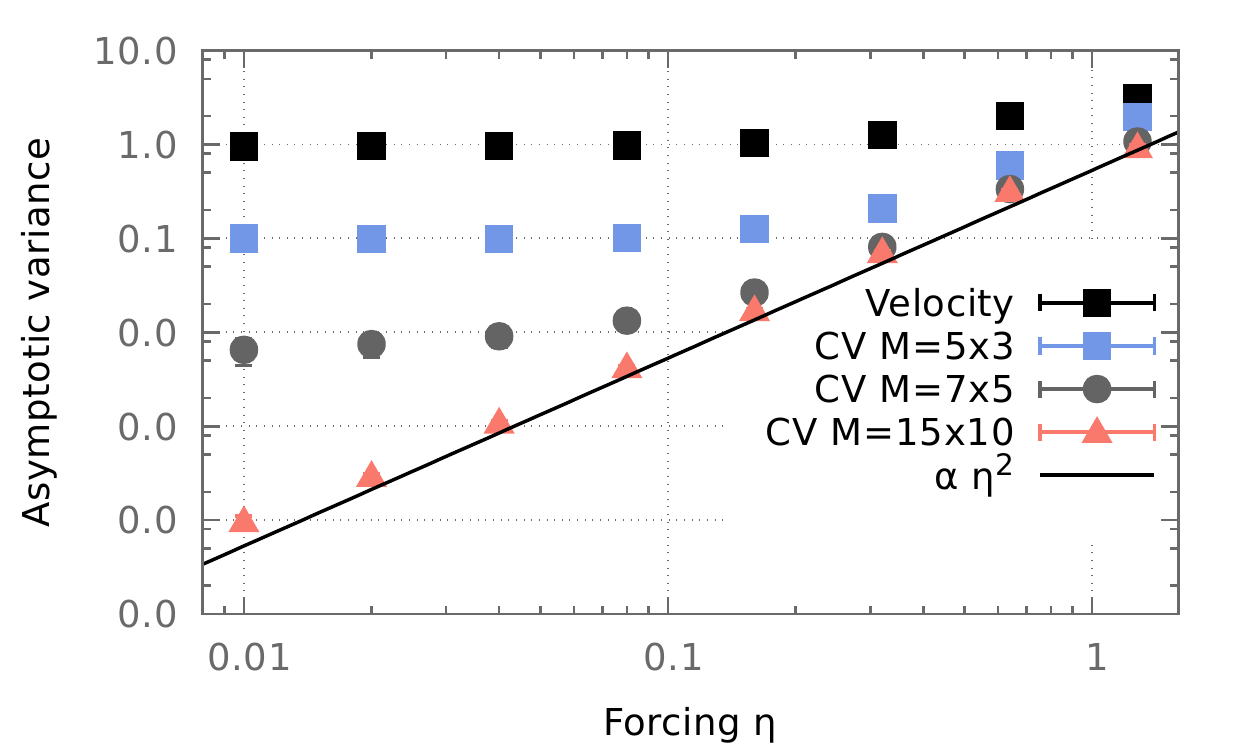}
\end{center}
\label{fig:langevin variance}
\end{figure}

\section{Thermal transport in atom chains}
\label{s:quasi harmonic}

Thermal transport in one-dimensional systems has been the topic of many investigations, both from theoretical and numerical points of view~\cite{Bonetto00, Lepri16, Dhar08}. Determining which microscopic ingredients influence the scaling of the conductivity with respect to the length of the chain is still an active line of research. Studying numerically this scaling requires to simulate chains of thousands of particles. In these systems the temperature gradient and the thermal flux are both very small, which induces large statistical errors when estimating the conductivity. Introducing variance reduction techniques not requiring the knowledge of the invariant probability measure could alleviate (at least partly) these difficulties.

\subsection{Full dynamics}

\subsubsection{Equations of motion}

We consider a chain composed of $N$ particles interacting through a nearest-neighbor potential~$v$ (see Figure~\ref{fig:thermal chain}). The evolution is dictated by a Hamiltonian dynamics and a thermalization mechanism at the boundaries, where the first and the last particles are submitted to Ornstein-Uhlenbeck processes, at temperatures $\TL$ and $\TR$ respectively. The unknowns are the momenta~$p=(p_n)_{1 \leqslant n \leqslant N}$ of the particles and the interparticle distances $r=(r_n)_{1 \leqslant n \leqslant N-1}$. In these variables, the dynamics reads:
\begin{equation}
\label{eq:dynamics chain}
\left\{ \begin{aligned}
	\dd r_n &= \frac 1 m (p_{n+1} - p_n) \, \dd t, \\
	\dd p_1 &= v'(r_1) \, \dd t - \frac \gamma m p_1 \dt + \sqrt{2 \gamma \TL} \dd W_t^L, \\
	\dd p_n &= (v'(r_n) - v'(r_{n-1})) \, \dd t, \\
	\dd p_N &= -v'(r_{N-1}) \, \dd t - \frac \gamma m p_N \dt + \sqrt{2 \gamma \TR} \dd W_t^R,
\end{aligned} \right.
\end{equation}
where $m>0$ is the mass of a particle, $\gamma>0$ is the friction coefficient, $\TL \geqslant \TR$ and $W_t^L$, $W_t^R$ are two independent standard one-dimensional Brownian motions. Notice that the ends of the chain are free. The Hamiltonian of the system is the sum of the potential and kinetic energies:
\begin{equation*}
	H(r,p) = V(r) + \sum_{n=1}^N \frac {p_n^2}{2 m}, \quad V(r) = \sum_{n=1}^{N-1} v(r_n).
\end{equation*}
The infinitesimal generator of the dynamics~\eqref{eq:dynamics chain} reads
\begin{equation*}
\begin{aligned}
	\calL &= \frac 1 m \sum_{n=1}^{N-1} (p_{n+1} - p_n) \parn + \sum_{n=1}^N (v'(r_n) - v'(r_{n-1})) \papn \\
	&\quad - \frac \gamma m p_1 \papone + \gamma T_L \papone^2 - \frac \gamma m p_N \papN + \gamma T_R \papN^2,
\end{aligned}
\end{equation*}
using the convention $v'(r_0) = v'(r_N) = 0$.
\begin{figure}
\begin{center}
\tikzsetnextfilename{scales}
\begin{tikzpicture}[scale=.4]
\input{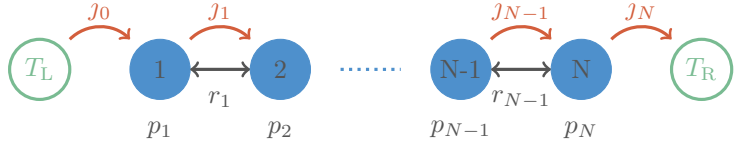}
\end{tikzpicture}
\caption{Heat transport in a one-dimensional chain.}
\label{fig:thermal chain}
\end{center}
\end{figure}

\subsubsection{Properties of the dynamics}
\label{sss:chain properties}

\newcommand{\rmLinftheta}{\rmL_\theta^\infty}
\newcommand{\rmLinfthetanot}{\rmL_{\theta,0}^\infty}

Let us recall some properties of the dynamics~\eqref{eq:dynamics chain} which hold under the following assumption.
\begin{assumption}
\label{as:potential chain}
The interaction potential $v$ is $\calC^\infty$ and there exist $k \geqslant 2$ and $a > 0$ such that
\begin{equation*}
	\forall r_1 \in \bbR, \quad \lim_{\tau \to +\infty} \tau^{-k} v(\tau r_1) = a |r_1|^k, \quad \lim_{\tau \to +\infty} \tau^{1-k} v'(\tau r_1) = k a |r_1|^{k-1} \mathrm{sign}(r_1).
\end{equation*}
Moreover the interaction potential is not degenerate: for any $r_1 \in \bbR$ there exists $m=m(r_1) \geqslant 2$ such that $\partial_m v(r_1) \neq 0$.
\end{assumption}
These conditions hold for the potentials we use in the numerical simulations reported in Section~\ref{ss:chain numerics}. When Assumption~\ref{as:potential chain} holds, the dynamics admits a unique invariant probability measure~$\pi$~(see \cite{Carmona07}). This invariant probability measure is explicit when the chain is at equilibrium ($\TL = \TR = \betainv$), in which case it has the tensorized form
\begin{equation}
\label{eq:measure chain}
	\pieq(\dd r \, \dd p) = Z_\beta^{-1} \exp(- \beta \left( \frac {|p|^2} {2 m} + V(r) \right)) \, \dd r \, \dd p,
\end{equation}
where $Z_\beta^{-1}$ is a normalization constant. Let us emphasize that the reference system considered later on in Section~\ref{ss:chain simplified} is not at equilibrium. Following the framework considered in~\cite{Carmona07} (which is itself based on~\cite{ReyBellet06_chap2}), we consider in this section the Lyapunov functions $\calK_\theta = \rme^{\theta H}$. There exist $\theta_* > 0$ such that $\calK_\theta \subset \rmL^2(\pi)$ for any $\theta \in [0,\theta_*)$. The functional spaces we use are also indexed by the continuous parameter $\theta \in [0,\theta_*)$:
\begin{equation*}
  \rmL_\theta^\infty = \left\{ \varphi \ {\rm measurable} \ \big| \ \| \varphi \, \rme^{-\theta H} \|_{\rmL^\infty} < + \infty \right\},
\end{equation*}
and the space $\calS$ is defined similarly to~\eqref{eq:core}. For $\theta \in [0,\theta_*)$, we also define the vector space $\rmLinfthetanot$ of functions of $\rmLinftheta$ with mean zero with respect to $\pi$. One can prove the exponential decay of the semi-group on the associated functional space $\rmLinfthetanot$ (see~\cite{Carmona07}): for any $\theta \in [0,\theta_*)$, there exist $C, \lambda > 0$ such that, for any $\varphi \in \rmLinfthetanot$,
\begin{equation*}
	\forall t \geqslant 0, \quad \| \rme^{t \calL} \varphi \|_{\rmL_\theta^\infty} \leqslant C \rme^{-\lambda t} \| \varphi \|_{\rmL_\theta^\infty}.
\end{equation*}
This implies that $\calL$ is invertible on $\rmLinfthetanot$, and that its inverse is bounded.

\paragraph{Validity of Assumptions 1 and 2.}
Assumption~\ref{as:pieta} holds true since there exist a unique invariant probability measure with positive density with respect to the Lebesgue measure, and the generator of the dynamics is hypoelliptic~\cite{Carmona07,Kliemann87,ReyBellet06_chap2}. The first part of Assumption~\ref{as:lyapunov in L2} is also satisfied for $\theta \in [0,\theta_*)$. Note that at equilibrium ($\TL=\TR=\betainv$) the invariant probability measure is explicit and $\theta^* = \beta/2$. The product of two Lyapunov functions $\calK_\theta$ and $\calK_{\theta'}$ is in a Lyapunov space only if $\theta + \theta' < \theta_*$, so the second part of Assumption~\ref{as:lyapunov in L2} is not satisfied.

\subsubsection{Heat flux and conductivity}

When studying heat transport in atom chains the typical quantity of interest is the thermal flux through the chain:
\begin{equation}
\label{eq:heat flux}
	\forall n \in [1, N-1], \quad j_n(r,p) = -\frac {p_n + p_{n+1}} 2  v'(r_n),
\end{equation}
see for example the review~\cite{Lepri03} on thermal transport in low-dimensional lattices for further background material. We also make use of the two boundary elementary fluxes:
\begin{equation}
\label{eq:heat flux boundaries}
	j_0(r,p) = \frac \gamma m \left(\TL - \frac {p_1^2} m\right), \qquad j_N(r,p) = \frac \gamma m \left( \frac {p_N^2} m - \TR \right).
\end{equation}
The definition of the elementary fluxes $j_n$ is motivated by the local energy balance, centered on particle~$n$:
\begin{equation}
\label{eq:energy balance}
	\forall n \in [1, N], \quad \calL \varepsilon_n = j_{n-1} - j_n, \qquad \varepsilon_n(r,p) = \frac {v(r_{n-1})} 2 + \frac {p_n^2}{2 m} + \frac {v(r_n)} 2.
\end{equation} 
The quantity $\calL \varepsilon_n$ is of mean zero since it is in the image of the generator. Therefore the elementary fluxes $j_n$ all have the same stationary values:
\begin{equation}
\label{eq:equality fluxes}
	\Esp_\pi[j_0] = \Esp_\pi[j_1] = \cdots = \Esp_\pi[j_N].
\end{equation}
Any linear combinations of such fluxes, namely
\begin{equation}
\label{eq:linear combination flux}
	J_\lambda = \sum_{n=0}^N \lambda_n j_n, \qquad \sum_{n=0}^N \lambda_n = 1,
\end{equation}
has the same stationary value. The most common choice is the spatial mean
\begin{equation}
\label{eq:heat flux sum}
	\widetilde R = \frac 1 {N-1} \sum_{n=1}^{N-1} j_n.
\end{equation}
We call the latter observable "standard heat flux" in the sequel. Notice that it does not depend on the boundary fluxes $j_0$ and $j_N$. The linear response of $\widetilde R$ (or of any flux $J_\lambda$) with respect to the temperature gradient $\frac {\TL-\TR}{N-1}$ defines the effective conductivity:
\begin{equation}
\label{eq:def conductivity}
	\kappa = \frac{N-1}{\TL-\TR}\Esp_\pi[\widetilde R],
\end{equation}
which is here the transport coefficient of interest. There exist infinitely many observables having the same expectation as $\widetilde R$, see \eqref{eq:linear combination flux} for example. Let us first discuss the choice of observable for the heat flux, before trying to reduce its variance by constructing a control variate.

\paragraph{Asymptotic variance of the heat fluxes at equilibrium.}

The chain is supposed to be at equilibrium in all this paragraph ($\TL=\TR=\betainv$). The conclusions remain unchanged for nonequilibrium systems when $\TL-\TR$ is small since the results are only perturbed to first order with respect to this quantity. In the remainder of Section~\ref{s:quasi harmonic}, an index 'eq' refers to the equilibrium dynamics and to the equilibrium invariant probability measure~$\pieq$. In this setting the asymptotic variance $\sigma_{\widetilde R, {\rm eq}}^2$ for the standard heat flux $\widetilde R$ is not smaller than the one associated with any elementary flux $(j_n)_{1\leqslant n \leqslant N-1}$. These two variances are in fact equal, as made precise in the following proposition (similar in spirit to Remark~\ref{rmk:variance irreversible}).

\begin{prop}
\label{prop:variance invariance}
Consider an observable $\varphi \in \calS$ and a function $U \in \calS$ which does not depend on $p_1$ and $p_N$. Then adding $\calL U$ to the observable does not modify the variance:
\begin{equation}
  \label{eq:equality_variances}
  \sigma_{\varphi + \calL U, {\rm eq}}^2 = \sigma_{\varphi, {\rm eq}}^2.
\end{equation}
\end{prop}
\begin{proof}
At equilibrium, the invariant probability measure $\pieq$ is explicit. The symmetric part of the generator can then be computed and it corresponds to the fluctuation-dissipation part of the process:
\begin{equation}
  \label{eq:def_LFD}
  \frac 1 2 (\calL + \calL^*) = \LFD := - \frac{\gamma}{\beta} \left(  \papone^* \papone + \papN^* \papN \right),
\end{equation}
where adjoints are considered on $\Lpieq$. When $U$ does not depend on $p_1$ nor $p_N$, it holds $\LFD U = 0$. The claimed result then follows from Lemma~\ref{lemma:variance LU}.
\end{proof}

The equality~\eqref{eq:equality_variances} is perturbed by terms of order $\TL-\TR$ for out of equilibrium dynamics according to linear response theory. Upon taking for $U$ a linear combination of the energies $(\varepsilon_n)_{2 \leqslant n \leqslant N-1}$, we directly obtain, thanks to~\eqref{eq:energy balance}, that all the fluxes of the form~\eqref{eq:linear combination flux} which do not depend on the boundary fluxes ($\lambda_0=\lambda_N=0$) share the same asymptotic variance at equilibrium; in particular
\begin{equation}
  \label{eq:flux same asymptotic variance}
  \forall 1 \leqslant n \leqslant N-1, \quad \sigma_{j_n, {\rm eq}}^2 = \sigma_{\widetilde R, {\rm eq}}^2.
\end{equation}

\begin{remark}
\label{rmk:flux eq}
Linear response theory indicates that the previous asymptotic variances are related to the conductivity through the Green--Kubo formula~\cite{Kubo91}:
\begin{equation}
  \label{eq:variance and conductivity}
  \sigma_{\widetilde R, {\rm eq}}^2 = \frac {2 \kappa}{\beta^2 (N-1)}.
\end{equation}
We show in Appendix~\ref{ap:proof variance end} that, in this equilibrium situation, the variance of the two boundary fluxes $j_0$ and $j_N$ is also related to the conductivity as:
\begin{equation}
  \label{eq:variance boundary}
  \sigma_{j_0, {\rm eq}}^2 = \frac \gamma {m \beta^2} - \frac {2 \kappa}{\beta^2 (N-1)}, \qquad
  \sigma_{j_N, {\rm eq}}^2 = \frac \gamma {m \beta^2} - \frac {2 \kappa}{\beta^2 (N-1)}.
\end{equation}
Apart from special cases such as integrable systems (as the harmonic system considered in Appendix~\ref{ap:harmonic}), we generically observe numerically that $\frac {\kappa(N)} {N-1} \xrightarrow[N \to \infty]{} 0$. The boundary fluxes therefore have (asymptotically in~$N$) a larger variance than bulk fluxes. Since the variances are perturbed to first order in $\TL-\TR$ in nonequilibrium situations, the same conclusion holds for temperature differences which are not too large.
\end{remark}

In the next section we construct a modified observable by adding a control variate to the reference observable $R = \frac 1 2 (j_0 + j_N)$. This particular heat flux does not depend on the potential energy function~$v$, which simplifies the computation of the control variate (see Appendix~\ref{ap:harmonic}). In Appendix~\ref{ap:justification flux} we prove using Proposition~\ref{prop:variance invariance} that, at equilibrium, the asymptotic variance of the resulting modified observable does not depend on the choice of the reference observable~$R$.

\subsection{Simplified dynamics and control variate}
\label{ss:chain simplified}

We split the interaction potential into a harmonic part with parameters $\omegahar > 0$ and $\rhar \in \bbR$, and an anharmonic part~$w$:
\begin{equation}
  \label{eq:potential decomposition}
  v(r_1) = v_0(r_1) + w(r_1), \qquad v_0(r_1) = \frac 1 2 m \omegahar^2 (r_1-\rhar)^2.
\end{equation}
The potential energy is then decomposed as
\begin{equation*}
  V(r) = V_0(r) + W(r), \qquad V_0(r) = \sum_{n=1}^{N-1} v_0(r_n), \qquad W(r) = \sum_{n=1}^{N-1} w(r_n).
\end{equation*}
Following the general strategy outlined in Section~\ref{s:strategy} we decompose the generator as
\begin{equation*}
	\calL = \Lnot + \tcalL,
\end{equation*}
where $\Lnot$ is the generator of the harmonic chain corresponding to the harmonic interaction potential~$v_0$ and $\tcalL$ is the generator of the anharmonic perturbation:
\begin{equation*}
	\tcalL = \sum_{n=1}^N \left( w'(r_n) - w'(r_{n-1}) \right) \papn,
\end{equation*}
with the same convention $w'(r_0) = w'(r_N) = 0$ as for the potential $v$. We simplify the Poisson problem for the optimal control variate
\begin{equation*}
	-\calL \Phi = R - \Esp[R],
\end{equation*}
into the harmonic Poisson problem
\begin{equation}
\label{eq:simplified poisson chain}
	-\Lnot \Phinot = R - \Espnot[R].
\end{equation}
Note that the observable $R$ does not depend on the potential, contrarily to other heat fluxes such as $\widetilde R$ in~\eqref{eq:heat flux sum}, so that the right hand side of the Poisson equation needs not be changed when looking for an approximate control variate. Equation~\eqref{eq:simplified poisson chain} can be solved analytically for $\Phinot$. In Appendix~\ref{ap:harmonic} we show that
\begin{equation}
\begin{aligned}
\label{eq:cv chain}
	\Espnot[R] &= \frac {\nu^2} {1 + \nu^2} \frac {\gamma (\TL-\TR)} {2 m}, \\
	\Phinot(r,p) &= \frac m {2 \gamma \left( 1 + \nu^2 \right)} \left[-\omegahar^2 \sum_{n=1}^{N-1} (r_n-\rhar) (p_n + p_{n+1}) + \frac \gamma {2m^2} \left( p_N^2 - p_1^2 \right) \right] + C,
\end{aligned}
\end{equation}
where $\nu = \frac {m \omegahar} \gamma$ and $C \in \bbR$ is such that $\Espnot[\Phinot] = 0$. The modified observable is therefore
\begin{equation}
\label{eq:modified obs chain}
\begin{aligned}
	&\quad (R+\calL \Phinot)(r,p) = \Espnot[R] + \tcalL \Phinot(r,p) \\
		&= \frac 1 {2 ( 1 + \nu^2 )} \left[ \nu \omegahar (\TL-\TR) - \nu \omegahar \sum_{n=1}^{N-1} (r_n-\rhar) \left( w'(r_{n+1}) - w'(r_{n-1}) \right)
		 - \left( \frac {p_1} m w'(r_1) + \frac {p_N} m w'(r_{N-1}) \right) \right] \\
		&= \frac 1 {2(1+\nu^2)} \left[ \nu \omegahar (\TL-\TR) - \sum_{n=1}^{N-1} \left( \widetilde v_{n+1}(r,p) - \widetilde v_{n-1}(r,p) \right) w'(r_n) \right],
\end{aligned}
\end{equation}
where
\begin{equation*}
\widetilde v_n(r,p) = \left\{
\begin{array}{cl}
  -\frac {p_1} m \quad &\mbox {if} \quad n=0, \\
  -\nu \omegahar (r_n-\rhar) \quad &\mbox {if} \quad 1 \leqslant n \leqslant N-1, \\
  \frac {p_N} m \quad &\mbox {if} \quad  n=N.
\end{array}
\right.
\end{equation*}
Notice that, by construction, this observable is constant when the chain is harmonic (\emph{i.e.} $w=0$).

\paragraph{Harmonic fitting.}

For a given pair potential $v=v(r)$, there is some freedom in the decomposition~\eqref{eq:potential decomposition}, namely the choice of the parameters $\omegahar$ and $\rhar$. The optimal choice would be such that the variance of the modified observable~\eqref{eq:modified obs chain} is minimal, but this condition is not practical. A possible (and simpler) heuristic is to choose these coefficients in order to minimize the $\rmL^2(\pi_{\rm eq})$ norm of the anharmonic force $-\nabla W$ at equilibrium, namely when $\TL=\TR=\betainv$. In view of the tensorized form~\eqref{eq:measure chain} of the invariant probability measure at equilibrium,
\begin{equation*}
  \| \nabla W(r) \|_{\rm eq}^2 = (N-1) z_\beta^{-1} \int_\bbR \left( v'(r_1) - m \omegahar^2 (r_1-\rhar) \right)^2\, \rme^{- \beta v(r_1)} \dd r_1,
\end{equation*}
where $z_\beta = \int_\bbR \rme^{- \beta v(r_1)} \dd r_1$. Therefore the minimization problem defining $\rhar$ and $\hat \Omega = m \omegahar^2$ writes
\begin{equation}
\label{eq:minimization}
	\argmin_{\omegahar, \rhar} \int_\bbR \left( v'(r_1) - \hat \Omega (r_1-\rhar) \right)^2\, \rme^{- \beta v(r_1)} \dd r_1.
\end{equation}
There exists a minimizer $(\rhar, \hat \Omega)$ since the function to be minimized is continuous and coercive; uniqueness is proved in Appendix~\ref{ap:proof minimizer fpu}. Define the moments of the marginal measure for inter-particle distances as:
\begin{equation*}
	\calM_k = \int_\bbR r_1^k \rme^{- \beta v(r_1)} \dd r_1.
\end{equation*}
The Euler-Lagrange equation associated with~\eqref{eq:minimization} provides the expression of the minimizer (see Appendix~\ref{ap:proof minimizer fpu}):
\begin{equation}
  \label{eq:rhar omegahar values}
  \rhar = \frac{\calM_1}{\calM_0}, \qquad \hat \Omega = m \omegahar^2 = \betainv \frac{\calM_0^2}{\calM_0 \calM_2 - \calM_1^2},
\end{equation}
where, by the Cauchy-Schwarz inequality, $\calM_0 \calM_2 - \calM_1^2 > 0$ for any continuous potential $v$ which is not constant.

\paragraph{Validity of Assumptions 3 to 5.}

The standard way to verify Assumption~\ref{as:generator eta} would be to show that the  coefficients in the Lyapunov condition exhibited in~\cite{Carmona07} depend continuously on perturbations of the potential $v$. This is not straightforward, especially if the exponent~$k$ introduced in Assumption~\ref{as:potential chain} is discontinuous with the perturbation amplitude at $\eta=0$ (which corresponds to the harmonic chain). For example, for the FPU potential~\eqref{eq:fpu}, $k=2$ for the harmonic chain ($\eta=0$) and $k=4$ for $\eta >0$.  We show that Assumption~\ref{as:generator simple} holds under Assumption~\ref{as:potential chain} in Appendix~\ref{ap:ass 4}. Moreover it is clear that $\calS$ is stable by $\Lrep$. A simple computation shows that $\pinot$ is a Gaussian probability measure, which implies that $\Lrep^* \bfone \in \Lpinot$. Therefore Assumption~\ref{as:Lrep} holds as well.

\subsection{Numerical results}
\label{ss:chain numerics}

We consider a Fermi-Pasta-Ulam (FPU) potential
\begin{equation}
  \label{eq:fpu}
  v(r_1) = \frac a 2 r_1^2 + \frac b 3 r_1^3 + \frac c 4 r_1^4,
\end{equation}
where $c=\frac {b^2}{3 a}$ is such that $v''(r_1) = a + 2 b r_1 + 3 c r_1^2 = \frac 1 a (a + br_1)^2$ is positive except at a single point where it vanishes. This choice makes the potential both asymmetric and convex. Symmetric potentials indeed exhibit special behaviors~\cite{Spohn14}, whereas we want to be as general as possible. On the other hand, non-convex potentials are not typical in the literature on the computation of thermal conduction in one-dimensional chains. In the following we fix $a=1$ and vary $b$ only. The parameters $\rhar$ and $\omegahar$ are given by~\eqref{eq:rhar omegahar values}, where the moments $\calM_k$ are computed using one-dimensional numerical quadratures.

The system is simulated for a time $\tsimu=10^8$ with time steps $\Delta t = 10^{-2}$, and with $m=\gamma=1$. The atom chain is simulated either at equilibrium with $\TL=\TR=2$, or for $\TL=3$ and $\TR=1$. The dynamics is discretized using a Geometric Langevin Algorithm scheme as in~\eqref{eq:gla}. The estimator of the asymptotic variance is made precise in Appendix~\ref{ap:variance}. The decorrelation time is set to $\tdeco = 3 N$ for the standard flux $R$ and to $\tdeco = 32$ for the modified observable.

We plot in Figure~\ref{fig:autocorrelation chain} the autocorrelation profiles of the heat flux for a chain of size $N=128$ at equilibrium, for two different anharmonicities. Let us comment this picture in greater detail. The results first show that the chosen decorrelation time $\tdeco$ is sufficiently large. Similar plots were used to check this is also the case for all the range of anharmonicities $b$ and numbers of particles $N$ we consider. They can also be used to understand the eventual variance reduction granted by the control variate~\eqref{eq:cv chain}. For a small anharmonicity $b=0.08$ we see that the asymptotic variance, which is twice the limit of the cumulated autocorrelation (right plot) is greatly reduced for two reasons. First, the signal amplitude, which is related to the autocorrelation value at $t=0$ (left plot), is slightly smaller, and (right plot) the contribution of the times $0\leqslant t \leqslant 2$ is twice smaller for the modified flux. Second, and that is the actual reason for the variance reduction, there is anticorrelation for $2 \leqslant t \leqslant 5$. For a larger anharmonicity it appears that the amplitude of the modified observable is much larger, but this is compensated by a long-time anticorrelation. The resulting asymptotic variance of the modified flux is slightly smaller than the one of the standard flux. The plots are essentially the same out of equilibrium, when $\TL = 3$ and $\TR=1$ (numerical results not presented here).

\begin{figure}
\begin{center}
\includegraphics[scale=0.5]{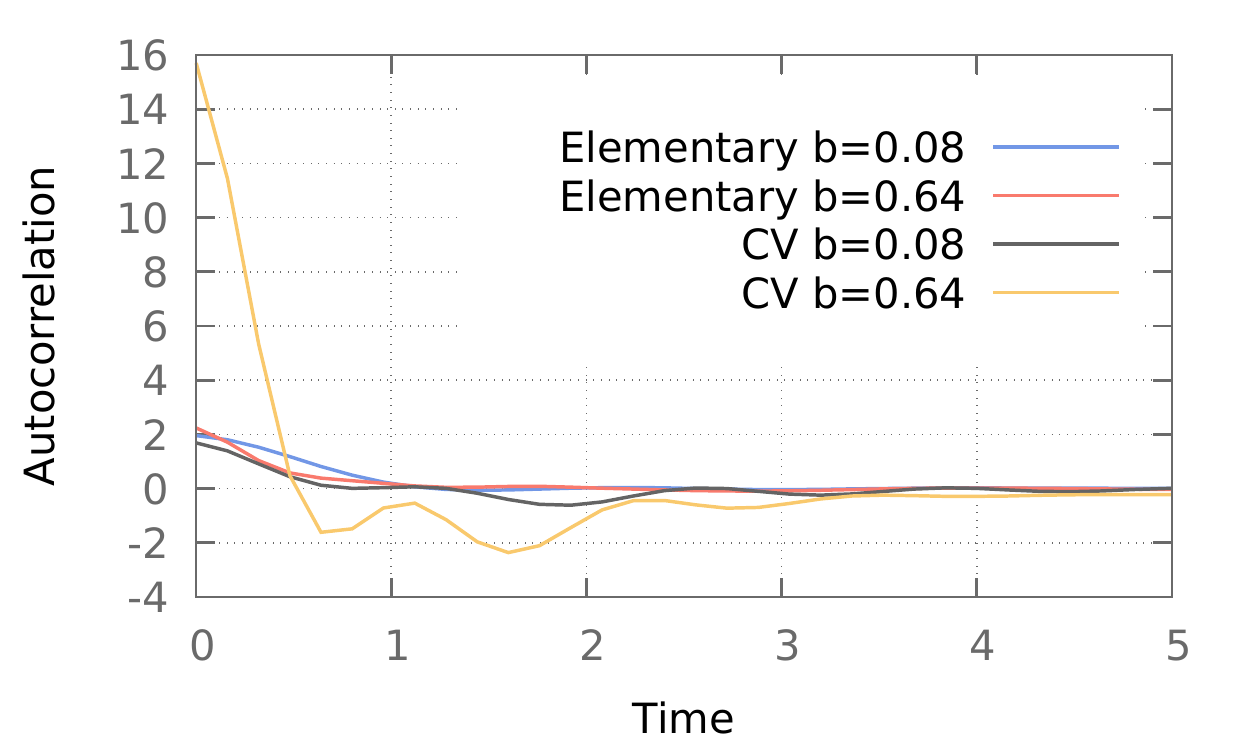}
\includegraphics[scale=0.5]{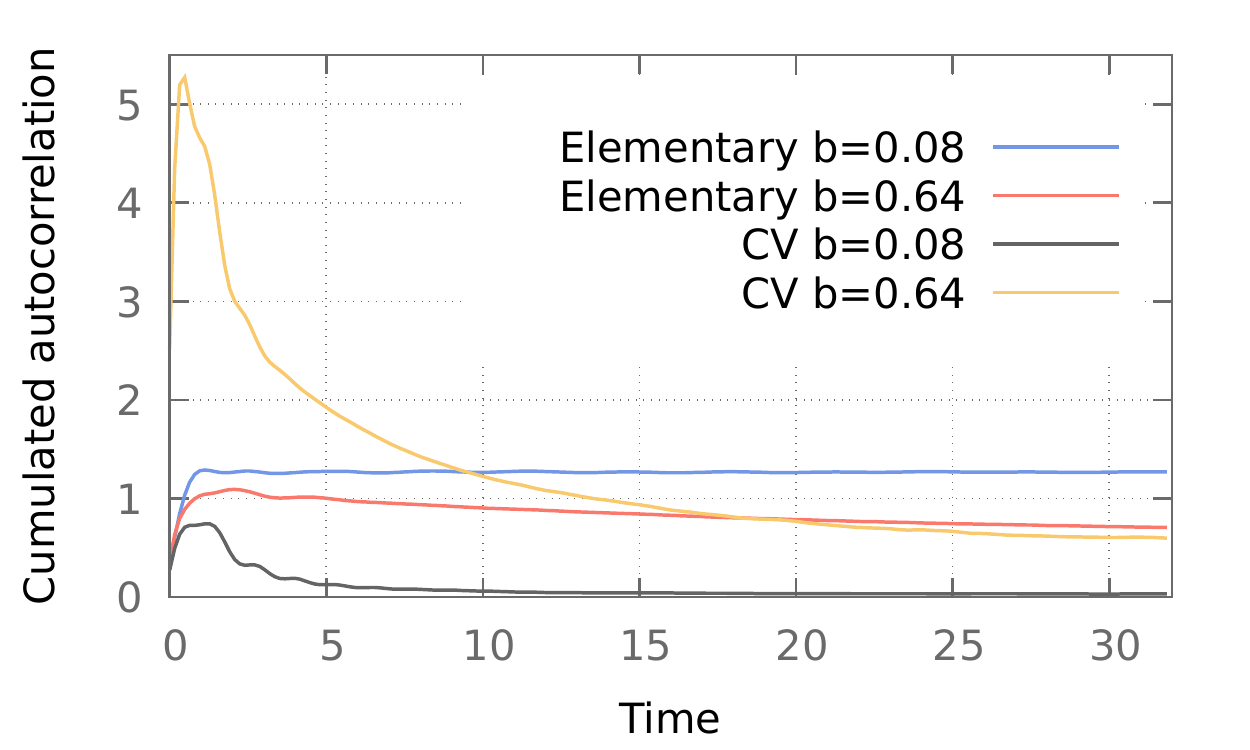}
\caption{Comparison of results obtained either with the elementary flux $j_n$~\eqref{eq:heat flux} or with the modified observable~\eqref{eq:modified obs chain}, for two different anharmonicities $b=0.08, 0.64$, and for a chain of $N=128$ particles at equilibrium. Left: Autocorrelation profile. Right: Cumulated autocorrelation $t \mapsto \int_0^t C$ at longer times.}
\label{fig:autocorrelation chain}
\end{center}
\end{figure}

The asymptotic variances, with associated error bars (see Appendix~\ref{ap:variance}), are plotted on Figure~\ref{fig:variance chain} for a whole range of anharmonicities $b$ and numbers of particles $N$. The two left plots are at equilibrium ($\TL=\TR = 2$) while the two right plots are out of equilibrium ($\TL=3$ and $\TR = 1$). We check that the variances are extremely similar in both cases, which is expected by linear response theory. We observe that the asymptotic variance of the modified flux scales as $b^2$ for $b \ll 1$, as expected from Theorem~\ref{th:var order 2 eta exact}, providing an excellent variance reduction in this case. Note that, in the limit $b \to 0$, the variance of the standard flux tends to $\frac {\gamma \nu^2}{m \beta^2(1+\nu^2)}=2$ (since $\nu=\frac{m \omega}{\gamma}=1$ here), which is the theoretical value predicted at equilibrium in view of the expression of the mean flux for a harmonic chain (see Equations~\eqref{eq:cv chain},~\eqref{eq:variance and conductivity}, and~\eqref{eq:def conductivity}). The modified flux can sometimes have a larger variance than the standard one, for example in the regime $b=\rmO(1)$ and $N$ large. Note that for the particular choice $\omegahar=0$ the modified flux is $\frac 1 2 (j_1 + j_{N-1})$, which has the same asymptotic variance as the standard flux~$\widetilde R$ according to~\eqref{eq:flux same asymptotic variance}. Therefore, for any set of parameters, there exist an optimal choice of the coefficients $\omegahar, \rhar$ providing a modified flux whose asymptotic variance is smaller or equal to its counterpart without control variate. In the present application these two coefficients are instead chosen according to the heuristic~\eqref{eq:rhar omegahar values}, leading to a degradation of the asymptotic variance in certain cases.

\begin{figure}
\begin{center}
\includegraphics[scale=0.5]{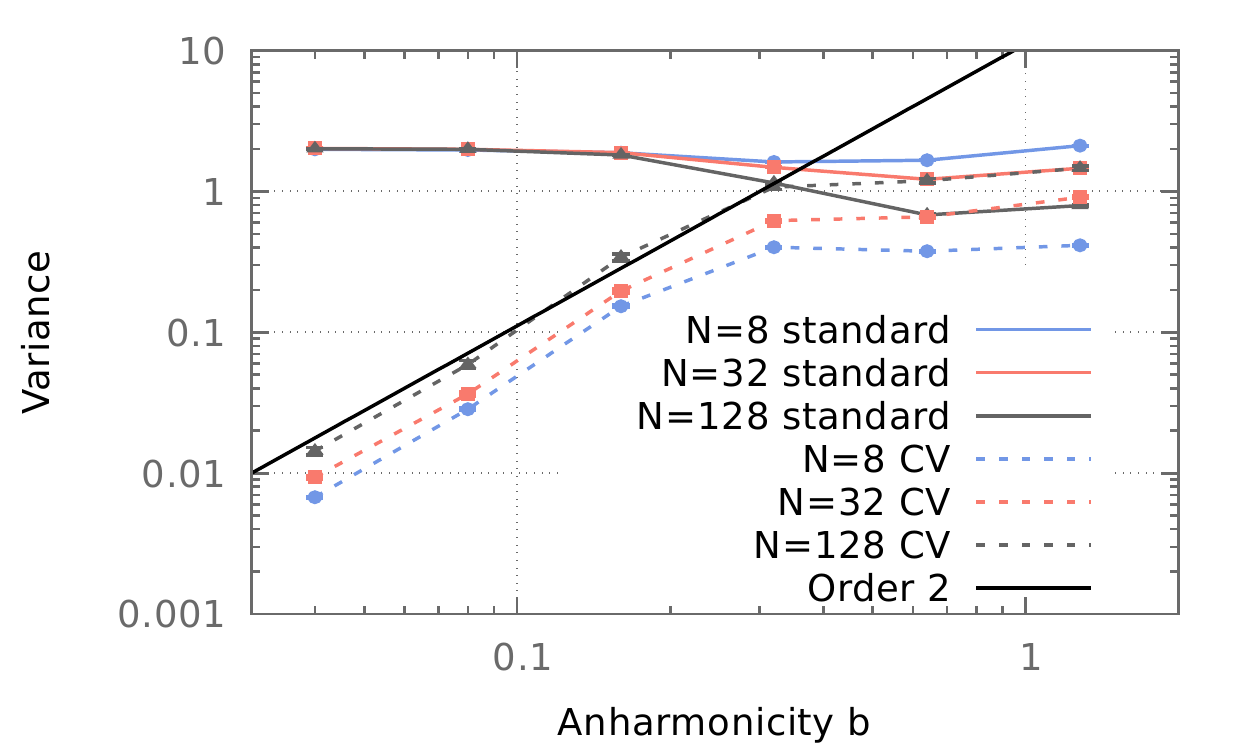}
\includegraphics[scale=0.5]{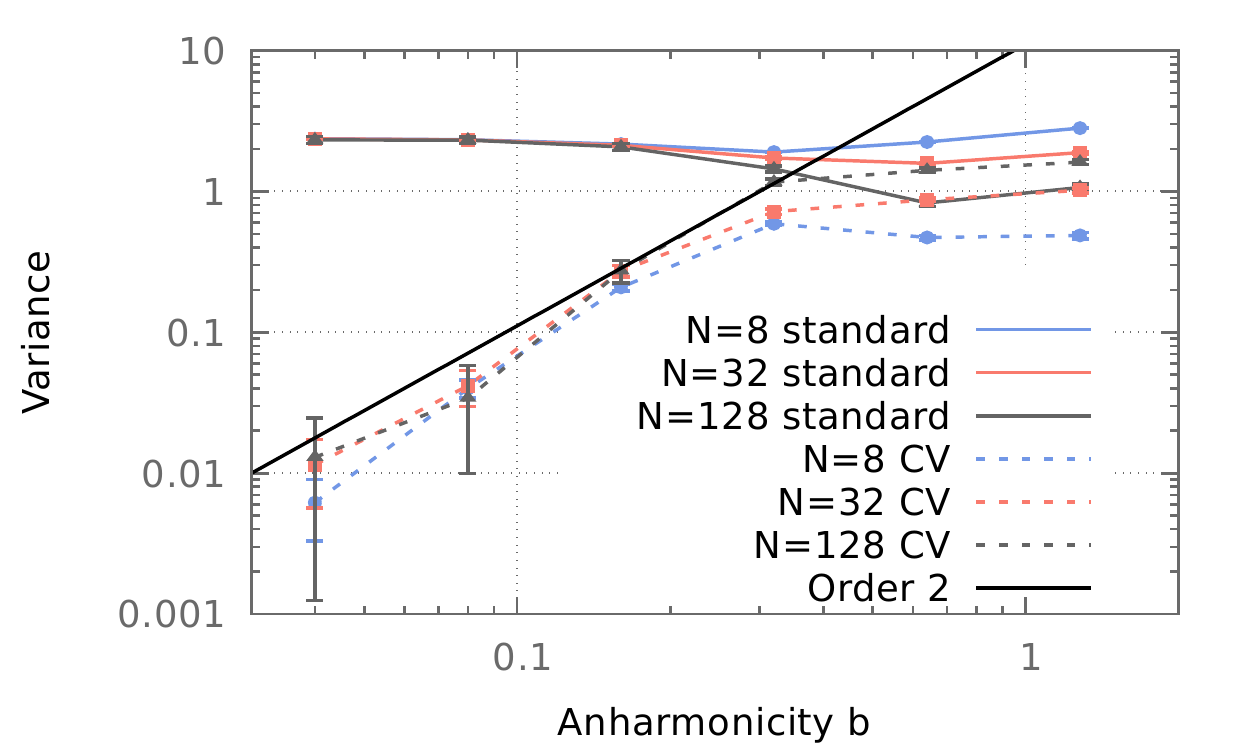}\\
\includegraphics[scale=0.5]{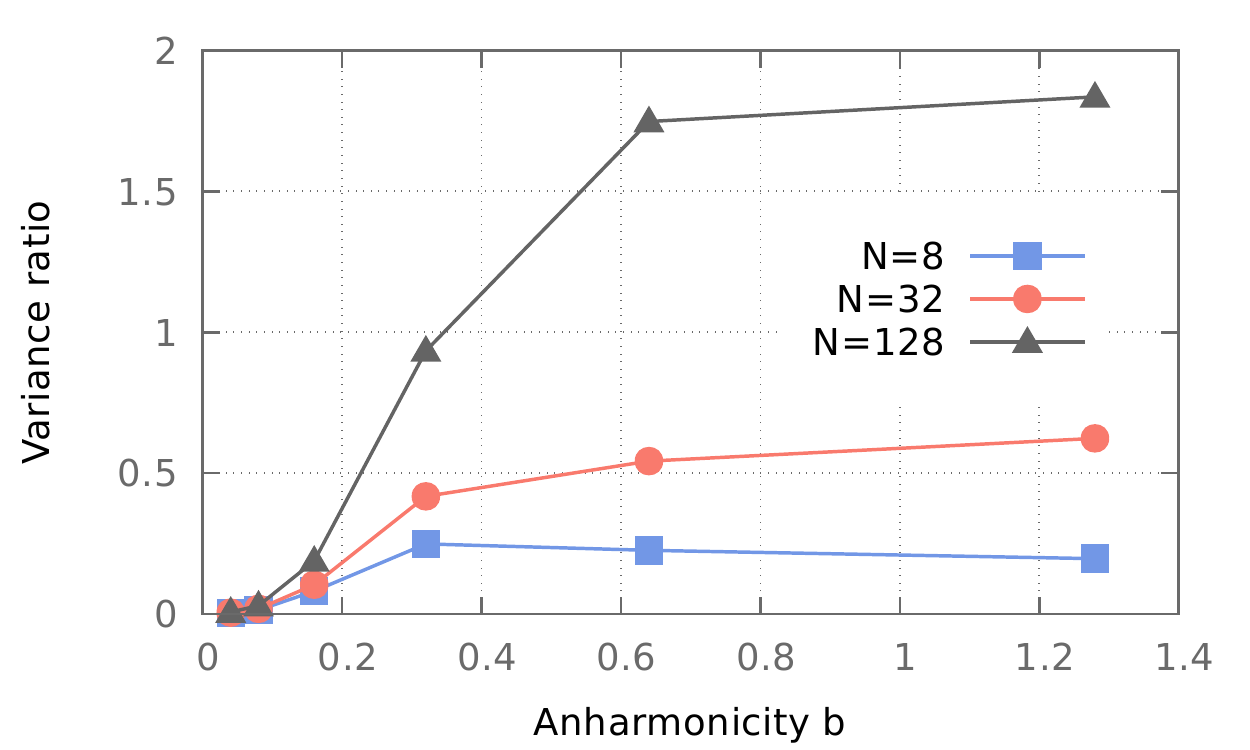}
\includegraphics[scale=0.5]{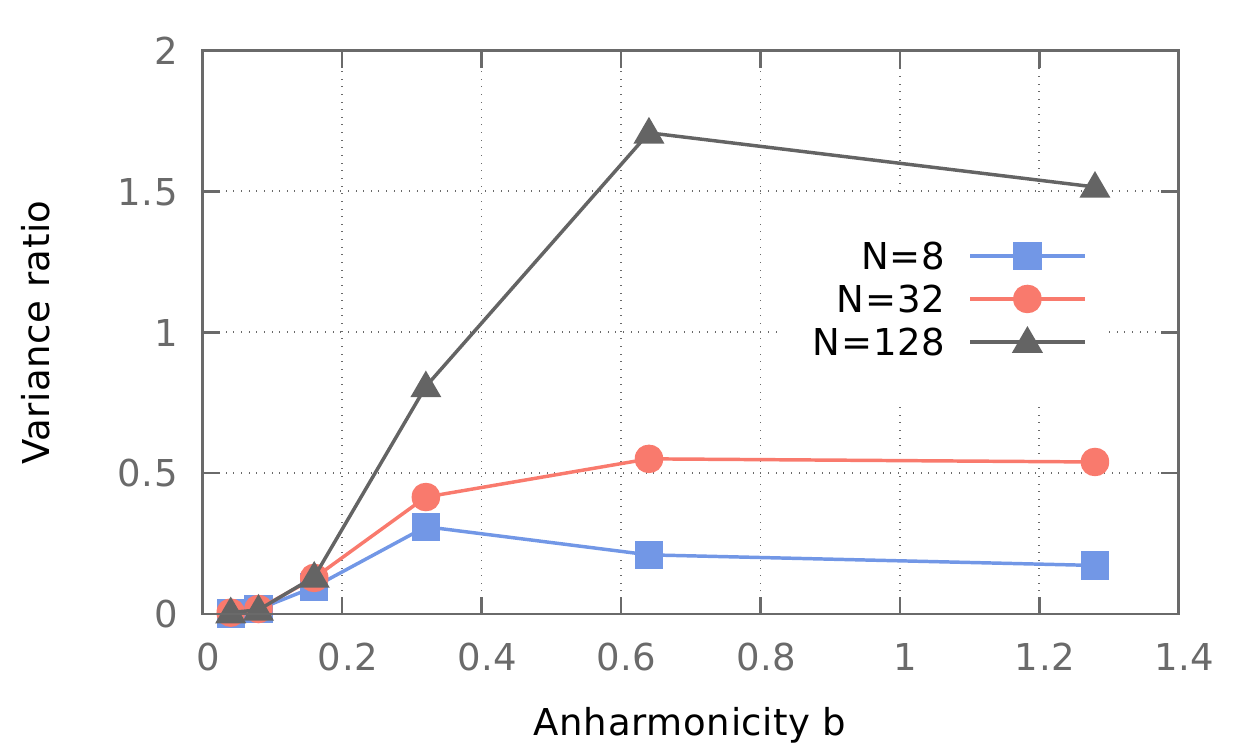}
\caption{Comparison of variances for the standard and modified fluxes. The reported variance ratio corresponds to the modified variance divided by the standard one. Left: Chain at equilibrium ($\TL=\TR=2$). Right: Chain out of equilibrium ($\TL=3$ and $\TR=1$).}
\label{fig:variance chain}
\end{center}
\end{figure}

\section{Solvated dimer under shear}
\label{s:dimer}

\newcommand{\naqi}{\nabla_{q_i}}
\newcommand{\Deqi}{\Delta_{q_i}}

A solvated dimer is a pair of bonded particles in a bath constituted of many other particles. It serves as a prototypical model of a molecule in solvent (e.g. peptide in water). This model has been used in the context of free energy computations~\cite{Lelievre10}. We apply to this system an external shearing force as in~\cite{Joubaud12}, also coined sinusoidal transverse field in the physics literature (see~\cite[Section 9.1]{Todd17} and \cite{Evans13_book, Todd07}), so that the invariant measure of the system is not known. A typical question is the influence of the shear force on the average bond length of the dimer.

\subsection{Full dynamics}

We consider $N$ particles in a two-dimensional box of length $L$ with periodic boundary conditions, with positions $q=(q_1,\cdots,q_N) \in \calD = (L \bbT)^{2N}$. Two of these particles, with positions $q_1$ and $q_2$, form a dimer whereas the other $N-2$ particles, with positions $q_3,\cdots,q_N$, constitute the solvent. The potential energy of the system is composed of three parts:
\begin{equation*}
  \begin{aligned}
    V(q) &= v(|q_1-q_2|) + \sum_{i \in \{1,2\}} \sum_{j=3}^N v_{\rm sol} (|q_i-q_j|) + \sum_{3 \leqslant i<j \leqslant N} v_{\rm sol} (|q_i-q_j|) \\
    &=: V_{\rm dim}(q) + V_{\rm inter}(q) + V_{\rm sol}(q),
  \end{aligned}
\end{equation*}
where $V_{\rm dim}$ is the potential energy of the dimer, $V_{\rm inter}$ is the interaction energy between the dimer and the solvent and $V_{\rm sol}$ is the potential energy of the solvent. The two particles forming the dimer interact via a double-well potential: denoting by $r=|q_1-q_2|$ the bond length,
\begin{equation}
  \label{eq:double well}
  v(r) = h \left[ 1 - \left( \frac {r-r_0}{\Delta r} \right)^2 \right]^2,
\end{equation}
where $r_0$, $\Delta r > 0$ (see Figure~\ref{fig:potentials}, Left). The potential $v$ presents two minima: one associated with a compact state of length $r=r_0-\Delta r$ and one associated with a stretched state of length $r=r_0+\Delta r$. These minima are separated by a potential barrier of height $h$. The particles of the solvent interact both with the other particles of the solvent and the particles of the dimer through a purely repulsive potential. In the following we consider two types of potentials with compact support (see Figure~\ref{fig:potentials}, Right): a soft repulsion potential (used in~\cite{Groot97} for example)
\begin{equation}
\label{eq:soft}
\forall r > 0, \quad v_{\rm sol}(r) = \varepsilon \left( 1 - \frac r {\rcut} \right) ^2  \bbone_{r \leqslant \rcut},
\end{equation}
where $\varepsilon, \rcut > 0$; and a singular Coulomb-like potential:
\begin{equation}
  \label{eq:coulomb}
  \forall r > 0, \quad v_{\rm sol}(r) = \varepsilon \left( \frac {\frac 1 {\sqrt r} - \frac 1 {\sqrt {\rcut}}} {\frac 1 {\sqrt \sigma} - \frac 1 {\sqrt {\rcut}}} \right)^2 \bbone_{\rm r < \rcut} = \varepsilon \frac \sigma r \left( \frac{1 - \sqrt{\frac r {\rcut}}} {1 - \sqrt{\frac \sigma {\rcut}}} \right)^2 \bbone_{r \leqslant \rcut}.
\end{equation}
This potential behaves like $\frac 1 r$ for $r \to 0$, reaches the value $\varepsilon$ at $r=\sigma$ and vanishes at $r=\rcut$,where its derivative also vanishes. Note that we recover the Coulomb potential $\varepsilon \frac \sigma r$ in the limit $\rcut \to +\infty$.
\begin{figure}
\caption{Pairwise potentials for the solvated dimer model. Left: Potential for the dimer and associated free energy in vacuum (see~\eqref{eq:1D dimer poisson problem}). Right: Potentials for the solvent interaction.}
\begin{center}
\includegraphics[scale=.58]{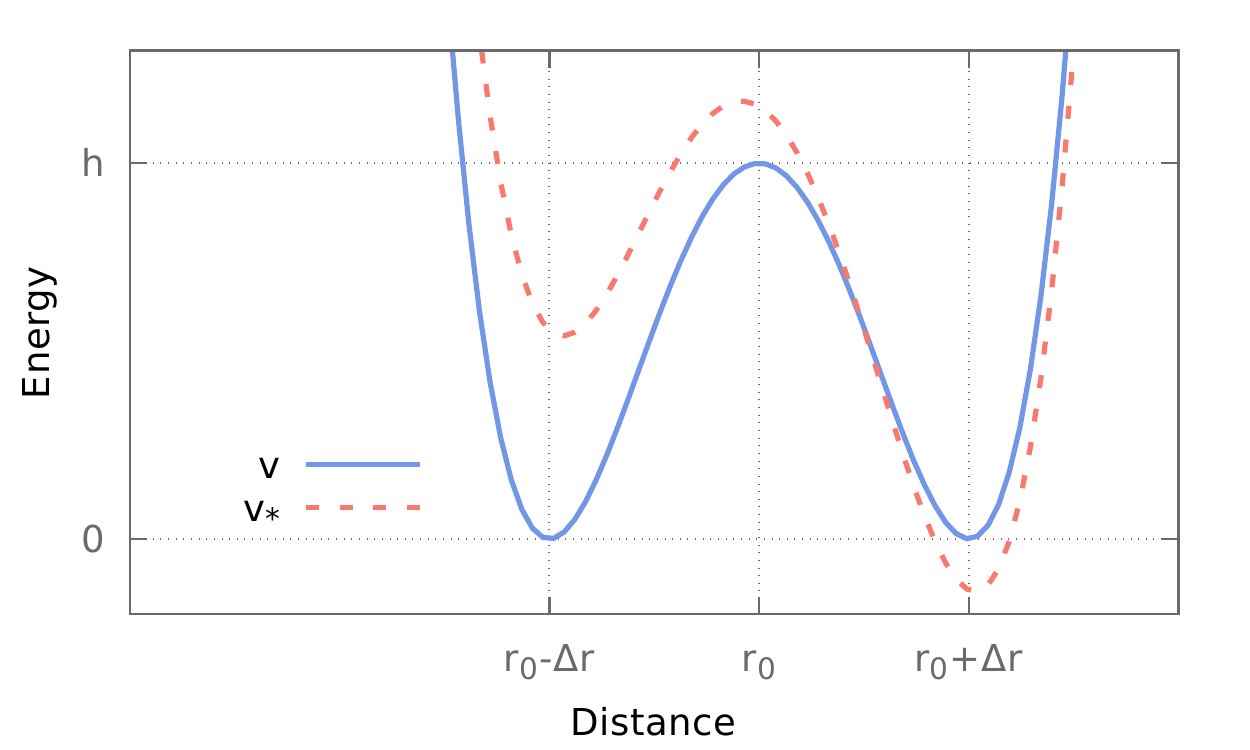}
\includegraphics[scale=.58]{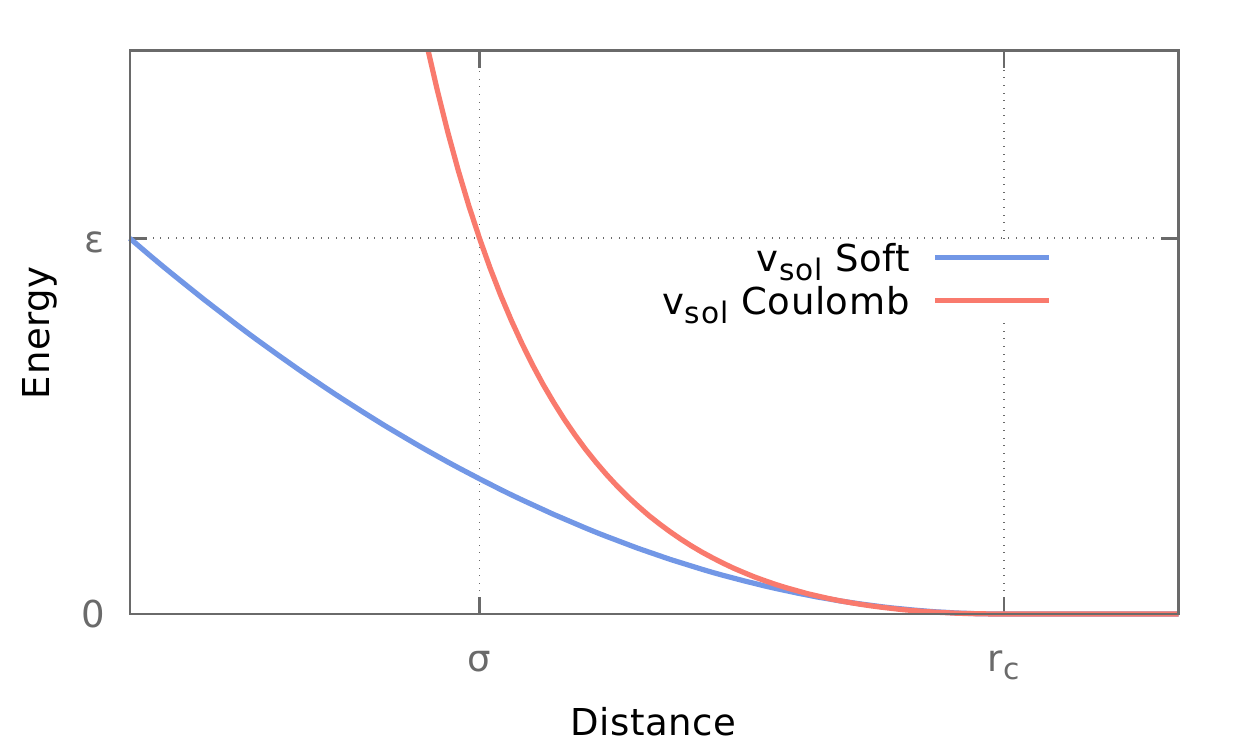}
\end{center}
\label{fig:potentials}
\end{figure}

The system is driven out of equilibrium by a shearing force of amplitude $\nu$. More precisely, a particle located at a position $(q_{i,x}, q_{i,y})$ experiences the force~\cite{Joubaud12, Todd17}:
\begin{equation*}
	(0,f(q_{i,x}))= \left( 0,\nu \sin(2\pi\frac{q_{i,x}}L) \right).
\end{equation*}
This force is in the $y$ direction and depends only on $x$. It therefore induces a non-equilibrium forcing since it is not of gradient type. We are interested in computing the mean length of the dimer $R(q,p) = |q_1 - q_2|$ as a function of this external forcing. The corresponding average is denoted by $\Esp[|q_1 - q_2|]$.

For simplicity we study the overdamped dynamics associated with~$V$, but everything can be adapted to the Langevin case. Since the space $\calD$ is compact and the noise in the dynamics is non degenerate, there exist a unique invariant probability measure $\pi$ by the Doeblin condition when the potentials under consideration are smooth. This invariant measure depends on~$\nu$, and is not explicit. Proving a similar result for singular potentials such as the Coulomb-like potential~\eqref{eq:coulomb} would require more work. 

The generator can be decomposed as:
\begin{equation*}
\begin{aligned}
  \calL &= -\nabla V(q) \cdot \nabla + \betainv \Delta + \nu \sum_{i=1}^N f(q_{i,x}) \partial_{q_{i,y}} 
  = \Ldim + \Linter + \Lsol + \nu \Lpert,	
\end{aligned}
\end{equation*}
where
\begin{alignat*}{3}
	\Ldim &= \sum_{i=1,2} \left( -\naqi V_{\rm dim}(q) \cdot \naqi + \betainv \Deqi \right),& \qquad
	\Linter &= - \nabla V_{\rm inter}(q) \cdot \nabla,\\
	\Lsol &= \sum_{i=3}^N \left( -\naqi V_{\rm sol}(q) \cdot \naqi + \betainv \Deqi \right), &\qquad
	\Lpert &= \sum_{i=1}^N f(q_{i,x}) \partial_{q_{i,y}}.
\end{alignat*}
Note that $\Ldim$ is the generator of the dynamics of the dimer at equilibrium in vacuum and $\Lsol$ is the generator of the dynamics of the solvent at equilibrium and without dimer.

\subsection{Simplified dynamics and control variate}

We consider the following reference Poisson equation where the system is at equilibrium and the interaction between the dimer and the solvent has been switched off:
\begin{equation}
\label{eq:dimer ref poisson problem}
	-\Lnot \Phinot = R - \Espnot[R],
\end{equation}
where
\begin{equation*}
	\Lnot = \Ldim + \Lsol.
\end{equation*}
Let us show that this equation admits a solution $\Phinot$ depending only on the length $|q_1-q_2|$ of the dimer. In order to highlight the dependence on the dimension of the underlying space, let us denote by $d=2$ this dimension. Assume that $\Phinot$ is defined for any $q \in (L \bbT)^{dN}$ by $\Phinot(q) = \frac 1 2 \psi(|q_1-q_2|)$ for some smooth function $\psi$. The Laplacian of $\Phinot$ can be rewritten using spherical coordinates as:
\begin{equation}
\label{eq:spherical laplacian}
	\Delta \Phinot(q) = \psi''(|q_1-q_2|) + \frac {d-1}{|q_1-q_2|} \psi'(|q_1-q_2|),
\end{equation}
where $d=2$ is the dimension of the underlying physical space. We obtain by substituting $\Phinot$ into~\eqref{eq:dimer ref poisson problem} that $\psi$ satisfies the following one-dimensional differential equation:
\begin{equation}
\begin{aligned}
\label{eq:1D dimer poisson problem}
	\forall r >0, \quad v_*'(r) \psi'(r) - \betainv \psi''(r) &= r - r^*,
\end{aligned}
\end{equation}
where $v_*(r) = v(r) - \frac {d-1} \beta \ln(r)$ and $r^* = \Esp_*[r]$ is the expectation of the length $r$ with respect to the probability measure $\pi_*(\dd r) = Z_*^{-1} \rme^{-\beta v_*(r)} \, \dd r$. Note the additional term $- \frac {d-1} \beta \ln(r)$ in the expression of $v_*$ coming from~\eqref{eq:spherical laplacian}, which can be interpreted as an entropic contribution. 

Let us first discuss the well-posedness of~\eqref{eq:1D dimer poisson problem}. The double-well potential~\eqref{eq:double well} considered here is such that $v_*$ is a bounded perturbation of a convex function. Therefore $\pi_*(\dd r)$ satisfies a log-Sobolev inequality and thus a Poincaré inequality by the Holley-Stroock theorem~\cite{Holley87} and the Bakry-Emery criterion~\cite{Bakry85}. This implies that the one-dimensional Poisson problem~\eqref{eq:1D dimer poisson problem} then admits a unique solution in 
\[
\rmH^1(\pi_*) \cap \rmL_0^2(\pi_*) = \left\{ \varphi \in H^1(\pi_*), \ \int_0^{+\infty} \varphi \, \dd\pi_* = 0 \right\}
\]
by the Lax-Milgram theorem for the variational formulation:
\begin{equation*}
  \forall u \in \rmH^1(\pi_*) \cap \rmL_0^2(\pi_*), \quad \betainv \int_0^\infty \psi'(r) u'(r) \, \pi_*(\dd r) = \int_0^\infty (r-r_*) u(r) \, \pi_*(\dd r).
\end{equation*}
We discuss precisely in Appendix~\ref{ap:resolution ode} how we numerically solve~\eqref{eq:1D dimer poisson problem}. Knowing the solution $\psi$, the corresponding modified observable then writes
\begin{equation*}
\begin{aligned}
  (R + \calL \Phinot)(q) &= |r_{12}| + \betainv \psi''(|r_{12}|) \\
  & + \left[ \frac 1 2 \big( \nabla_{q_1} V(q) - \nabla_{q_2} V(q) - \nu (f(q_{1,x}) - f(q_{2,x}))e_y \big) \cdot \frac {r_{12}}{|r_{12}|} +  \frac {d-1}{\beta |r_{12}|} \right] \psi'(|r_{12}|),
\end{aligned}
\end{equation*}
where $r_{12} = q_2 - q_1$ and $e_y=(0,1)$. Note that $\nabla_{q_1} V$ and $\nabla_{q_2} V$ are the forces that apply on particles $1$ and $2$ respectively, which depend also on the solvent variables.

\subsection{Numerical results}

We simulate a system of $N=64$ particles in $d=2$ dimensions, using periodic boundary conditions. We fix $L=8$ (so that the particle density is $1$), and $\beta = 1$. The parameters of the potentials are set to $\rcut=2.5$, $\varepsilon=1$, $h=1$, $r_0=3$ and $\Delta r = 1$ (see Figure~\ref{fig:potentials}). For the finite difference method used to solve the Poisson equation~\eqref{eq:1D dimer poisson problem} we use a mesh size $\Delta r = 10^{-3}$ on an interval $[0,r_{\rm max}]$ with $r_{\rm max} = 10$ (see Appendix~\ref{ap:resolution ode}).

The influence of the shearing on the average dimer length is plotted on Figure~\ref{fig:dimer length}. We see that a shear force of amplitude $\nu=1$ increases the mean length by roughly $1\%$, and that the response of the mean length to the nonequilibrium forcing is of order $2$. The response is small thus difficult to estimate accurately, hence the need for control variates to alleviate this issue.

In the case of an unsolvated dimer, Figure~\ref{fig:dimer variance reduction} (Left) shows that the variance of the modified observable scales like $\nu^2$, as predicted by Theorem~\ref{th:var order 2 eta exact}. Note that in the limit $\nu \to 0$ the modified observable is the constant $\Esp_0[R] = r_*$, which is computed by a numerical quadrature, so that the variance converges to zero.

When the solvent interacts with the dimer the variance of the modified observable plateaus at a certain value when $\nu \to 0$ , as expected from Theorem~\ref{th:var order 2 eta}. For the soft potential~\eqref{eq:soft}, the variance scales like $\nu$ for a forcing amplitude of order $1$, which is expected from Theorem~\ref{th:var order 2 eta} (see Figure~\ref{fig:dimer variance reduction}, Right). The variance stabilizes at a value which is ten times smaller than the initial one. For the Coulomb-like potential the influence of the solvent on the dimer is stronger and the control variate does not perform as well, as seen on Figure~\ref{fig:dimer variance reduction} (Bottom). For a small shearing the variance is however reduced by a factor 4.

\begin{figure}
\caption{Left: Mean length of a dimer, either unsolvated (in vacuum) or in a solvent with soft or Coulomb-like potential. Right: Relative variation of this mean length induced by the shearing. The solid line represents the reference scaling $\nu^2$.}
\begin{center}
\includegraphics[scale=.58]{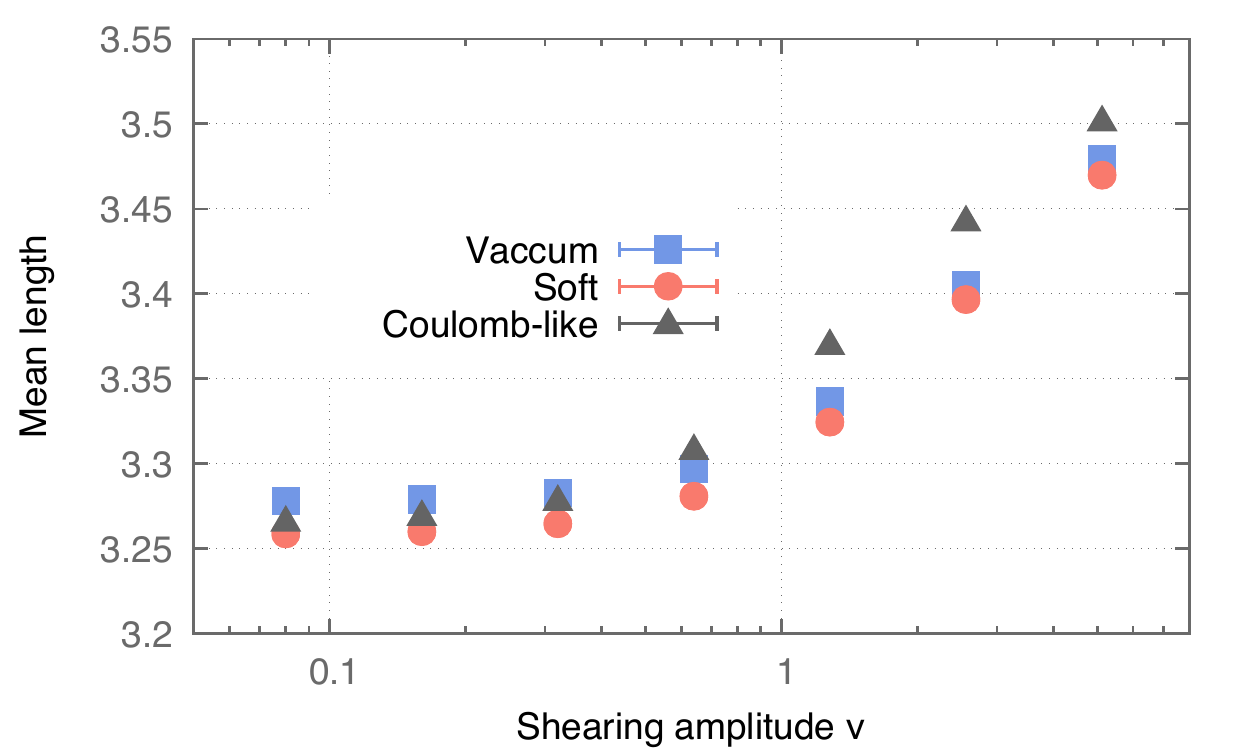}
\includegraphics[scale=.58]{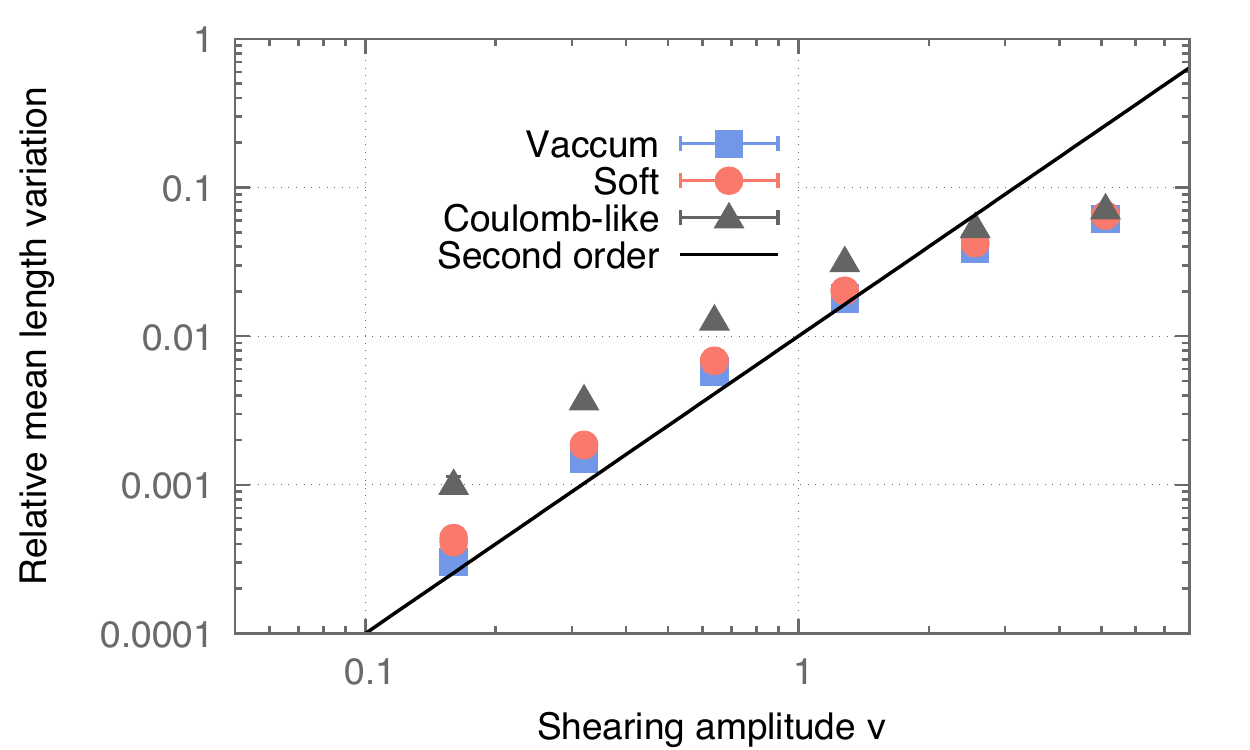}
\end{center}
\label{fig:dimer length}
\end{figure}

\begin{figure}
\caption{Asymptotic variance of the length of the dimer, with or without control variate. Left: Unsolvated dimer. Right: Solvent with the soft potential~\eqref{eq:soft}. Bottom: Solvent with the Coulomb-like potential~\eqref{eq:coulomb}.}
\begin{center}
\includegraphics[scale=.58]{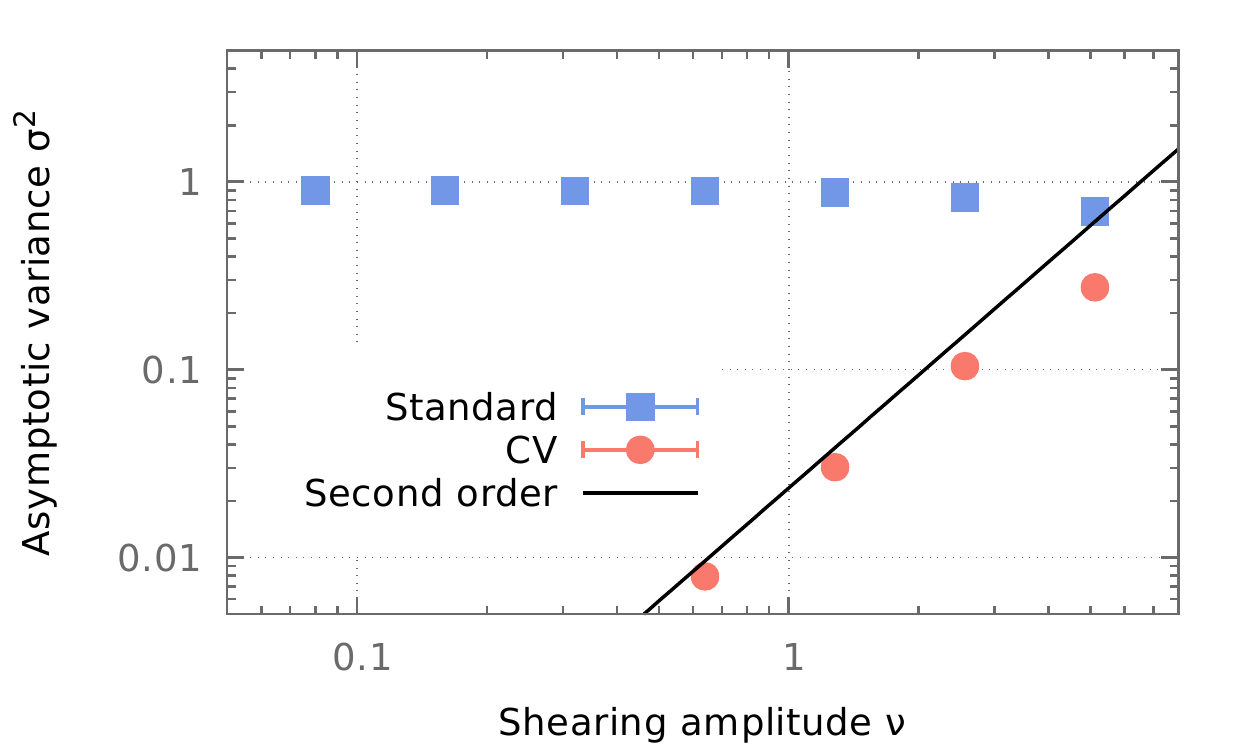}
\includegraphics[scale=.58]{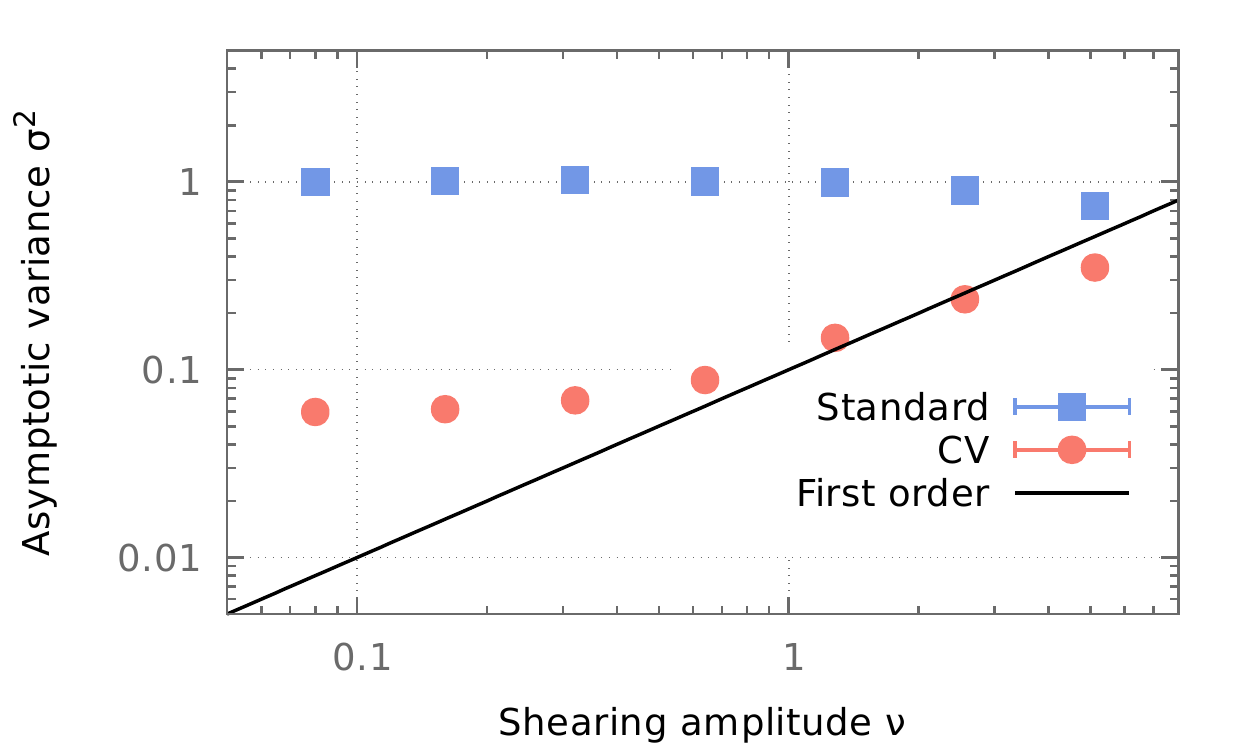}
\includegraphics[scale=.58]{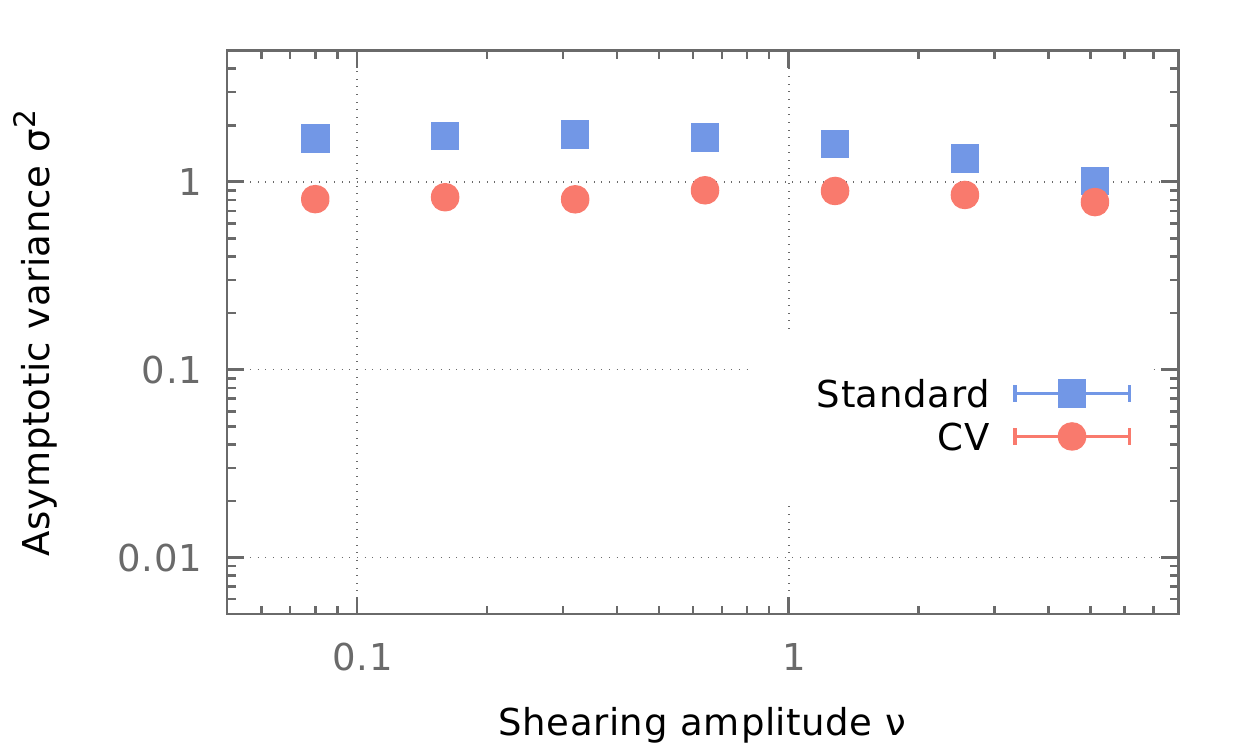}
\end{center}
\label{fig:dimer variance reduction}
\end{figure}

\paragraph{Generalization.}

The variance reduction strategy discussed here can be easily adapted to similar systems. For example Langevin dynamics can be treated by replacing~\eqref{eq:1D dimer poisson problem} by a two-dimensional PDE where the variables are the dimer length and the radial part of the momentum associated to this length. The Poisson equation should then be solved using a Galerkin approximation similar to what is done Section~\ref{s:sinus potential}. One could also consider a solvated molecule more complex than a dimer. In this case the PDE~\eqref{eq:1D dimer poisson problem} would be posed in several dimensions and thus become rapidly impossible to solve in practice. In general one has to reduce the system to a few relevant variables corresponding to a simplified Poisson equation in order to use the control variate approach developed here. This is connected to coarse-graining, \emph{i.e.} finding a few (nonlinear) functions of the degrees of freedom providing some macroscopic information on the system -- think of identifying an appropriate molecular backbone for proteins. An alternative route, which does not require an a priori physical knowledge of the system, would be to resort to greddy methods~\cite{Mokdad07, Temlyakov08, Cances11, Figueroa12}. If the system possesses in addition a specific symmetry or structure, one can make profit of dedicated tensor formats~\cite{Hackbusch12} as done for the Schrödinger equation in~\cite{Yserentant10}. This approach would be particularly adapted when studying an isotropic system composed of identical particles for example. Let us also mention recent advances on the resolution of Poisson equations based on deep convolutional neural networks~\cite{Shan18, Avrutskiy17}, which offer the promise of a better scalability with respect to the dimension of the system.

\section*{Acknowledgements}

The idea of using control variates came out of discussions with Antonietta Mira (USI) while Gabriel Stoltz was participating to the workshop “Free-energy calculations: a mathematical perspective” at Oaxaca. We thank Antoine Levitt (ENPC), Greg Pavliotis (Imperial College), Stefano Lepri (ISC) and Jonathan Goodman (NYU) for helpful discussions. This work is supported by the Agence Nationale de la Recherche under grant ANR-14-CE23-0012 (COSMOS); as well as the European Research Council under the European Union's Seventh Framework Programme (FP/2007-2013) -- ERC Grant Agreement number 614492. We also benefited from the scientific environment of the Laboratoire International Associ\'e between the Centre National de la Recherche Scientifique and the University of Illinois at Urbana-Champaign. Part of this work was done during the authors’ stay at the Institut Henri Poincar\'e - Centre Emile Borel during the trimester “Stochastic Dynamics Out of Equilibrium” (April-July 2017). The authors warmly thank this institution for its hospitality.

\appendix
\section{Proofs of Theorems~\ref{th:var order 2 eta exact} and~\ref{th:var order 2 eta}}
\label{ap:proof Th2}

Let us first prove Theorem~\ref{th:var order 2 eta exact}, and deduce Theorem~\ref{th:var order 2 eta} in a second step. We suppose in all this section that Assumptions~1 to~5 hold true. The norm and scalar product indexed by $\eta$ correspond to the canonical ones on $\Lpieta$. We start by giving a useful technical result.
\begin{lemma}
\label{lemma:bound expect}
For any $\eta_* > 0$ and $n \in \bbN$, there exists $C_{n,\eta_*} \in \mathbb{R}_+$ such that, for any $|\eta| \leqslant \eta_*$,
\begin{equation*}
	\forall \varphi \in \rmLinfn, \quad\left| \Espeta[\varphi] - \Espnot[\varphi] \right| \leqslant C_{n,\eta_*} \, |\eta|\, \| \varphi \|_\rmLinfn.
\end{equation*}
\end{lemma}

\begin{proof}
For any $\psi \in \calS$,
\begin{equation*}
	\Espnot[\Leta \psi] = \eta \Espnot[\Lrep \psi],
\end{equation*}
so that, for a given $\varphi \in \calS$, the previous equality applied to $\psi = \Letainv \Pieta \varphi$ leads to
\begin{equation*}
  \Espnot[\Pieta \varphi] = \eta \Espnot[\Lrep \Letainv \Pieta \varphi] = \eta \lang \Lrep \Letainv \Pieta \varphi, \bfone \rangnot = \eta \lang \Letainv \Pieta \varphi, \Lrep^* \bfone \rangnot.
\end{equation*}
Since $\Espnot[\Pieta \varphi] = \Espnot[\varphi] - \Espeta[\varphi]$ and $|\varphi| \leqslant \| \varphi \|_{\rmL_n^\infty} \calK_n$, we obtain
\begin{equation*}
  \left| \Espeta[\varphi] - \Espnot[\varphi] \right| \leqslant |\eta| \big\| \Letainv \big\|_{\calB(\Pieta \rmLinfn)} \| \varphi \|_\rmLinfn \big\| \calK_n \big\|_0 \big\| \Lrep^* \bfone \big\|_0 \leqslant C_{\eta_*, n} |\eta| \| \varphi \|_\rmLinfn,
\end{equation*}
since $\Lrep^* \bfone \in \Lpinot$ by Assumption~\ref{as:generator simple} and $\big\| \calK_n \big\|_0 < +\infty$ by Assumption~\ref{as:lyapunov in L2}. The proof is concluded by the density of $\calS$ in~$\rmLinfn$.
\end{proof}

\begin{corollary}
\label{coro:scalar product}
For any $\eta_* > 0$ and $n, n' \in \bbN$, there exists $C_{n,n',\eta_*} \in \mathbb{R}_+$ such that, for any $|\eta| \leqslant \eta_*$,
\begin{equation}
\label{eq:bound scalar product}
	\forall \varphi \in \rmLinfn,\ \forall \psi \in \rmL_{n'}^\infty, \quad | \lang \varphi, \psi \rangeta - \lang \varphi, \psi \rangnot | \leqslant C_{n,n', \eta_*} |\eta| \| \varphi \|_\rmLinfn \| \psi \|_{\rmL_{n'}^\infty};
\end{equation}
and, for a given $\psi \in \calS$, there exists $C_{\psi, n, \eta_*} \mathbb{R}_+$ such that
\begin{equation}
\label{eq:bound scalar product L}
	\forall \varphi \in \rmLinfn, \quad | \lang \varphi, \Leta \psi \rangeta - \lang \varphi, \Lnot \psi \rangnot | \leqslant C_{\psi, n, \eta_*} |\eta| \| \varphi \|_\rmLinfn.
\end{equation}
\end{corollary}

\begin{proof}
  In view of Assumption~\ref{as:lyapunov in L2} there exist $m \in \bbN$ depending only on $n$ and $n'$ such that $\| \calK_n \calK_{n'} \|_{\rmL_{m}^\infty} < +\infty$. Therefore, writing
\begin{equation*}
  \varphi \psi = \frac \varphi {\calK_n} \frac \psi {\calK_{n'}} \, \calK_n \calK_{n'}.
\end{equation*}
we obtain
\begin{equation*}
  \| \varphi \psi \|_\rmLinfm \leqslant \| \varphi  \|_\rmLinfn \, \| \psi \|_{\rmL_{n'}^\infty} \, \| \calK_n \calK_{n'} \|_\rmLinfm.
\end{equation*}
The estimate~\eqref{eq:bound scalar product} then follows from Lemma~\ref{lemma:bound expect} since $\lang \varphi, \psi \rangeta = \Espeta[\varphi \psi]$. Fix now $\psi \in \calS$. There exist $n',n'' \in \bbN$ such that $\Lnot \psi \in \rmL_{n'}^\infty$ and $\Lrep \psi \in \rmL_{n''}^\infty$. Therefore, using Lemma~\ref{lemma:bound expect} twice,
\begin{equation*}
\begin{aligned}
  & | \lang \varphi, \Leta \psi \rangeta - \lang \varphi, \Lnot \psi \rangnot | \leqslant \big| \lang \varphi, \Lnot \psi \rangeta - \lang \varphi, \Lnot \psi \rangnot \big| + |\eta | \left| \lang \varphi, \Lrep \psi \rangeta \right|\\
  & \qquad \leqslant C_{n,n',\eta_*} |\eta|\, \| \varphi \|_\rmLinfn \| \Lnot\psi \|_{\rmL_{n'}^\infty} + |\eta| \left| \lang \varphi, \Lrep \psi \rangnot \right| + \eta^2 C_{n,n'',\eta_*} \| \varphi \|_\rmLinfn \| \left\| \Lrep \psi \right\|_{\rmL_{n''}^\infty}.
\end{aligned}
\end{equation*}
This implies~\eqref{eq:bound scalar product L} since $n'$ and $n''$ depend only on $\psi$.
\end{proof}

We can now provide the proof of Theorem~\ref{th:var order 2 eta exact}.

\begin{proof}[Proof of Theorem~\ref{th:var order 2 eta exact}.]
When $\Pi_0 A$ is not bounded one needs to define an approximation of $-\Letainv \Pi_\eta \phi_\eta$ at order $K$ in $\eta$, as done in~\cite{Redon16} for instance:
\begin{equation*}
  Q^K := -\Pi_\eta \Lnotinv \Pi_0 \sum_{k=1}^K \eta^k A^k R \in \calS.
\end{equation*}
Let us show that this is indeed a good approximation. Using successively~\eqref{eq:Leta Linv} and~\eqref{eq:projected modified obs} the corresponding truncation error reads:
\begin{equation*}
\begin{aligned}
	\Pi_\eta \phi_\eta + \Leta Q^K &= \Pi_\eta \phi_\eta - \Leta \Lnotinv \Pi_0 \sum_{k=1}^K \eta^k A^k R \\
	&= \eta \Pi_\eta A R - \Pi_\eta (1 - \eta A) \sum_{k=1}^K \eta^k A^k R \\
	&= \eta^{K+1} \Pi_\eta A^{K+1} R,
\end{aligned}
\end{equation*}
which implies:
\begin{equation}
\label{eq:QK}
	Q^K + \Letainv \Pieta \phi_\eta = \eta^{K+1} \Letainv \Pi_\eta A^{K+1} R.
\end{equation}

Let us first show that the corresponding approximated asymptotic variance $\sigma_{\phi_{\eta, K}}^2 := 2 \lang \Pi_\eta \phi_\eta, Q^K \rangeta$ is close to $\sigma_{\phi_\eta,\eta}^2$ (defined in~\eqref{eq:variance eta}). Indeed,
\begin{equation*}
\begin{aligned}
	\sigma_{\phi_\eta, \eta}^2 - \sigma_{\phi_{\eta, K}}^2 &= 2 \lang \Pi_\eta \phi_\eta, -\Letainv \Pi_\eta \phi_\eta - Q^K \rangeta = 2 \eta^{K+1} \lang \Pi_\eta \phi_\eta, -\Letainv \Pi_\eta A^{K+1} R \rangeta.
\end{aligned}
\end{equation*}
Note that $\Pi_\eta A^{K+1} R \in \calS$ because $\calS$ is stable by $\Lnotinv \Pi_0$ and $\Lrep$ in view of Assumptions~\ref{as:generator simple} and~\ref{as:Lrep}. Since $\Pi_\eta \phi_\eta \in \calS$ as well, there exist $n \in \bbN$ (depending on $R$ and $K$) and $m \in \bbN$ (depending on $R$) such that $\Pi_\eta A^{K+1} \Pi_0 R \in \rmLinfn$ and  $\Pi_\eta \phi_\eta \in \rmLinfm$. Note that $m$ does not depend on $\eta$ in view of the expression~\eqref{eq:modified obs} of $\phi_\eta$. Using Assumption~\ref{as:generator simple} we obtain, for any $\eta_* > 0$ and $|\eta| \leqslant \eta_*$,
\begin{equation}
\label{eq:bounds app}
\begin{aligned}
	| \sigma_{\phi_\eta,\eta}^2 - \sigma_{\phi_{\eta, K}}^2 | &\leqslant 2 |\eta|^{K+1} \| \Pi_\eta \phi_\eta \|_\rmLinfm \left\| \Letainv \Pi_\eta A^{K+1} R \right\|_\rmLinfn \lang \calK_m, \calK_n \rangeta \\
	&\leqslant 2 |\eta|^{K+2} \| \Pieta A R \|_\rmLinfm \left\| \Letainv \right\|_{\calB(\Pieta \rmLinfn)} \left\| A^{K+1} R \right\|_\rmLinfn \| \calK_m \|_{L^2(\pi_\eta)} \ \| \calK_n \|_{L^2(\pi_\eta)},
\end{aligned}
\end{equation}
where the four terms on the right hand side are uniformly bounded for $|\eta| \leqslant \eta_*$ in view of Assumption~\ref{as:generator eta} and Lemma~\ref{lemma:bound expect}. This shows that there exists $C_{R, \eta_*, K} \in \mathbb{R}_+$ such that, for any $|\eta| \leqslant \eta_*$,
\begin{equation}
\label{eq:truncation error sigma}
\sigma_{\phi_\eta,\eta}^2 - \sigma_{\phi_{\eta, K}}^2 = \eta^{K+2} E_{R, \eta, K},
\end{equation}
where $|E_{R, \eta, K}| \leqslant C_{R, \eta_*, K}$.

At this stage it is sufficient to prove the expansion~\eqref{eq:var order 2 eta exact} for $\sigma_{\phi_{\eta, K}}^2$. The approximate variance $\sigma_{\phi_{\eta, K}}^2$ can be expanded in powers of $\eta$ as follows:
\begin{equation*}
  \begin{aligned}
    \sigma_{\phi_{\eta, K}}^2 &= 2 \lang \Pi_\eta \phi_\eta, Q^K \rangeta 
    = 2 \lang \eta \Pi_\eta A R, Q^K \rangeta 
    =  -2 \eta \sum_{k=1}^K \eta^k \lang \Pi_\eta A R,  \Lnotinv \Pi_0 A^k R \rangeta.
  \end{aligned}
\end{equation*}
In fact it suffices to consider $K=1$. We use Lemma~\ref{lemma:bound expect} and Corollary~\ref{coro:scalar product} to replace integrals with respect to $\pieta$ by integrals with respect to $\pinot$: there exists $C_{R, \eta_*} \in \mathbb{R}_+$ such that, for any $|\eta| \leqslant \eta_*$,
\begin{equation*}
  \sigma_{\phi_{\eta, 1}}^2 = -2 \eta^2 \lang \Pi_\eta A R,  \Lnotinv \Pi_0 A R \rangeta = -2 \eta^2 \lang A R,  \Lnotinv \Pi_0 A R \rangnot + \eta^3 \widetilde{E}_{R, \eta},
\end{equation*}
with $|\widetilde{E}_{R, \eta}| \leqslant C_{R, \eta_*}$. The claimed result then follows by~\eqref{eq:truncation error sigma}.
\end{proof}

\newcommand{\LetaS}{\calL_\eta^{\rm S}}
The following lemma is useful for the proof of Theorem~\ref{th:var order 2 eta} and is also used in Section~\ref{s:quasi harmonic}. Denote by $\LetaS$ the symmetric part of $\Leta$ on $\Lpieta$, defined as
\[
\forall \varphi, \psi \in \mathcal{S}, \qquad \left\langle \LetaS \varphi, \psi \right\rangle_\eta = \frac12 \Big( \langle \calL_\eta \varphi, \psi \rangle_\eta + \langle \varphi, \calL_\eta \psi \rangle_\eta \Big).
\]
Note that the action of this operator is not explicit when $\pieta$ is not known.
\begin{lemma}
\label{lemma:variance LU}
For any $\varphi, U \in \calS$,
\begin{equation*}
  \sigma_{\varphi + \Leta U,\eta}^2 = \sigma_{\varphi,\eta}^2 + \lang -\LetaS U, 2 \Letainv \Pieta \varphi + \Pieta U \rangeta.
\end{equation*}
\end{lemma}
\begin{proof}
By definition of the asymptotic variance,
\begin{equation*}
\begin{aligned}
  \sigma_{\varphi + \Leta U,\eta}^2 &= \lang \varphi + \Leta U, -\Leta^{-1} \Pieta \left( \varphi + \Leta U \right) \rangeta \\
  &= \sigma_{\varphi,\eta}^2 - \lang \varphi, \Pieta U \rangeta - \lang \Leta U, \Letainv \Pieta \varphi \rangeta - \lang \Leta U, \Pieta U \rangeta \\
  &= \sigma_{\varphi,\eta}^2 - \lang \Pieta U, \Leta \Letainv \Pieta \varphi \rangeta + \lang \Leta U, -\Letainv \Pieta \varphi \rangeta + \lang -\LetaS U, \Pieta U \rangeta \\
  &= \sigma_{\varphi,\eta}^2 + \lang -\LetaS U, 2 \Letainv \Pieta \varphi + \Pieta U \rangeta ,
\end{aligned}
\end{equation*}
which is the desired result.
\end{proof}

We now deduce Theorem~\ref{th:var order 2 eta} from Theorem~\ref{th:var order 2 eta exact} using Lemma~\ref{lemma:variance LU}.

\begin{proof}[Proof of Theorem~\ref{th:var order 2 eta}.]
  We use Lemma~\ref{lemma:variance LU} with $U=\varepsilon f$ to compute the asymptotic variance of
  \begin{equation*}
    \phi_{\eta, \varepsilon} = \phi_\eta + \varepsilon \Leta f,
  \end{equation*}
  with $\phi_\eta$ given by~\eqref{eq:modified obs}. Noting that (from~\eqref{eq:QK} with $K=1$)
  \begin{equation*}
    \Letainv \Pieta \phi_\eta = \eta^2 \Letainv \Pieta A^2 R + \eta \Pieta \Lnotinv \Pinot A R,
  \end{equation*}
  it comes
  \begin{equation}
    \label{eq:calcul th4}
    \begin{aligned}
      \sigma_{\phi_{\eta, \varepsilon},\eta}^2&= \sigma_{\phi_\eta,\eta}^2 + \varepsilon \lang -\LetaS f, 2 \Letainv \Pieta \phi_\eta + \varepsilon \Pieta f \rangeta \\
      &= \sigma_{\phi_\eta,\eta}^2 + 2 \varepsilon \eta \lang -\LetaS f, \Pieta \Lnotinv \Pinot A R \rangeta + \varepsilon^2 \lang -\LetaS f, \Pieta f \rangeta \\
& \quad     +2 \varepsilon \eta^2 \lang -\LetaS f, \Letainv \Pieta A^2 R \rangeta\\
      &= \sigma_{\phi_\eta,\eta}^2 + \varepsilon \eta \lang \Leta f, -\Lnotinv \Pinot A R \rangeta - \varepsilon \eta \lang f, \Leta \Lnotinv \Pinot A R \rangeta\\
      &\quad + \varepsilon^2 \lang -\Leta f, f \rangeta 
      +\varepsilon \eta^2 \lang \Leta f, -\Letainv \Pieta A^2 R \rangeta - \varepsilon \eta^2 \lang f, \Pieta A^2 R \rangeta.
    \end{aligned}
\end{equation}
In order to retain only the leading order terms in the expansion in $\eta$ and $\varepsilon$ we first bound the two last terms in the last equation of~\eqref{eq:calcul th4} in a fashion similar to~\eqref{eq:bounds app}. Then we change the scalar products in $\Lpieta$ by their equivalents in $\Lpinot$ and replace $\Leta$ by $\Lnot$ (controlling the error with Corollary~\ref{coro:scalar product}). All higher order terms are gathered in the remainder, using the inequalities $|\varepsilon| \eta^2 \leqslant |\varepsilon|^3 + |\eta|^3$ and $\varepsilon^2 |\eta| \leqslant |\varepsilon|^3 + |\eta|^3$. Finally, there exist $\varepsilon_*>0$ and $C_{R, \eta_*, \varepsilon_*, f} \in \mathbb{R}_+$ such that, for any $|\eta| \leqslant \eta_*$ and any $|\varepsilon| \leqslant \varepsilon_*$,
\begin{equation*}
  \sigma_{\phi_{\eta, \varepsilon},\eta}^2 = \sigma_{\phi_\eta,\eta}^2 + \varepsilon \eta \lang (\Lnot+\Lnot^*) f, -\Lnotinv \Pinot A R \rangnot + \varepsilon^2 \lang -\Lnot f, \Pinot f \rangnot + (\varepsilon^3 + \eta^3) E_{R, \eta, \varepsilon, f},
\end{equation*}
where $|E_{R, \eta, \varepsilon, f}| \leqslant C_{R, \eta_*, \varepsilon_*, f}$. Formula~\eqref{eq:variance with cv} then follows in view of Theorem~\ref{th:var order 2 eta exact}.
\end{proof}

\section{Technical results used in Section~\ref{s:quasi harmonic}}

\subsection{Equivalence of modified flux observables}
\label{ap:justification flux}

There exist infinitely many observables whose average is the average heat flux in the chain. In particular (see~\eqref{eq:equality fluxes}) any linear combination of the elementary fluxes with weights summing to $1$ (\textit{i.e.} of the form~\eqref{eq:linear combination flux}) has the same average. The procedure described in Section~\ref{s:strategy} allows to construct a modified observable~$\phi$ starting from any observable $R$. A legitimate question is which choice of $R$ provides the modified observable with the smallest asymptotic variance. We show here that, starting from any linear combination of the form~\eqref{eq:linear combination flux}, the resulting modified observable has the same asymptotic variance in the equilibrium setting. Note that the linear combination can involve the fluxes at the ends of the chain $j_0$ and $j_N$.

Consider two fluxes $R^1$ and $R^2 = R^1 + \calL U$, where $R^1$, $R^2$ are linear combinations of the elementary fluxes $(j_n)_{0 \leqslant n \leqslant N}$, while $U$ is a linear combinations of the elementary energies $(\varepsilon_n)_{1 \leqslant n \leqslant N}$. The function~$U$ can indeed be assumed to be of this form since $j_{n+1} = j_n - \calL \varepsilon_n$ in view of~\eqref{eq:energy balance}. The functions $R^1, R^2$ and~$U$ have their counterparts in the simplified (harmonic) setting: $R_0^2 = R_0^1 + \Lnot U_0$. The two associated simplified Poisson equations read
\begin{equation*}
\left\{ \begin{aligned}
  -\Lnot \Phi_0^1 &= R_0^1 - \Esp_0[R_0^1], \\
  -\Lnot \Phi_0^2 &= R_0^2 - \Esp_0[R_0^2].
\end{aligned} \right.
\end{equation*}
The right hand side of these two equations is modified as well since the definition of the fluxes $j_n$ depends on the potential $v$. The average $\Esp_0[R_0^1] = \Esp_0[R_0^2]$ is the heat flux for the harmonic chain. The solutions of these Poisson equations satisfy $\Phi_0^2 = \Phi_0^1 - U_0$ (up to elements of the kernel of $\Lnot$, which are constants~\cite{Carmona07}), so the two corresponding modified observables are such that
\begin{equation*}
  \phi_2 = R^2 + \calL \Phi_0^2 = R^1 + \calL U + \calL (\Phi_0^1 - U_0) = \phi_1 + \calL (U - U_0).
\end{equation*}
Assume now that the chain is at equilibrium ($\TL=\TR$). In view of Proposition~\ref{prop:variance invariance}, the two modified observables thus have the same asymptotic variance (\emph{i.e.} $\sigma^2_{\phi_1} = \sigma^2_{\phi_2}$) as soon as $U-U_0$ does not depend on $p_1$ nor on $p_N$. This is indeed the case for the elementary energies $\varepsilon_n - \varepsilon_{n,0} = \frac 1 2 (w(r_{n-1})+w(r_n))$ for $0 \leqslant n \leqslant N$, where $\varepsilon_{n,0}$ is defined by~\eqref{eq:energy balance} with $v$ replaced by $v_0$. This is in particular true at the ends of the chain ($n=0$ and $n=N$), so the boundary flux $R$ defined in \eqref{eq:heat flux sum} and the standard (bulk) flux $\widetilde R$ provide two modified observables with the same asymptotic variance.

When the temperature difference $\TL-\TR$ is not too large, the asymptotic variances of the two modified observables are approximately equal. This shows that the choice of the linear combination of the form~\eqref{eq:linear combination flux}, from which the modified observable $\phi$ is constructed, does not significantly change the asymptotic variance in this regime.

\subsection{Computation of the asymptotic variances of \texorpdfstring{$j_0$}{j0} and \texorpdfstring{$j_N$}{jN}}
\label{ap:proof variance end}

Since we are in the setting of Remark~\ref{rmk:flux eq}, we assume in this section that the system is at equilibrium ($\TL = \TR = \betainv$). Recall that $\calL \varepsilon_1 = j_0 - j_1$, and more precisely $\LFD \varepsilon_1 = j_0$ where $\LFD$ is the symmetric part of the generator at equilibrium, which is known explicitly (see~\eqref{eq:def_LFD}). Therefore, using Lemma~\ref{lemma:variance LU} with $\varphi = j_0$ and $U = - \varepsilon_1$ (so that $\varphi + \calL U = j_1$),
\begin{equation*}
\begin{aligned}
  \sigma_{j_1}^2 &= \sigma_{j_0}^2 + \lang \LFD \varepsilon_1, 2 \calLinv j_0 - \varepsilon_1 \rangeq \\
  &= \sigma_{j_0}^2 + 2 \lang j_0, \calLinv j_0 \rangeq + \lang \gamma  \betainv \papone^* \papone \varepsilon_1, \varepsilon_1 \rangeq\\
  &= -\sigma_{j_0}^2 + \gamma \betainv \left\| \papone \left( \frac {p_1^2} {2m} \right) \right\|_{\rm eq}^2 \\
  &= -\sigma_{j_0}^2 + \frac \gamma m \betainvinv.
\end{aligned}
\end{equation*}
Therefore,
\begin{equation*}
  \sigma_{j_0}^2 = \frac {\gamma} m \betainvinv - \sigma_{j_1}^2,
\end{equation*}
from which~\eqref{eq:variance boundary} follows in view of~\eqref{eq:variance and conductivity}. Similar computations give the result for $j_N$.

\subsection{Euler-Lagrange equation for~\texorpdfstring{\eqref{eq:rhar omegahar values}}{[eq]}}
\label{ap:proof minimizer fpu}

Denoting by $\hat \Omega = m \omegahar^2$, the minimization problem~\eqref{eq:minimization} can be recast as minimizing the following function for $(\rhar, \hat \Omega) \in \bbR \times (0,+\infty)$:
\begin{equation*}
\begin{aligned}
  f(\rhar, \hat \Omega) &= \int_\bbR \left[ v'(r_1) - \hat \Omega (r_1-\rhar) \right]^2\, \rme^{- \beta v(r_1)} \, \dd r_1 \\
  &= \int_\bbR \left[ v'(r_1)^2 - 2 \hat \Omega v'(r_1) (r_1-\rhar)	+ \hat \Omega^2 (r_1-\rhar)^2 \right] \, \rme^{- \beta v(r_1)} \, \dd r_1 \\
  &= \int_\bbR \left[ v'(r_1)^2 - 2 \betainv \hat \Omega	+ \hat \Omega^2 (r_1-\rhar)^2 \right] \, \rme^{- \beta v(r_1)} \, \dd r_1 \\
  &= C - 2 \betainv \hat \Omega \calM_0 + \hat \Omega^2 (\calM_2 - 2 \calM_1 \rhar + \calM_0 \rhar^2),
\end{aligned}
\end{equation*}
with $C=\int_\bbR v'(r_1)^2 \rme^{- \beta v(r_1)} \, \dd r_1$ and where the third line is obtained with an integration by parts. The gradient of $f$ vanishes if and only if:
\begin{equation*}
\left\{
\begin{aligned}
  0 &= \hat \Omega^2 (-2 \calM_1 + 2 \calM_0 \rhar), \\
  0 &= - 2 \betainv \calM_0 + 2 \hat \Omega (\calM_2 - 2 \calM_1 \rhar + \calM_0 \rhar^2).
\end{aligned}
\right.
\end{equation*}
The only solution of this system is indeed given by~\eqref{eq:rhar omegahar values}.

\subsection{Harmonic chain}
\label{ap:harmonic}

We establish in this section the formulas~\eqref{eq:cv chain} using the linear structure of the harmonic chain, see~\cite[Appendix B]{Lepri03} for similar computations. The interaction potential writes $v_0(r) = \frac 1 2 m \omega^2 (r-\rhar)^2$, so~\eqref{eq:dynamics chain} reduces to
\begin{equation}
\label{eq:harmonic process}
\left\{
\begin{aligned}
	\dd r_n &= \frac 1 m (p_{n+1} - p_n) \, \dd t, \\
	\dd p_1 &= m \omega^2 (r_1-\rhar) \, \dd t - \frac \gamma m p_1 \, \dd t + \sqrt{2 \gamma \TL} \dd W_t^L, \\
	\dd p_n &= m \omega^2 (r_n - r_{n-1}) \, \dd t, \\
	\dd p_N &= -m \omega^2 (r_{N-1}-\rhar) \, \dd t - \frac \gamma m p_N \, \dd t + \sqrt{2 \gamma \TR} \dd W_t^R.
\end{aligned}
\right.
\end{equation}
In order to simplify the algebra we make the change of variables 
\begin{equation*}
  x = (p_1, m \omega (r_1-\rhar), p_2, \cdots, p_{N-1}, m \omega (r_{N-1}-\rhar), p_N) \in \bbR^{2N-1},
\end{equation*}
and denote by $\nu = \frac {m \omega} \gamma > 0$ the dimensionless ratio between the respective time scales of the harmonic potential and of the fluctuation-dissipation process. The process~\eqref{eq:harmonic process} is in fact a generalized Ornstein-Uhlenbeck process:
\begin{equation}
  \label{eq:dynamics harmonic}
  \dd x = \frac \gamma m \bfA x \, \dd t + \sqrt {2 \gamma \betainv} \left( \bfS + \frac 1 2 \beta (\TL-\TR) \ \bfR \right)^\half \, \dd W_t,
\end{equation}
where $\betainv = (\TL+\TR)/2$ and 
\begin{equation*}
  \bfA = \nu \left( \bfJ - \bfJ^\top \right) - \bfS \in \bbR^{2N-1 \times 2N-1},
\end{equation*}
with
\begin{equation*}
\begin{aligned}
	\bfJ &= \begin{pmatrix}
	0 & 1 & & (0) \\
	& \ddots & \ddots & \\
	& & \ddots & 1 \\
	(0)& & & 0
	\end{pmatrix} , \ 
	\bfS &= \begin{pmatrix}
	1 & & (0) \\
	& (0) & \\
	(0) && 1
	\end{pmatrix} , \
	\bfR &= \begin{pmatrix}
	1 & & (0) \\
	& (0) & \\
	(0) && -1
	\end{pmatrix}.
\end{aligned}
\end{equation*}
The generator of this process writes, for any smooth function $\varphi$:
\begin{equation*}
	\Lnot \varphi(x) = \frac \gamma m x^\top \bfA^\top \nabla \varphi(x) + \gamma \betainv \left(\bfS + \frac 1 2 \beta (\TL-\TR) \ \bfR \right) : \nabla^2 \varphi(x).
\end{equation*}
Recall that the observable we consider is the heat flux $R = \frac 1 2 (j_0 + j_N)$ at the ends of the chain, with $j_0$ and $j_N$ given by~\eqref{eq:heat flux boundaries}. This corresponds to the following quadratic form:
\begin{equation*}
	R(x) = -\frac \gamma {2 m^2} x^\top \bfR x + \frac {\gamma (\TL-\TR)} {2m}.
\end{equation*}
We look for the solution $\Phinot$ to the Poisson equation
\begin{equation}
\label{eq:harmonic poisson}
	-\Lnot \Phinot = R - \Esp_0[R].
\end{equation}
The observable $R$ is the sum of a quadratic part and a constant. Since $\Lnot$ stabilizes the space of functions $x \mapsto a + x^\top \bfM x$ with $a \in \bbR$ and~$\bfM$ a symmetric matrix, we consider the ansatz
\begin{equation*}
	\Phi_0(x) = \frac 1 {2m} x^\top \bfK x + C,
\end{equation*}
where $\bfK \in \bbR^{(2N-1) \times (2N-1)}$ is symmetric and $C \in \bbR$ is chosen such that $\Esp_0[\Phi_0] = 0$. The Poisson equation~\eqref{eq:harmonic poisson} then writes: for all $x \in \bbR^{2N-1}$,
\begin{equation*}
	-\frac \gamma {m^2} x^\top \bfA^\top \bfK x - \gamma \betainv (\bfS + \beta (\TL-\TR) \ \bfR) : \frac 1 m \bfK = -\frac \gamma {2m^2} x^\top \bfR x + \frac {\gamma (\TL-\TR)} {2m} - \Esp_0[R],
\end{equation*}
which is equivalent to
\begin{equation}
\label{eq:lyapunov eq}
\left\{
\begin{aligned}
	\bfA^\top \bfK + \bfK \bfA &= \bfR, \\
	\Esp_0[R] &= \frac {\gamma (\TL-\TR)} {2m} + \frac {\gamma \betainv} m \left(\bfS + \beta \frac{\TL-\TR} 2 \ \bfR \right) : \bfK,
\end{aligned}
\right.
\end{equation}
by separating the constant and the quadratic term. The solution is in fact fully explicit since there is an analytical formula for $\bfK$.

\begin{prop}
The solution to~\eqref{eq:lyapunov eq} is the following symmetric matrix
\begin{equation}
  \label{eq:definition K}
  \bfK = -\frac 1 {2(1+\nu^2)} \big[ \nu (\bfJ + \bfJ^\top) + \bfR \big].
\end{equation}
In particular,
\begin{equation*}
  \Esp_0[R] = \frac {\nu^2}{1+\nu^2} \frac {\gamma (\TL-\TR)} {2m}.
\end{equation*}
\end{prop}
\begin{proof}
  Denoting by $\bfM = \bfJ - \bfJ^\top$ and $\bfN=\bfJ + \bfJ^\top$, the following relations hold true
  \begin{equation*}
    \begin{aligned}
      \bfM \bfN - \bfN \bfM = 2 \bfR, \quad \bfM \bfR = - \bfN \bfS, \quad &\bfR \bfM = \bfS \bfN, \quad \bfR \bfS = \bfS \bfR = \bfR,\\
      \bfR:\bfR = 2, \quad \bfR:\bfS = 0, \quad &\bfS:\bfN = 0, \quad \bfR:\bfN = 0.
    \end{aligned}
  \end{equation*}
  This allows to develop $\bfA \bfK + \bfK \bfA^\top$ with $\bfK$ defined in~\eqref{eq:definition K} and obtain $\bfR$. By injecting the expression of $\bfK$ into~\eqref{eq:lyapunov eq} we obtain the expression of $\Esp_0[R]$.
\end{proof}

\begin{remark}
  \label{rmk:Hurwitz}
There exists in fact a unique solution to the Lyapunov equation~\eqref{eq:lyapunov eq} for any right hand side, since $\bfA$ is Hurwitz~\cite{Bartels72}. This latter assertion is equivalent to the exponential decay of the semigroup $\rme^{t\calL}$, proved in~\cite{Carmona07} for example for more general interaction potentials. To prove that $\bfA$ is Hurwitz, take a non-zero eigenvector $x$ associated to an eigenvalue $\lambda \in \bbC$. Suppose that $\scrR(\lambda) \geqslant 0$. Then,
\begin{equation*}
	-|x_1|^2 - |x_{2N-1}|^2 = \bar x^\top \bfA x = \scrR(\lambda) |x|^2 \geqslant 0,
\end{equation*}
so $\scrR(\lambda)=0$ and $x_1 = x_{2N-1} = 0$. Using $\bfA x = \lambda x$ we iteratively obtain $x_2 = 0$, then $x_3 = 0$, and so on until $x=0$. The contradiction proves that any eigenvalue of $\bfA$ has a negative real part.
\end{remark}

The optimal harmonic control variate $\Phinot$ is thus
\begin{equation*}
\begin{aligned}
	\Phinot(x) &= -\frac 1 {2 m (1+\nu^2)} \left[\nu \sum_{k=1}^{2N-2} x_k x_{k+1} + \frac 1 2 x_1^2 - \frac 1 2 x_{2N-1}^2 \right] + C\\
	&= \frac m {2 \gamma ( 1 + \nu^2 )} \left[-\omega^2 \sum_{n=1}^{N-1} (r_n-\rhar) (p_n + p_{n+1}) + \frac \gamma {2m^2} \left(p_N^2 - p_1^2 \right) \right] + C \\
	&= \frac m {2 \gamma ( 1 + \nu^2 )} \sum_{n=0}^N (j_{n,0} - \Esp_0[R]),
\end{aligned}
\end{equation*}
where $j_{n,0}$ is the $n$-th elementary flux~\eqref{eq:heat flux} with $v$ replaced by $v_0$. This function indeed has the dimensions of an energy since it is the product of some characteristic time by a heat flux.

\subsection{Proof of Assumption~\ref{as:generator simple} for the harmonic chain}
\label{ap:ass 4}

The space $\mathcal{S}$ is easily seen to be stable by $\calL$. We prove next that $\calL^{-1}\varphi$ is in~$\mathcal{S}$ when $\varphi \in \mathcal{S}$. Note first that it is possible to analytically integrate the dynamics~\eqref{eq:dynamics harmonic} as
\begin{equation}
\label{eq:evolution harmonic}
x_t = \rme^{\gamma t \bfA / m} x_0 + \sqrt{\frac {2 \gamma} \beta} \int_0^t \rme^{\gamma (t-s) \bfA / m} \left(\bfS + \frac 1 2 \beta (\TL-\TR) \bfR \right)^\half \, \dd W_t.
\end{equation}
The matrix $\bfA$ is Hurwitz (see Remark~\ref{rmk:Hurwitz}) so there exist $\lambda > 0$ and $C_\bfA \geqslant 1$ such that the Frobenius norm of the associated semi-group decays exponentially with rate $\lambda$:
\begin{equation*}
  \left\| \rme^{\gamma t \bfA / m} \right\| \leqslant C_\bfA \rme^{-\lambda t} \leqslant C_\bfA.
\end{equation*}
Take $\varphi \in \calS$ with mean zero with respect to~$\pi$. There exist $\theta_0, \theta_1 \in [0,\theta_*/2)$ such that $\varphi \in \rmL_{\theta_0}^\infty$ and, for any $n\in [1,2N-1]$, $\partial_{x_n} \varphi \in \rmL_{\theta_1}^\infty$. By the results of~\cite{Carmona07} (recalled in Section~\ref{sss:chain properties}) we know already that $\calLinv \varphi \in \rmL_{\theta_0}^\infty$. Denoting by $| \cdot |$ the Euclidean norm in $\bbR^{2N-1}$, and using~\eqref{eq:evolution harmonic},
  \begin{equation}
    \label{eq:bound derivative semigroup}
    \begin{aligned}
      \left| \nabla_{x_0} \left( \rme^{t \calL} \varphi \right) (x_0) \right| &= \left| \nabla_{x_0} \bbE_{x_0}[ \varphi(x_t)] \right| = \left| \Esp_{x_0} [\rme^{\gamma t \bfA / m} \nabla \varphi(x_t)] \right| \\
      &\leqslant \left\| \rme^{\gamma t \bfA / m} \right\| \, \big| \Esp_{x_0}[ \nabla \varphi(x_t) ] \big| \leqslant C_\bfA \rme^{-\lambda t} \| \nabla \varphi \|_{\rmL_{\theta_1}^\infty} \Esp_{x_0} [\calK_{\theta_1}(x_t)].
\end{aligned}
\end{equation}	

By the exponential decay of the semi-group $\rme^{t \calL}$ on the functional space $\rmL_{\theta_1}^\infty$ (see~\cite{Carmona07}), there exist $C_{\theta_1},\lambda'$ such that 
\begin{equation*}
  \left| \rme^{t \calL} \calK_{\theta_1}(x_0) - \Esp[\calK_{\theta_1}] \right| \leqslant C_{\theta_1} \rme^{-\lambda' t} \calK_{\theta_1}(x_0),
\end{equation*}
so that
\begin{equation*}
	\Esp_{x_0} [\calK_{\theta_1}(x_t)] \leqslant \Esp[\calK_{\theta_1}] + C_{\theta_1} \rme^{-\lambda' t} \calK_{\theta_1}(x_0) \leqslant C_{\theta_1}' \calK_{\theta_1}(x_0),
\end{equation*}
with $C_{\theta_1}' = \max\left( \Esp[\calK_{\theta_1}], C_{\theta_1} \right)$. Using this result and integrating~\eqref{eq:bound derivative semigroup} from $t=0$ to~$\infty$,
\begin{equation*}
\begin{aligned}
  \left| \nabla_{x_0} \calLinv \varphi (x_0) \right| &\leqslant \int_0^\infty \left| \nabla_{x_0} \rme^{t \calL} \varphi (x_0) \right| \dd t \leqslant \int_0^\infty C_\bfA \rme^{-\lambda t} \| \nabla \varphi \|_{\rmL_{\theta_1}^\infty} C_{\theta_1}' \calK_{\theta_1}(x_0) \, \dd t,
\end{aligned}
\end{equation*}
so that
\begin{equation*}
	\left\| \nabla_{x_0} \calLinv \varphi \right\|_{\rmL_{\theta_1}^\infty} \leqslant \frac {C_\bfA C_{\theta_1}'} \lambda \| \nabla \varphi \|_{\rmL_{\theta_1}^\infty}.
\end{equation*}
This implies that $\nabla \calLinv \varphi \in \rmL_{\theta_1}^\infty$. Similar formulas hold for higher order derivatives. This allows to show that $\calLinv \varphi \in \calS$, proving that the core $\Pinot \calS$ is stable by $\calLinv$. 

\section{Resolution of the differential equation~\texorpdfstring{\eqref{eq:1D dimer poisson problem}}{[eq]}}
\label{ap:resolution ode}

\newcommand{\bfr}{\mathbf{r}}

The Poisson equation~\eqref{eq:1D dimer poisson problem} can be easily solved using finite differences. In order to provide a stable numerical solution of this equation, let us first determine its boundary conditions. Denoting by $\varphi = \psi'$, \eqref{eq:1D dimer poisson problem} can be reformulated as
\begin{equation}
  \label{eq:1D problem phi}
  \betainv \varphi'(r) = r_* - r + v_*'(r) \varphi(r).
\end{equation}
Note that it is sufficient to determine $\varphi$ in order to evaluate $\calL \Phi_0$.
\begin{prop}
\label{prop:BC_1D}
Assume that $v \in \calC^1((0,+\infty), \bbR)$ is such that
\begin{equation}
  \label{eq:conditions_v}
  \limsup_{r \to 0} v'(r) < +\infty \quad \mbox{and} \quad \frac {v'(r)} r \xrightarrow[r \to +\infty]{} + \infty.
\end{equation}
Then~\eqref{eq:1D problem phi} admits a unique solution $\varphi \in \rmL^2(\pi_*)$ whose primitives are in $\rmL^2(\pi_*)$. Moreover this solution in continuous on $[0,+\infty)$, $\varphi(0) = 0$ and $\varphi$ converges to $0$ at $+\infty$.
\end{prop}
The conditions~\eqref{eq:conditions_v} are satisfied for the double-well potential~\eqref{eq:double well} and for many potentials used in practice. They imply in particular that $v_*(r) \xrightarrow[r \to 0]{} +\infty$ and that $\pi_*$ vanishes at $0$ and $+\infty$.
\begin{proof}
Let us introduce the function
\begin{equation*}
  f(r) = \int_0^r \beta \left( r_* - s \right) \rme^{-\beta v_*(s)} \dd s.
\end{equation*}
We prove that $\varphi(r) = f(r) \rme^{\beta v_*(r)}$ is the only bounded solution to~\eqref{eq:1D problem phi}, and that it vanishes at the boundary of the domain. We first obtain bounds on~$f$ to this end. The function~$f$ satisfies $f(0)=0$ and $f'(r) = \beta \left( r_* - r \right) \rme^{-\beta v_*(r)}$. Using the short-hand notation $z(r) = \int_0^r \rme^{-\beta v_*}$ (with limiting value $z_\infty$ as $r\to +\infty$) and $e(r) = \int_0^r s \rme^{-\beta v_*(s)} \dd s$ (with limiting value $e_\infty$ as $r\to +\infty$), $f$ can be rewritten as
\begin{equation}
  \begin{aligned}
    \label{eq:f(r) rewritten}
    f(r) &= \beta r_* z(r) - \beta e(r) \\
    &= \beta e_\infty \left( \frac {z(r)} {z_\infty} - \frac {e(r)}{e_\infty} \right) = \beta e_\infty \left( \frac {e_\infty-e(r)}{e_\infty} - \frac {z_\infty-z(r)} {z_\infty} \right),
  \end{aligned}
\end{equation}
since $r_* = e_\infty / z_\infty$, which shows that $f(r) \xrightarrow[r \to \infty]{} 0$. Note that $f$ is increasing on $[0,r_*]$, decreasing on $[r_*,+\infty]$ and vanishes at $0$ and infinity. Therefore,~$f \geqslant 0$. Let us now bound the behavior of this function near $0$ and $+\infty$, in order to prove that $\varphi$ vanishes at $0$ and at $+\infty$. In view of~\eqref{eq:conditions_v}, there exist $0<\varepsilon<M<+\infty$ such that $v_*'(r)=v'(r)- \frac {d-1}{\beta r}$ is negative on $(0, \varepsilon]$ and positive on $[M, +\infty)$. Define
\begin{equation*}
  \overline{v_*'}(r) = \sup_{0 < s \leqslant r} v_*'(s), \qquad \underline{v_*'}(r) = r \inf_{s \geqslant r} \frac {v_*'(s)} s.
\end{equation*}
The functions $\overline{v_*'}$ and $\underline{v_*'}$ are increasing on $(0,+\infty)$, $\overline{v_*'}$ converges to $-\infty$ as $r \to 0$ while $\underline{v_*'}$ converges to $+\infty$ as $r \to +\infty$. Moreover, by definition,
\begin{equation*}
  \forall 0 < s \leqslant r \leqslant \varepsilon, \quad 1 \leqslant \frac {v_*'(s)}{\overline{v_*'}(r)}, \qquad \mbox{and} \quad \forall M \leqslant r \leqslant s, \quad 1 \leqslant \frac {v_*'(s)/s}{\underline{v_*'}(r)/r}.
\end{equation*}
Therefore,
\begin{equation}
\begin{aligned}
\label{eq:bounds e z}
	\forall r \leqslant \varepsilon, \quad z(r) &= \int_0^r \rme^{-\beta v_*(s)} \dd s \leqslant \frac 1 {\beta \overline{v_*'}(r)} \int_0^r \beta v_*'(s) \rme^{-\beta v_*(s)} \dd s = -\frac 1 {\beta \overline{v_*'}(r)} \rme^{-\beta v_*(r)}, \\
	\forall r \geqslant M, \quad e_\infty-e(r) &= \int_r^\infty s \rme^{-\beta v_*(s)} \dd s\leqslant \frac r {\beta \underline{v_*'}(r)} \int_r^\infty \beta v_*'(s) \rme^{-\beta v_*(s)} \dd s = \frac{r}{\beta \underline{v_*'}(r)} \rme^{-\beta v_*(r)}.
\end{aligned}
\end{equation}
From~\eqref{eq:f(r) rewritten} and~\eqref{eq:bounds e z} we deduce that the solution $\varphi(r) = f(r) \rme^{\beta v_*(r)}$ of~\eqref{eq:1D problem phi} is non negative on~$\bbR_+^*$ and satisfies
\begin{equation*}
\begin{aligned}
	\forall r \leqslant \varepsilon, \quad 0 \leqslant \varphi(r) \leqslant \beta r_* z(r) \rme^{\beta v_*(r)} \leqslant r_* \ \frac 1 {\left| \overline{v_*'}(r) \right|},\\
	\forall r \geqslant M, \quad 0 \leqslant \varphi(r) \leqslant \beta (e_\infty - e(r)) \rme^{\beta v_*(r)} \leqslant \frac r {\underline{v_*'}(r)}.
\end{aligned}
\end{equation*}
This shows that $\varphi$ vanishes at $0$ and $+\infty$. Moreover, any primitive $\psi$ of $\varphi$ is in $\rmL^2(\pi_*)$ (because $\psi'=\varphi$ is bounded and $\pi_*$ integrates functions which increase linearly). The other solutions of~\eqref{eq:1D problem phi} differ from this one by a factor proportional to $\rme^{\beta v_*(r)}$ (which is the solution of the homogeneous equation associated with~\eqref{eq:1D problem phi}) so that their primitives $\psi$ are not in $\rmL^2(\pi_*)$.
\end{proof}

Proposition~\ref{prop:BC_1D} shows that the solution $\varphi$ of~\eqref{eq:1D problem phi} we are interested in corresponds to the boundary condition $\varphi(0) = 0$. This solution is estimated numerically using a finite difference method. The expectation $r_* = \Esp_*[r]$ is computed with a one-dimensional numerical quadrature. The so-obtained solution is then interpolated by a function $\widehat \varphi$ which is affine on each mesh, so that $\calL \widehat \psi$ can be evaluated exactly at any point. This ensures that the modified observable is not biased since the control variate indeed belongs to the image of $\calL$.

\section{Asymptotic variance estimator}
\label{ap:variance}

\newcommand{\Ndeco}{N_{\rm deco}}
\newcommand{\Nrep}{N_{\rm rep}}

In the three applications we consider, we provide estimators of the asymptotic variances associated with some function $\varphi$ together with error bars on this quantity. We make precise in this section this estimator of the variance and how error bars on these variance estimates are computed. Under Assumptions~\ref{as:pieta} to~5, the stochastic process admits a unique invariant probability measure $\pi$, and the asymptotic variance is well defined for an observable $\varphi = \Pi \varphi + \Esp[\varphi] \in \calS$ (we suppress in this section the subscripts $\eta$ in order to simplify the notation). The empirical mean of $\varphi$ is
\begin{equation*}
	\widehat \varphi_t = \frac 1 t \int_0^t \varphi(x_t) \, \dd t.
\end{equation*}
The associated asymptotic variance~\eqref{eq:def variance} can be computed using the Green--Kubo formula~\cite{Kubo91}
\begin{equation*}
  \begin{aligned}
    \sigma_\varphi^2 & = 2 \int_\mathcal{X} \varphi \left(-\calLinv \Pi \varphi\right) \, \dd \pi
    = 2 \int_0^\infty \Esp_{x_0}\left[ \Pi\varphi(x_s) \Pi \varphi(x_0) \right] \, \dd s \\
    & = 2 \int_0^\infty \left( \Esp_{x_0}\left[ \varphi(x_s) \varphi(x_0) \right] - \Esp[\varphi]^2 \right) \, \dd s,
\end{aligned}
\end{equation*}
where $\Esp$ denotes the expectation with respect to initial conditions $x_0$ distributed according to the invariant probability measure~$\pi$ and for all realizations of the dynamics with generator~$\mathcal{L}$. All these expressions are well defined if we assume a sufficiently fast decay of the associated semi-group (see~\cite[Section 3.1.2]{Lelievre16}). In order to approximate $\sigma_\varphi^2$ we first truncate the time integral as
\begin{equation*}
  \sigma_\varphi^2 \approx 2 \int_0^{t_{\rm deco}} \Esp_{x_0}\left[ \varphi(x_s) \varphi(x_0) \right] \, \dd s - 2 t_{\rm deco} \Esp[\varphi]^2,
\end{equation*}
where the integrand $\Esp\left[ \varphi(x_s) \varphi(x_0) \right]$ is neglected for $s > \tdeco$. The expectations in the integrand are estimated using an empirical average over all the continuous trajectory $(x_t)_{t \in [0,T]}$ (see~\cite{Anderson71}):
\begin{equation}
\label{eq:estimator sokal continu}
\widehat{\sigma_\varphi}^2 = \frac 1 \tsimu \int_0^T \int_{-\tdeco}^{\tdeco} \varphi(x_t) \varphi(x_{t+s}) \, \dd t \dd s - 2 t_{\rm deco} \widehat \varphi_T^2,
\end{equation}
which is a biased estimator of $\sigma_\varphi^2$:
\begin{equation}
  \label{eq:bias estimator}
  \Esp \left[ \widehat{\sigma_\varphi}^2 \right] = 2 \int_0^{t_{\rm deco}} \Esp_{x_0}\left[ \varphi(x_s) \varphi(x_0) \right] \, \dd s - 2 t_{\rm deco} \Esp[\varphi]^2.
\end{equation}
Of course, in practice, the formula for $\widehat{\sigma_\varphi}^2$ is slightly changed in order not to involve $x_t$ for $t<0$ or $t>T$. The double integral is approximated using a Riemann sum or a trapezoidal rule for instance. Consider a discretization $(x^n)_{1 \leqslant n \leqslant N_{\rm iter}}$ of the trajectory $(x_t)_{t \in [0,T]}$ with a timestep~$\Dt$, of length $\tsimu = N_{\rm iter} \Dt$. Introducing $\Ndeco = \tdeco / \Dt$, the discretized version of the estimator~\eqref{eq:estimator sokal continu} is
\begin{equation}
\label{eq:estimator sokal}
	\widehat {\widehat{\sigma_\varphi}}^2 = \frac {\Dt} {N_{\rm iter}} \sum_{i=1}^{N_{\rm iter}} \sum_{j=-\Ndeco}^{\Ndeco} \varphi(x^i) \varphi(x^{i+j}) - 2 t_{\rm deco} \left( \frac 1 {N_{\rm iter}}\sum_{i=0}^{N_{\rm iter}} \varphi(x^n) \right)^2.
\end{equation}
This is the estimator we use throughout this work to provide error bars on average properties. The leading term of the variance of the estimator $\widehat{\widehat{\sigma_\varphi}}^2$ in the regime $\Dt \ll 1 $ and $1 \ll \Ndeco \ll N_{\rm iter}$ is
\begin{equation*}
	\rmVar \left[ \widehat{\widehat {\sigma_\varphi}}^2 \right] \approx \frac {2 (2 \Ndeco+1)} {N_{\rm iter}} \sigma_\varphi^4 \approx \frac {4 t_{\rm deco}} {\tsimu} \sigma_\varphi^4.
\end{equation*}
Here we made the assumption that Isserlis' theorem~\cite{Isserlis18} holds, as if $(x_t)_t$ was a Gaussian process. It is thus straightforward to provide error bars for the estimator $\widehat{\widehat{\sigma_\varphi}}^2$, and even to choose the simulation time $\tsimu$ a priori. Indeed the relative standard statistical error on the variance is very explicit:
\begin{equation*}
	\frac{\sqrt{\rmVar \left[ \widehat{\widehat {\sigma_\varphi}}^2 \right]}}{\sigma_\varphi^2} \approx 2 \sqrt {\frac{\tdeco}{\tsimu}}.
\end{equation*}
For example to estimate the variance with an uncertainty of $1\%$ one should run the simulation for a time $\tsimu = 10^4 \times \tdeco$. There is a trade-off concerning the choice of $t_{\rm deco}$: if it is too small the estimators of the integrals are biased in view of~\eqref{eq:bias estimator}, but if it is too large the variance of the estimator increases. In practice one picks a large value of $t_{\rm deco}$ and uses the cumulated empirical autocorrelation profile to check a posteriori that this value is indeed sufficiently large.

\paragraph{Block averaging.}

Let us relate the previous estimator of the variance to the common variance estimator $\widetilde{\sigma_\varphi}^2$ considered in the method of block averaging (or batch means); see~\cite{Petersen89} as well as the references in~\cite[Section 2.3.1.3]{Lelievre10}. This method consists in cutting the trajectory into several blocks, computing the empirical average of $\varphi$ on each block, and estimating the variance of these random variables (considered as independent and identically distributed). If the size of the blocks is $2 \tdeco$ this estimator has the same variance as $\widehat{\sigma_\varphi}^2$ but the bias is different since
\begin{equation*}
  \Esp \left[ \widetilde{\sigma_\varphi}^2 \right] = 2 \int_0^{2t_{\rm deco}} \left( 1 - \frac s {2 \tdeco} \right) \Esp_{x_0}\left[ \varphi(x_s) \varphi(x_0) \right] \, \dd s - 2 t_{\rm deco} \Esp[\varphi]^2.
\end{equation*}

\paragraph{Implementation.}

It is crucial to compute on-the-fly the first term of the estimator~\eqref{eq:estimator sokal}, without resorting to a double sum which is computationally prohibitive. In practice the sum $S_i = \sum_{j=0}^{\Ndeco} \varphi(x^{i-j})$ is not recomputed from scratch at every time step but updated using $S_{i+1} = S_i + \varphi(x^{i+1}) - \varphi(x^{i-\Ndeco})$. The complexity of this algorithm is thus independent of the choice of $\tdeco$.

\newpage
\bibliographystyle{abbrv}
\bibliography{bibliography}

\end{document}